\documentclass{article}
\usepackage[utf8]{inputenc}
\usepackage{amssymb,amsmath,latexsym,amsthm,amsfonts,bbm,dsfont}
\usepackage{geometry}
\usepackage[colorlinks=true,citecolor=red,linkcolor=blue,pdfpagetransition=Blinds]{hyperref}
\numberwithin{equation}{section}
\usepackage{mathtools, stmaryrd}
\usepackage[english]{babel}
\usepackage{scalerel}
\usepackage{textgreek}
\usepackage{float}
\usepackage{xcolor}
\usepackage{graphicx}
\usepackage{mathtools}
\usepackage{xfrac}

\newtheorem{theorem}{Theorem}
\newtheorem{proposition}{Proposition}
\newtheorem{lemma}{Lemma}
\newtheorem{corollary}{Corollary}
\newtheorem{remark}{Remark}

\newcommand{\taut}{\text{\texttau}}
\begin{document}
\title{Discrete Carleman estimates and application to controllability for a fully-discrete parabolic operator with dynamic boundary conditions  \footnote{The first author is partially supported by FONDECYT Grant 1221892. The second author has been supported by FONDECYT Grant 3200830. The third author is funded by POSTDOC-DICYT Grant 042233ZA-Postdoc, Vicerrector\'ia de Investigaci\'on, Desarrollo e Innovaci\'on. The fourth author is partially supported by Anid PAI Convocatoria Nacional Subvenci\'on a la Instalaci\'on en la Academia Convocatoria a\~no 2019 PAI77190106.}}

\date{\today}
\author{
	Rodrigo Lecaros\footnotemark[2] \and
	Roberto Morales\footnotemark[2] \and
	Ariel P\'erez\footnotemark[3] \and Sebasti\'an Zamorano\footnotemark[3]}
\maketitle 

\footnotetext[1]{Departamento de Matem\'atica, Universidad T\'ecnica Federico Santa Mar\'ia, Casilla 110-V, Valpara\'iso, Chile}
\footnotetext[2]{Departamento de Matem\'atica y Ciencia de la Computaci\'on, Universidad de Santiago de Chile, Las Sophoras 175, Estaci\'on Central,  Santiago, Chile}
\footnotetext{rodrigo.lecaros@usm.cl,\;\;roberto.moralesp@usm.cl,\;\;ariel.perez.c@usach.cl,\;\;sebastian.zamorano@usach.cl.}
\begin{abstract}
    We consider a fully-discrete approximations of 1-D heat equation with dynamic boundary conditions for which we provide a controllability result. The proof of this result is based on a relaxed observability inequality for the corresponding adjoint system. This is done by using a suitable Carleman estimate for such models where the discrete parameters $h$ and $\Delta t$ are connected to the one of the large Carleman parameters.
\end{abstract}

\section{Introduction and main results}

\subsection{Motivation}
Let $T>0$, $\Omega=(0,1)$ and $\omega$ be a nonempty subset of $\Omega$. The main purpose of this work is to study controllability properties for fully-discrete approximations of the following controlled heat equation with dynamic boundary conditions
\begin{align}
\label{eq:original:model}
\begin{cases}
\partial_{t} y-\partial_{x}^{2} y+by=\mathbbm{1}_\omega v,&  (x,t)\in \Omega\times (0,T),\\
\dot{y}_{\Gamma_0} (t)-\partial_{x} y(0,t)+b_{\Gamma_0}(t)y_{\Gamma_0}(t)=0,& t\in (0,T),\\
\dot{y}_{\Gamma_1}(t)+\partial_{x} y(1,t)+b_{\Gamma_1}(t)y_{\Gamma_1}(t)=0,& t\in (0,T),\\
y_{\Gamma_0}(t)=y(0,t),\, y_{\Gamma_1}(t)=y(1,t),& t\in (0,T),\\
y(x,0)=y^0(x),\, y_{\Gamma_0}(0)=y_{\Gamma_0}^0,\, y_{\Gamma_1}(0)=y_{\Gamma_1}^0,&  x\in \Omega.
\end{cases}
\end{align}
Here $(y_{\Gamma_0},y, y_{\Gamma_1})=(y_{\Gamma_0}(t),y(x,t), y_{\Gamma_1}(t))$ is the state, $(y_{\Gamma_0}^0,y^0, y_{\Gamma_1}^0)$ is the initial condition, the control function $v$ is acting on an open subinterval $\omega\subset \Omega$ thorough $\mathbbm{1}_\omega (x)$, the characteristic function of $\omega $ in $\Omega$, and $b=b(x,t)$, $(b_{\Gamma_0},b_{\Gamma_1})=(b_{\Gamma_0}(t),b_{\Gamma_1}(t))$ stand for potential functions in $\Omega$ and at $x=0,$ $x=1$, respectively.  We observe that the initial conditions of \eqref{eq:original:model} may not be related a priori. However, if $y^0$ is regular enough, then $y^0(0)=y_{\Gamma_0}^0$ and $y^0(1)=y_{\Gamma_1}^0$.


Defining the Hilbert space $\mathbb{L}^2:=\mathbb{R}\times L^2(\Omega)\times \mathbb{R}$ and for $k\in \mathbb{N}$,
\begin{align}
    \mathbb{H}^k:=\{(u_{\Gamma_0},u,u_{\Gamma_1})\in \mathbb{L}^2\,;\, u\in H^k(\Omega),\, u(0)=u_{\Gamma_0},\,u(1)=u_{\Gamma_1}\},
\end{align}
in \cite{maniar2017null}, the authors showed that for any $(y_{\Gamma_0}^0,y^0,y_{\Gamma_1}^0)\in \mathbb{L}^2$ and $v\in L^2(\omega\times (0,T))$, system \eqref{eq:original:model} has a unique weak solution $(y_{\Gamma_0},y,y_{\Gamma_1})$ which belongs to $C^0([0,T];\mathbb{H}^1)$. Moreover, there exists a constant $C>0$ such that
\begin{align*}
    \max_{t\in [0,T]} \|(y_{\Gamma_0},y,y_{\Gamma_1})(t)\|_{\mathbb{H}^1}^2 \leq C\left( \|(y_{\Gamma_0}^0,y^0,y_{\Gamma_1}^0)\|_{\mathbb{L}^2}^2 + \|v\|_{L^2(\omega\times (0,T))}^2 \right).
\end{align*}

Recently, parabolic models with dynamic boundary condition such as \eqref{eq:original:model} have attracted the attention of the scientific community due to the wide range of applications. Let us mention some of them: population dynamics, heat transfer between a solid and a moving fluid when the diffusion process is at the interface, climate science, colloid chemistry, diffusion phenomena in thermodynamics, among others (see e.g. \cite{MR2185209, MR2793495, MR2434980,  MR3262633, MR602478, MR2763567}).  We refer the reader interested in a complete derivation of this type of boundary conditions and their physical interpretation (based on the first and second thermodynamical principles) to \cite{goldstein2006}.

The so-called Wentzell boundary conditions were introduced in 1959 when A. D. Wentzell \cite{MR121855} answered (at least partially) the following question: what are the most general boundary conditions which restrict the closure of an Elliptic operator (e.g. $L=d\Delta$) to the infinitesimal operator of a semigroup of positive contraction operators acting on $C(\overline{\Omega})$? More precisely, Wentzell found a large class of boundary conditions which involve differential operators on the boundary which satisfy this requirement. We point out that these boundary conditions are of the same order as the operator acting in $\Omega$ and they have the form
	\begin{align*}
		Lu +db\partial_\nu^L u=0,\text{ on }\partial \Omega,
	\end{align*} 
where $n$ denotes the outward normal at $\partial \Omega$, $b$ a positive constant and $\partial_n ^L u$ is the outward co-normal derivative of $u$ with respect to $L$. In 2002, A. Favini et al. \cite{MR1890879} treated parabolic equations with Wentzell boundary conditions in the $L^p$-context. This paper have been attracted of the scientific community since these boundary conditions govern a dynamic behaviour at the boundary and its solutions possess an analytical nature. We refer to \cite{MR3565950} and references therein for more recent contributions in this direction.

On the other hand, let us introduce the notion of controllability that we want to study in this work. System \eqref{eq:original:model} is said to be null controllable at time $T>0$ if for any given initial datum $(y_{\Gamma_0}^0,y^0,y_{\Gamma_1}^0)\in \mathbb{L}^2$, there exists a control $v\in L^2(\omega\times (0,T))$ such that the corresponding solution $(y_{\Gamma_0},y,y_{\Gamma_1})$ satisfies \begin{align*}
    (y_{\Gamma_{0}},y,y_{\Gamma_{1}})(T)=(0,0,0),\text{ in }\overline{\Omega}.
\end{align*}
We emphasize that the control $v$ in \eqref{eq:original:model} is acting on a subset of $\Omega$. This means that, the first equation in \eqref{eq:original:model} is being controlled directly while the second and third equations are controlled indirectly through the side conditions at $x=0$ and $x=1$. 

Following the duality between controllability and observability,  the null controllability of \eqref{eq:original:model} was obtained in \cite[Theorem 1.1]{maniar2017null} by proving, under a suitable Carleman estimates, an observability inequality for the corresponding adjoint system (see \cite[Proposition 4.1]{maniar2017null}). More precisely, let $(z_{\Gamma_0},z,z_{\Gamma_1})$ be the solution of
\begin{align}
\label{eq:adjoint:original:model:02}
    \begin{cases}
    -\partial_{t} z-\partial_{x}^{2} z+bz=0,& (x,t)\in \Omega\times (0,T),\\
    -\dot{z}_{\Gamma_0}(t)-\partial_{x} z(0,t) +b_{\Gamma_0}(t) z_{\Gamma_0}(t)=0,&t\in (0,T),\\
    -\dot{z}_{\Gamma_1}(t)+\partial_{x} z(1,t)+b_{\Gamma_1}(t) z_{\Gamma_1}(t)=0,&t\in (0,T),\\
    z_{\Gamma_0}(t)=z(0,t),\, z_{\Gamma_1}(t)=z(1,t),& t\in (0,T),\\
    z(x,T)=z^T(x),\, z_{\Gamma_0}(T)=z_{\Gamma_0}^T,\, z_{\Gamma_1}(T)=z_{\Gamma_1}^T,& x\in \Omega. 
    \end{cases}
\end{align}
Then, system \eqref{eq:original:model} is null controllable at any time $T>0$ if and only if there exists $C_{\text{obs}}>0$ such that the following observability inequality holds
\begin{align}
\label{continuous:observability:inequality}
    \|z(0)\|_{L^2(\Omega)}^2 + |z_{\Gamma_0}(0)|^2 + |z_{\Gamma_1}(0)|^2 \leq C_{\text{obs}}\int_0^T\int_\omega |z|^2 dxdt,
\end{align}
for all $(z_{\Gamma_0}^T,z^T,z_{\Gamma_1}^T)\in \mathbb{L}^2$. 

Let us mention another articles related to the controllability of parabolic equations with dynamical boundary conditions. In \cite{khoutaibi2020null} the authors studied the controllability properties of parabolic equations with drift terms using the same approach. The boundary null controllability has been studied in \cite{maniar2022boundary}. The local null controllability for the semilinear problem was addressed in \cite{khoutaibi2022null}. Additionally, in the context of inverse problem we refer to \cite{L:inverse}. In optimal control problem, we found the work \cite{MR3009728} where the authors studied the existence of minimizer by proving that the corresponding elliptic operator generates a strongly continuous semigroup of contractions and applying the concept of maximal parabolic regularity.

As already mentioned, this article is devoted  to study null controllability and observability properties for fully-discrete approximations of systems \eqref{eq:original:model} and its dual problem \eqref{eq:adjoint:original:model:02}, respectively. Namely, we will analyze the controllability of discrete models obtained after discretizing the system \eqref{eq:original:model} by suitable numerical methods and their possible convergence towards the controls of the continuous models under consideration when the mesh-size parameters tend to zero.  The first main advantage of the discrete approximation is that it yields approximate controls that control, at least partially, the approximated numerical dynamics. And the second one is that we can use numerical algorithms with high convergence capability to solve the control problem. It is a well-known fact that controlled discrete systems may differ considerably from its continuous counterpart. That is, the null controllability result contained in \cite{maniar2017null} does not necessarily remain valid in the discrete setting. In this sense, the study and analysis of controllability properties for fully-discrete system with dynamical boundary conditions is a challenge and interesting problem, which deserve to be investigate. As far as we know, this is the first work that analyse this kind of question for parabolic equations with dynamical boundary conditions.

The controllability properties of semi-(in space and/or time) or fully-discrete approximations of parabolic equations have been studied in \cite{boyer:victor:2020, BoyerRousseau2010, BoyerRousseau2014, CLTP-2022, Thuy, discreto:calor}. The current work sings up in this direction and its aim is two-fold. The first one is to provide a discrete Carleman estimate for a  fully-discrete approximation of the parabolic system \eqref{eq:adjoint:original:model:02}. On this concern, in recent years, the development of discrete or semi-discrete Carleman estimates have been used to obtain, for instance, controllability result of spatial discrete parabolic systems \cite{BoyerDeTeresa,BoyerRousseau2010,BoyerRousseau2011,BoyerRousseau2014,CLTP-2022,Thuy}, time discrete parabolic systems \cite{boyer:victor:2020} and the fully discrete case \cite{discreto:calor}. The second aim of the present paper is to illustrate the theoretical results in the study of distributive controllabity of the fully-discrete aproximation of the linear controlled system \eqref{eq:original:model}.  At the best of our knowledge, discrete Carleman estimates, together with the study of the controllability results to this kind of parabolic discrete systems with dynamical boundary conditions are also new.

\subsection{Discrete settings}\label{subsec:discrete setting}

In order to better describe the discrete problem to be studied, we introduce the discretization to be used. Let $M,N\in \mathbb{N}$ and $T>0$. We adopt the notation $\llbracket c,d \rrbracket=[c,d]\cap \mathbb{Z}$. Define $h=1/(M+1)$ and $\Delta t=T/N$. Let us also define the uniform meshes in space and time as follows:
\begin{align*}
    \mathcal{K}=\{x_i:=ih\,;\, i\in \llbracket 0,M+1 \rrbracket\},\quad
    \mathcal{N}=\{t_j:=j\Delta t\, ;\, j\in \llbracket 0,N-1 \rrbracket\}. 
\end{align*}

We point out that here and subsequently we shall use uniform meshes in both variables, i.e., meshes with constant discretization steps in each space and time. However, the introduction of more general non-uniform meshes is possible.


For a function $f:\overline{\Omega}\times (0,T)$, we denote the discrete approximation of $f$ at the point $(x_i,t_j)\in \mathcal{K} \times \mathcal{N}$ as 
\begin{align*}
    f_i^j =f(x_i,t_j),\quad \forall (i,j)\in \llbracket 0,M+1 \rrbracket  \times \llbracket 0,N \rrbracket.
\end{align*}

Then, the fully-discrete approximations of system \eqref{eq:original:model} can read as follows
\begin{align}
    \label{fully-discrete-heat-dbc}
    \begin{cases}
    \dfrac{y^{n+1}-y^{n}}{\Delta t}-\mathcal{A}_hy^{n+1}+\mathcal{C}^{n+1}y^{n+1}=\mathbbm{1}_\omega v^{n+1},&\text{ for }  n\in \llbracket 0,N-1 \rrbracket,\\
    y^{0}_{i}=g_{i},& \text{ for } i\in \llbracket 0,M+1 \rrbracket,
    \end{cases}
\end{align}
where $y^n=\left(\begin{array}{c} y_0^n\\ \vdots\\ y_{M+1}^n  \end{array}\right)$ (with $n\in \llbracket 0,N\rrbracket$) is the state of the system, $\mathcal{A}_h$ is a $(M+2)\times (M+2)$ matrix defined by
\begin{align}
    \mathcal{A}_h=\dfrac{1}{h^2}\left( 
    \begin{array}{cccccc}
    h & -h & 0 & \ldots & \ldots & 0\\
    1 & -2 & 1 & & & \vdots\\
    0 & \ddots & \ddots & \ddots & & \vdots \\
    \vdots & & 1 & -2 & 1 & 0\\
    \vdots & & & 1 & -2 & 1\\
    0 & \ldots & \ldots & 0 & -h & h 
    \end{array}
    \right). 
\end{align}
Moreover, $\mathcal{C}$ and $\mathbbm{1}_\omega$ are two $(M+2)\times (M+2)$ diagonal matrices given by
\begin{align*}
    (\mathcal{C})_{i,i}^n=\begin{cases}
    b_{\Gamma_0}(t_n),& \text{ if } x_i=0,\\
    b(x_i,t_n),& \text{ if }i\in \llbracket 1, M\rrbracket,\\
    b_{\Gamma_1}(t_n),&\text{ if }x_i=L,
    \end{cases}
    \quad (\mathbbm{1}_\omega)_{i,i}=\begin{cases}
    1,&\text{ if }x_i\in \omega,\\
    0,&\text{ if }x_i\notin \omega,
    \end{cases}
\end{align*}
with $b:\Omega\times (0,T) \to \mathbb{R}$ and $b_{\Gamma_0},b_{\Gamma_1}:(0,T)\to \mathbb{R}$ being continuous. The function $g_i$ denotes the approximation of the initial conditions. Finally, $v^n=\left(\begin{array}{c} v_0^n\\ \vdots\\ v_{M+1}^n \end{array} \right)$, $n\in \llbracket 0,N \rrbracket$, denotes a control function acting on the subset $\omega\subset \Omega$.

We point out that the discrete system \eqref{fully-discrete-heat-dbc} is the result of applying different schemes to approximate the derivatives. In particular, we use suitable centered finite difference method for space variable and an implicit Euler scheme for the time variable to the differential operators. More precisely, for the interior points we use the classical centered finite difference for the operator $\partial_{x}^{2} $ and we use 
\begin{align*}
\partial_{x} u(0)\approx \dfrac{1}{h}(u_1-u_0)+\mathcal{O}(h), \quad \partial_{x} u(1)\approx \dfrac{1}{h}( u_{M+1}-u_{M})+\mathcal{O}(h),
\end{align*}
to approximate the spatial derivatives of $y$ at the points $x=0$ and $x=1$, respectively. 

We can introduce a notion of controllability for the fully-discrete scheme. Specifically, we say that system \eqref{fully-discrete-heat-dbc} is null controllable at time $T>0$ if for any initial datum $g\in \mathbb{R}^{M+2}$, there exists a control $v=v_i^j$, $i\in \{0,\ldots,M+1\}$ and $j\in \{1,\ldots,N\}$ such that the corresponding solution fulfills 
\begin{align*}
    y^{N}_{i}=0,\quad \forall i\in\{0,\ldots,M+1\}.
\end{align*}

Following the previous setting and the duality between control and observability, we write the adjoint of the discrete system \eqref{fully-discrete-heat-dbc}, which is given by
\begin{align}
    \label{discrete:adjoint:problem}
    \begin{cases}
    -\dfrac{z^{n+1}-z^{n}}{\Delta t}-\mathcal{A}_h^{\ast} z^{n}+\mathcal{C}^{n} z^{n}=0,& \text{ for } n\in \llbracket 1,N\rrbracket,\\
    z_i^N=w_i,&\text{ for } i\in \llbracket 0,M+1 \rrbracket,
    \end{cases}
\end{align}
where $w_i$ denotes the approximation of the initial datum.

To study the controllability of the fully-discrete system \eqref{fully-discrete-heat-dbc} and, equivalently the observability inequality for the adjoint system \eqref{discrete:adjoint:problem},  it is necessary to rewrite our discrete control problem to make an analysis as close as possible to the continuous environment. It is worth mentioning that this will only help us as a guide, since in the discrete setting the appearance of new terms and changes in the concepts of control and observation are natural, as we will see in what follows.

\subsubsection{Spatial discretization}
\label{sec:spatial:discretization}
Now, we devote to give the notation for the spatial discretization case. For $x\in\mathcal{K}$, we define the translation operators $\taut_{\pm}(x)=x\pm h/2$. For $\mathcal{W}\subseteq\mathcal{K}$, with the translation operators, we can build two dual set given by
$$\mathcal{W}'=\taut_{-}(\mathcal{W})\cap \taut_{+}(\mathcal{W}),\quad \mathcal{W}^*=\taut_-(\mathcal{W})\cup \taut_+(\mathcal{W}).$$
For instance, if $\mathcal{W}$ has $N\in\mathbb{N}$ points then $\mathcal{W}'$ and $\mathcal{W}^{\ast}$ will have $N-1$ and $N+1$ points respectively. This situation can be seen more clearly in Figure \ref{spatial:sets}:
\begin{figure}[H]
    \centering
    \includegraphics[scale=1]{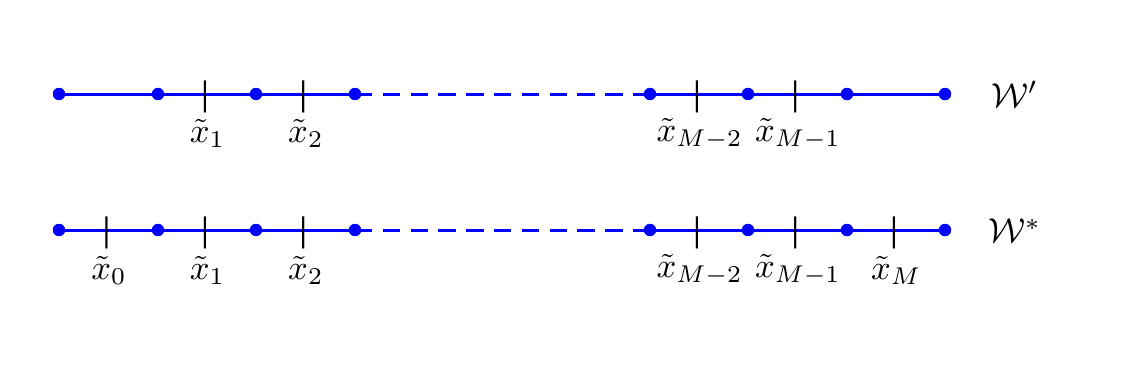} 
    \caption{The sets $ \mathcal{W}'$ and $\mathcal{W}^*$.}
    \label{spatial:sets}
\end{figure}
The dual of a dual set will be denoted as  $\overline{\mathcal{W}}:=\left(\mathcal{W}^{*}\right)^{\ast}$ and $\mathring{\mathcal{W}}:=\left(\mathcal{W}'\right)'$. We note that if  $\mathring{\overline{\mathcal{W}}}=\mathcal{W}$, it follows that for two consecutive points $x_{i},x_{i+1}\in \mathcal{W}$ we have $x_{i+1}-x_{i}=h$. Thus, any subset  
$\mathcal{W}\subset \mathcal{K}$ that verifies $\mathring{\overline{\mathcal{W}}}=\mathcal{W}$ will be called
regular mesh. We denote by $C(\mathcal{W})$ the set of function defined in $\mathcal{W}\subset{\mathcal{K}}$. Then, for $u\in C(\mathcal{W})$ the translation operators is given by
$\taut_{\pm} u(x)=u(x\pm h/2)$. We note that $\taut_{\pm}:C(\mathcal{W})\longrightarrow C(\mathcal{W}^{\ast})$. Moreover, for $u\in C(\mathcal{W}^{\ast})$ we have that $\taut_{\pm}u\in C(\overline{\mathcal{W}})$. 
We define the average and the difference operator for $u\in C(\mathcal{W})$ as
\begin{align*}
    A_{h} u=\dfrac{\taut_+ u+\taut_- u}{2},\quad
    D_{h} u=\dfrac{\taut_+ u-\taut_- u}{h},
\end{align*}
respectively.
For $\mathcal{W}\subseteq \mathcal{K}$ being a uniform mesh, we define the integral of $u$ on $\mathcal{W}$ as follows:
\begin{align*}
    \int_{\mathcal{W}} u=h\sum_{x\in \mathcal{W}} u(x).
\end{align*}

For functions sampled on $\mathcal{W}$, we define the $L^2$-inner product as follows
\begin{align*}
    \langle u,v \rangle_{L_{h}^2(\mathcal{W})}=\int_{\mathcal{W}} u\,v=h\sum_{x\in \mathcal{W}} u(x)v(x).
\end{align*}
The associated norm will be denoted by $\|u\|_{L_{h}^2(\mathcal{W})}$. Analogously, we define the $L_{h}^\infty(\mathcal{W})$-norm as
$$\|u\|_{L_{h}^\infty(\mathcal{W})}=\sup_{x\in \mathcal{W}} |u(x)|.$$
In addition, we define the $\mathbb{L}_{h}^{2}$-norm for $u\in C(\overline{W})$ by $\left\|u \right\|_{\mathbb{L}_{h}^{2}(\mathcal{W})}^2:=\left\| u\right\|_{L^{2}_{h}(\mathcal{W})}^2+\left\|u \right\|_{L^{2}_{h}(\partial\mathcal{W})}^2$, where
$\displaystyle \left\|u \right\|^{2}_{L_{h}^{2}(\partial\mathcal{W})}:=\sum_{x\in\partial\mathcal{W}}u(x)$.

To introduce the boundary conditions, we define the outward normal for $x\in \partial \mathcal{M}$  as
\begin{equation*}
 n(x):=\begin{cases}\ \ 
    1\ & \tau_{-}(x)\in \mathcal{M}^{\ast} \mbox{ and } \tau_{+}(x)\notin \mathcal{M}^{\ast},\\
    -1& \tau_{-}(x)\notin  \mathcal{M}^{\ast} \mbox{ and } \tau_{+}(x)\in \mathcal{M}^{\ast},\\
     \ \ 0 &\ \mbox{otherwise} .\end{cases}
\end{equation*}
We indicate by $\partial \mathcal{M}^{+}$ (resp. $\partial \mathcal{M}^{-}$) the set of points such that $n(x)=1$ (resp. $n(x)=-1$).  For every  $x \in \partial \mathcal{M},$ we also introduce the trace operator for $u\in C(\mathcal{M}^{\ast})$ as
\begin{equation*}
\ t_{r}(u)(x):=\begin{cases}
    \tau_{-}u(x)\ & n(x)=1,\\
    \tau_{+}u(x)& n(x)=-1,\\
     \ \ 0 & n(x)=0 .
     \end{cases}
\end{equation*}




\subsubsection{Time discretization}
\label{sec:time:discretization}
Now, we devote to introduce some notations in order to define the discretization of the time variable. To do this, we recall that $\Delta t= T/N$, with $N\in \mathbb{N}$. We define
\begin{align*}
    \mathcal{N}=&\{j\Delta t\,;\, j\in \llbracket 1, N\rrbracket\}, \quad \overline{\mathcal{N}}=\mathcal{N}\cup \{0\},\\
    \mathcal{N}^{\ast}=&\{(j-1/2)\Delta t\,;\, j\in \llbracket1,N\rrbracket\},\quad \overline{\mathcal{N}}^*=\mathcal{N}^{\ast} \cup \{T+\Delta t/2\}. 
\end{align*}

The schemes of these sets can be seen in Figure \ref{figuresets03}

\begin{figure}[H]
    \centering
    \includegraphics[scale=0.8]{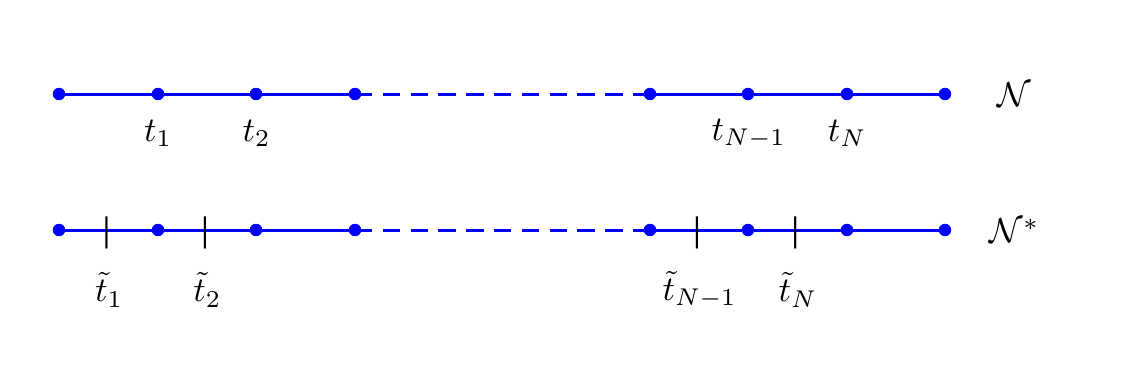}
    \caption{The sets $\mathcal{N}$ and $\mathcal{N}^*$}
    \label{figuresets03}
\end{figure}

We define the time discrete derivative of a function $y:[0,T]\to \mathbb{R}$ sampled on $\mathcal{N}$ as follows
\begin{align*}
    D_{t} y:=\dfrac{\taut^+ y-\taut^- y}{\Delta t}\quad \text{in } \mathcal{N}^{\ast},
\end{align*}
where $\taut^\pm y(t)=y(t\pm \Delta t/2)$, $t\in \mathcal{N} $.\\
Similarly to the spatial discrete variable, we define the outward normal for $t\in \partial \mathcal{N}$, where $\partial\mathcal{N}:=\{0,T\}$,  as
\begin{equation*}
 n(t):=\begin{cases}\ \ 
    1\ & t=T,\\
    -1& t=0,\\
     \ \ 0 &\ \mbox{otherwise} .\end{cases}
\end{equation*}
We indicate by $\partial \mathcal{N}^{+}$ (resp. $\partial \mathcal{N}^{-}$) the set of points such that $n(t)=1$ (resp. $n(t)=-1$).  

\subsection{Main results}
Given the setting previously introduced, we can reformulated the discrete controlled system \eqref{fully-discrete-heat-dbc} in the following way. For a discrete function $y\in C(\overline{\mathcal{M}}\times \overline{\mathcal{N}})$, we consider the following discrete system 
\begin{align}
\label{semi-discrete:model:ops}
    \begin{cases}
    D_{t} y-D_{h}^2 \taut^{+}y+\taut^{+}(b y)=\mathbbm{1}_\omega v,&  (x,t)\in \mathcal{M}\times \mathcal{N}^\ast,\\
    D_{t} y(0,t)-D_{h} \taut^{+}_{+}y(0,t)+b_{\Gamma_0}(t)\taut^{+} y(0,t)=0,&  t\in \mathcal{N}^\ast,\\
    D_{t} y(1,t)+D_{h} \taut^{+}_{-}y(1,t)+b_{\Gamma_1}(t)\taut^{+}y(1,t)=0,& t\in \mathcal{N}^\ast,\\
    y(x,0)=g(x),& x\in \overline{\mathcal{M}}.
    \end{cases}
\end{align}
Here, $g\in C(\overline{\mathcal{M}})$ stands for the initial data of the system and $\mathbbm{1}_{\omega}v\in C(\mathcal{M}\times \mathcal{N}^{\ast})$ is the control function which acts on a discrete subset of $\omega \subset \Omega$. Notice that \eqref{semi-discrete:model:ops} coincide with \eqref{fully-discrete-heat-dbc} using the notation previously introduced. 

Recall that our interest is the null controllability of \eqref{semi-discrete:model:ops}. That is, prove the existence of a control function $v$ such that it drives the initial data $g$ of system \eqref{semi-discrete:model:ops} to zero in time $T>0$.  To perform this analysis, we will reduce our problem to establishing an observability inequality for $q\in C(\overline{\mathcal{M}}\times\overline{\mathcal{N}}^{\ast})$ solution of the following adjoint problem
\begin{equation}\label{discrete primal system}
    \begin{cases}
    -D_t q-D_{h}^2 \taut^- q+b\taut^- (q)=0,& (x,t)\in \mathcal{M}\times \mathcal{N},\\
    -D_t q(0,t)-D_{h} \taut^{-}_{+} q(0,t)+b_{\Gamma_0}(t)\taut^- q(0,t)=0,& t\in \mathcal{N},\\
    -D_t q(1,t)+D_{h} \taut^{-}_{-} q(1,t)+b_{\Gamma_1}(t) \taut^- q(1,t)=0,& t\in \mathcal{N},\\
    q(x,T+\Delta t/2)=q_{T}(x),& x\in \overline{\mathcal{M}}.
    \end{cases}
\end{equation}

To establish this observability inequality, we will first prove a Carleman inequality associated with the solution of \eqref{discrete primal system}. Hence, we introduce the following weight function. Let $\mathcal{B}$ be an open subset of $\Omega$ and let $\mathcal{B}_0\Subset \mathcal{B}$ be a nonempty open interval. Then, it is a well-known that there exists a smooth function $\psi:\overline{\Omega}\to \mathbb{R}$ which satisfies (see e.g. \cite{fursikov1996controllability})
\begin{equation}\label{funcion-peso-1}
\psi>0,\text{ in }\Omega,\quad |\psi '(x)|\geq c,\, \forall x\in \Omega\setminus\overline{\mathcal{B}_0},\quad \psi'(0)>0\text{ and }\psi'(1)<0,
\end{equation}
with $c>0$ being a positive constant.

For $\lambda\geq 1$, we introduce the functions
\begin{align}\label{funcion-peso-2}
    \varphi(x)=e^{\lambda \psi(x)}-e^{\lambda K }<0,\quad \phi(x)=e^{\lambda \psi(x)},\quad x\in \overline{\Omega},
\end{align}
where $K>\|\psi\|_{C(\overline{\Omega})}$, and 
\begin{equation}\label{funcion-peso-3}
    \theta(t)=\frac{1}{(t+\delta T)(T+\delta T-t)},\quad t\in [0,T],
\end{equation}
for some $0<\delta <1/2$. For $\tau\geq 1$, we also define $s(t):=\tau \theta(t)$. For our computations, we will write $r(x,t)=e^{s(t)\varphi(x)}$ and $\rho(x,t)=(r(x,t))^{-1}$.

Finally,  we will use the following notation:
\begin{align*}
\mathcal{P}:=-D_t - D_h^2, \ N_{\Gamma_0}:=D_t+D_h\taut_{+}^{-}, \ N_{\Gamma_1}:=D_t-D_h\taut_{+}^{-}.
\end{align*}

Collecting the previous notation and weighted function, we can state our first main result, which is a fully-discrete Carleman inequality for the discrete adjoint system \eqref{discrete primal system}.
\begin{theorem}\label{theo:discrete:carleman}
Let  $\mathcal{B}$ be an open subset of $\Omega$, $\mathcal{B}_0\Subset \mathcal{B}$ be a nonempty open interval, $\psi$ given by \eqref{funcion-peso-1} and $\varphi$ defined according to \eqref{funcion-peso-2}. For $\lambda\geq 1$ sufficiently large, there exist $C>0$, $\tau_{0}\geq 1$, $h_{0}>0$, $\epsilon_{0}>0$ depending on $\mathcal{B}$, $\mathcal{B}_0$, $T$ and $\lambda$ such that 
\begin{multline}\label{eq:discrete:carleman}
    \int_{\mathcal{M}\times\mathcal{N}}\taut^{-}(s^{-1})|D_ {t}q|^{2}+\int_{\mathcal{M}\times\mathcal{N}^{\ast}}s ^{-1}|D_{h}^{2}q|^{2}+\int_{\mathcal{M}^{\ast}\times\mathcal{N}^{\ast}} s|D_{h}q|^{2}+ \int_{\mathcal{M}\times\mathcal{N}^{\ast}}s^{3}|q|^{2}\\
    +\int_{\mathcal{M}\times\mathcal{N}^{\ast}} s|A_{h}D_{h}q|^{2}+\int_{\partial\mathcal{M}\times\mathcal{N}^{\ast}}st_{r}(|D_{h} q|^2)+\int_{\partial \mathcal{M}\times\mathcal{N}^{\ast}}s^{3}t_{r}(A_{h}(|q|^{2}))+\int_{\partial \mathcal{M}\times \mathcal{N}}|D_{t}q|^{2}\\
    \leq  C_{\lambda_{1}}\left( \int_{\mathcal{M}\times\mathcal{N}}\tau^{-}(r^{2})|\mathcal{P}q|^{2}+\int_{\mathcal{B}\times\mathcal{N}^{\ast}}s^{3}r^{2}|q|^{2}+\int_{\mathcal{N}}(\taut^{-}r)^{2}|N_{\Gamma_{0}}q|^{2}+C_{\lambda}\int_{\mathcal{N}}(\taut^{-}r)^{2}|N_{\Gamma_{1}}q|^{2}\right.\\
    +\left. h^{-2}\int_{\partial\mathcal{M}\times\partial\mathcal{N}}|\taut^{+}rq|^{2}+h^{-2}\int_{\mathcal{M}\times\partial\mathcal{N}}|\taut^{+}rq|^{2} \right),
\end{multline}
for all $\tau\geq \tau_{0}(T+T^{2})$, $0<h<h_{0}$, $\Delta t>0$ and $0<\delta \leq 1/2$ such that
$\displaystyle \frac{\tau h}{\delta T^{2}}\leq \epsilon_{0}$ and  $\displaystyle \frac{\tau^{4}\Delta t }{\delta^{4}T^{6}}\leq \epsilon_{0}$.
\end{theorem}

The key difficulties and novelties of the present Carleman estimate can be summarized as follows:
\begin{enumerate}
\item Although the Carleman inequality \eqref{eq:discrete:carleman} has the structure of the continuous case (see \cite[Lemma 3.2]{maniar2017null}), a large number of extra terms appear when following the classical methodology to establish a Carleman inequality (for instance, \cite{fursikov1996controllability}), due to the discrete nature of the problem. The major problem is that it is not straightforward to absorb these extra terms with those on the left-hand side of the inequality, as is the case in the continuous setting. This means that the analysis must be done more carefully. For that reason and to be clearer, all the details are included in the proof of the Carleman estimate. 

\item Following the previous difficulty, let us notice that the last two terms on the right hand side of \eqref{eq:discrete:carleman}  are new with respect to the continuous case (see \cite[Lemma 3.2]{maniar2017null}), that is, they are characteristic of our discrete environment. Additionally, it is not possible to absorb this extra terms as we will see in the controllability result.

\item The novelty of the discrete Carleman estimate presented in this work is that we do not consider null Dirichlet boundary conditions, then our estimates are new and nonstandard. Namely, as far we know, this is the first discrete Carleman inequality considering dynamic boundary conditions.

\item Another novelty lies in the extension for arbitrary order of the discrete operator applied on the Carleman weight function which have been a useful tool to simplify and to obtain our discrete Carleman estimates. We postpone this discussion to the next section for more details (see Theorem \ref{prop:time:estimate}). From our point of view, this result could be used to answer very interesting problems proposed in \cite{GC-HS-2021}, such as the insensitizing and hierarchic control.
\end{enumerate}

As in the continuous case, this discrete Carleman estimate \eqref{eq:discrete:carleman} implies an observability inequality for system \eqref{discrete primal system}, which is the following second main result of this work. 
\begin{theorem}\label{observ-SZ}
Let $q_{T}\in L_h^2(\mathcal{M})$ and $q\in C(\overline{\mathcal{M}}\times \overline{\mathcal{N}}^*)$ be the corresponding solution of the discrete adjoint system \eqref{discrete primal system}.  Then, there exists a constant $C_{\text{obs}}>0$ such that  the solution $q$ of \eqref{discrete primal system} satisfies
\begin{equation}\label{ine:discrete:obs}
    \left\|\taut^{+}q(0) \right\|^{2}_{\mathbb{L}_{h}^{2}(\mathcal{M})}\leq C_{\text{obs}}\left(\int_{\omega\times\mathcal{N}^{\ast}}|q|^{2}+ e^{-\frac{C_{1}}{h^{\min\{\mu/4,1\}}}}\left\|q_{T} \right\|^{2}_{\mathbb{L}_{h}^{2}(\mathcal{M})}\right)
\end{equation}
with $C_{\text{obs}}:=e^{C(1+\frac{1}{T}+\left\|b \right\|_{\mathbb{L}^{\infty}_{h}(\overline{\mathcal{M}})}^{2/3}+T\left\|b \right\|_{\mathbb{L}^{\infty}_{h}(\overline{\mathcal{M}})})}$ and $\mu\geq 1$.
\end{theorem}

We observe the following facts.
\begin{enumerate}
\item Here again we have an extra term with respect to the continuous case, which is the last one in previous inequality (compare with \cite[Proposition 4.1]{maniar2017null}). This motivates us to call this inequality \emph{relaxed observability estimates}.

\item Additionally, without this second term in the right-hand side, such a discrete observability inequality is known to not hold in general. In fact, this would imply the exact null controllability of the fully-discrete problem \eqref{eq:original:model}, a contradiction with O. Kavian’s counter-example (see e.g. \cite{zuazua2006heat}), in which high-frequency modes can be insensitive to the control.

\item Following the duality between controllability and observability, one expects that under this inequality the null controllability property for \eqref{eq:original:model} can be obtained. However, because it is only possible to obtain the relaxed Carleman estimates \eqref{eq:discrete:carleman}, this weak observability implies what we call a $\phi$-null-controllability result for \eqref{eq:original:model} (see e.g. \cite{boyer2013hum, BoyerRousseau2011} for a deep treatment of this concept). 

\item In our case, the function $\phi$ is given by
\begin{equation}\label{eq:phi}
    \phi(h)=e^{-\frac{C_{1}}{h^{\min\{\mu/4,1\}}}},
\end{equation}
which is similar to the obtained for the discrete heat equation with homogeneous Dirichlet boundary condition. 
\end{enumerate}

Finally, our third and last main result stablish the $\phi$-null-controllability property for system \eqref{semi-discrete:model:ops}.
\begin{theorem}\label{theo:control} 
For $h$ sufficiently small and any initial data $g\in C(\mathcal{\overline{M}})$, there exists a control function $v$ such that $\displaystyle \left\| v\right\|_{\mathbb{L}_h^{2}(\omega\times\mathcal{N}^{\ast})}^{2}\leq  C_{0}\left\|g \right\|_{\mathbb{L}^{2}_{h}(\mathcal{M})}^{2}$, and the solution to \eqref{semi-discrete:model:ops} satisfies $\displaystyle\left\| y(x,T)\right\|_{\mathbb{L}^{2}_{h}(\mathcal{M})}\leq  C_{\text{obs}}\sqrt{\phi(h)}\left\|g \right\|_{\mathbb{L}^{2}_{h}(\mathcal{M})}$.
\end{theorem}

The proof of this Theorem is based on the use of the penalized HUM method (see \cite{boyer2013hum}). This approach allows us to characterize the controllability of the system through the asymptotic behavior of a certain penalized HUM functional, when the penalty parameter tends to zero. In our case, this penalisation parameter is $\phi(h)$ defined by \eqref{eq:phi}, where $h$ is the discrete parameter. We point out that $\phi(h)$ goes to zero as $h$ tends to zero, that is, we recover the continuous null controllability property.   Let us observe that the result contained in Theorem \ref{theo:control} allows us to see how the $\phi$-null-controllability of system \eqref{semi-discrete:model:ops} depends on the initial data $g$. This fact is important in view of applications to numerical analysis, since the initial data of the discrete system are approximations of the initial data of the continuous problem.

The paper is organized as follows. Section \ref{sec:preliminaries} is devoted to the notation and the preliminaries of discrete calculus formulas and discrete estimate to several application of the discrete operator on the Carleman weight function. In Section \ref{proofcarleman}, we prove the Carleman estimate presented in Theorem \ref{theo:discrete:carleman}. For sake of clarity, a large number of proofs of intermediate estimates are provided in Section \ref{sec:proof:carleman}. The relaxed observability inequality and the controllability result  (Theorems \ref{observ-SZ} and \ref{theo:control}, respectively) are proven in Section \ref{sec:obs}.

\section{Preliminaries}\label{sec:preliminaries}

\subsection{Discrete Calculus formulae}
In this section we state the elementary notions concerning discrete calculus formulas. We first present some useful identities for the average and difference operator. Then, we state the discrete integration by parts formulas for the spatial and time discrete variable. At the end, we apply these discrete calculus formulas to  prove that the proposed discrete approximation is stable with respect to the initial condition and the right-hand side.
\begin{lemma}[{\cite[Lemma 3.2 and 3.3]{boyer-2010}}]\label{lem:product:rule}
    Let $u,v$ be two functions in $C(\overline{\mathcal{M}})$. Then, for the difference operator we have
    $$D_{h} (uv)=D_{h} uA_{h} v+ A_{h} u D_{h}v,\quad \text{ in }  \mathcal{M}^{*},$$
    and for the average operator holds
    $$A_{h}(uv)=A_{h} u A_{h} v+\dfrac{h^2}{4} D_{h}uD_{h} v,\quad \text{ in }  \mathcal{M}^{*}.$$

\end{lemma}


Now, the following result represent a space discrete integration by parts formula for the difference and average operators.
\begin{proposition}[{\cite[Proposition 2.4]{LOP-2020}}]\label{pro:integral:space} 
Let $u,v$ be two functions in $C(\overline{\mathcal{M}})$, then the following identities hold for the difference and average operators, respectively
$$\displaystyle \int_{\mathcal{M}} u D_{h} v=-\int_{\mathcal{M}^{\ast}} vD_{h} u +\int_{\partial \mathcal{M}} u\,t_{r}(v)n$$
and
$$\displaystyle \int_{\mathcal{M}} u A_{h} v=\int_{\mathcal{M}^*} vA_{h} u-\frac{h}{2}\int_{\partial \mathcal{M}} u\,t_{r}(v).$$
\end{proposition}


\subsection{Time discrete calculus formulas}
In what follows, we present some identities that will be useful in our main results. For $f$ and $g$ continuously defined over $\mathbb{R}$, the following identity holds
\begin{equation}
    D_{t}(fg)=\taut^{+}fD_{t}g+D_{t}f\taut^{-}g.
    \end{equation}
Indeed, using the definition of $D_{t}$  and adding the term $\frac{1}{\Delta t}\taut^{-}f\taut^{+}g$ we write
\begin{equation}
    D_{t}(fg)=\frac{1}{\Delta t}\left(\taut^{+}f\taut^{+}g-\taut^{-}f \taut^{+}g\right)+\frac{1}{\Delta t}\left(\taut^{-}f \taut^{+}g-\taut^{-}f\taut^{-}g\right).
\end{equation}
Therefore, the above expression can be rewritten as
\begin{equation}\label{eq:time:product:rule}
    D_{t}(fg)=D_{t}f\taut^{+}g+D_{t}g\taut^{-}f.
\end{equation}
Similarly, it follows that 
\begin{equation}
    D_{t}(fg)=D_{t}f\taut^{-}q
+D_{t}g\taut^{+}f.
\end{equation}
Moreover, considering $f=g$ we obtain
\begin{equation}\label{eq:identity:square}
    \begin{split}
       \taut^{+}fD_{t}f&=\frac{1}{2}D_{t}(f^{2})+\frac{1}{2}\Delta t|D_{t}f|^{2}\\
        \taut^{-}fD_{t}f&=\frac{1}{2}D_{t}(f^{2})-\frac{1}{2}\Delta t|D_{t}f|^{2}.
    \end{split}
\end{equation}
The first identity from \eqref{eq:identity:square} follows combining $D_{t}(f^{2})=\taut^{+}fD_{t}f+D_{t}f\taut^{-}f$ and $\Delta t|D_{t}f|^{2}=(\taut^{+}f-\taut^{-}f)D_{t}f$. We analogously proceed to prove the second identity of \eqref{eq:identity:square}.\\
Let us finally comment that we shall use the above identities sampled on the primal and dual time-discrete meshes. Now, we present two integration by parts formulas for the time-discrete operator that we will apply on the next sections.
\begin{proposition}
    \label{properties:time:discret:ops}
    Let us consider $f\in C(\overline{\mathcal{N}})$ and $g\in C(\overline{\mathcal{N}}^{\ast})$. Then, the following identities hold
    \begin{equation}\label{eq:tau:menos} \int_{\mathcal{N}} f\,\taut^-(g)=\int_{\mathcal{N}^{*}}\taut^{+}(f)\, g.
    \end{equation}
    \begin{equation}\label{eq:tau:mas}     \int_{\mathcal{N}}f\,D_{t}g=-\int_{\mathcal{N}^{\ast}}g\,D_{t}f+\int_{\partial \mathcal{N}}f\,\taut^{+}g\,n,
    \end{equation}
    where $\partial \mathcal{N}:=\{0,T\}$.
\end{proposition}
\begin{proof}
The first integration by parts follows directly from the definition of the time-discrete integral. Indeed, we have
\begin{equation*}
    \int_{\mathcal{N}}f\taut^{-}g:=\Delta t\sum_{t\in\mathcal{N}}f(t)g(t-\Delta t/2),
\end{equation*}
and noting that $\taut^{-}\mathcal{N}=\mathcal{N}^{\ast}$ we obtain
\begin{equation}\label{eq:taumenos:integral}
    \int_{\mathcal{N}}f\taut^{-}g=\Delta t\sum_{t\in\mathcal{N}^{\ast}}\taut^{+}f\,g:=\int_{\mathcal{N}^{\ast}}\taut^{+}f\,g,
\end{equation}
which is the desired result. For the second identity we note that
\begin{equation}\label{eq:taumas:integral}
\int_{\mathcal{N}}f\,\taut^{+}g=\int_{\mathcal{N}^{\ast}}\taut^{-}f\,g+\Delta t\int_{\partial \mathcal{N}}    f\,\taut^{+}g\,n,
\end{equation}
where we have used the notation $\partial\mathcal{N}:=\{0,T\}$. Thus, combining \eqref{eq:taumenos:integral} and \eqref{eq:taumas:integral} yields
\begin{equation}\label{eq:time:difference:integral}
    \int_{\mathcal{N}}f\,D_{t}g=-\int_{\mathcal{N}^{\ast}}g\,D_{t}f+\int_{\partial \mathcal{N}}f\,\taut^{+}g\,n,
\end{equation}
which conclude the proof.
\end{proof}
Similar to \cite{GC-HS-2021}, let us point out some useful identities that will allow us to prove our main result. Considering $f,g\in C(\overline{\mathcal{N}}^{\ast})$, from \eqref{eq:taumenos:integral}, we have
\begin{equation*}
    \int_{\mathcal{N}}\taut^{-}f\,D_{t}g=\int_{\mathcal{N}^{\ast}}fD_{t}\taut^{+}g.
\end{equation*}
Then, applying \eqref{eq:time:difference:integral} it follows that
\begin{equation}\label{eq:integral:time:primal}
    \int_{\mathcal{N}}\taut^{-}f\,D_{t}g=-\int_{\mathcal{N}}D_{t}f\,\taut^{+}g+\int_{\partial \mathcal{N}}\taut^{+}(fg)n.
\end{equation}
In the case $f,g\in C(\overline{\mathcal{N}})$, similar steps prove that
\begin{equation}\label{eq:integral:time:dual}
    \int_{\mathcal{N}^{\ast}}\taut^{+}f\,D_{t}g=-\int_{\mathcal{N}^{\ast}}\taut^{-}g\,D_{t}f+\int_{\partial \mathcal{N}}f\,g\,n.
\end{equation}
The last two formulas cannot be considered changing their respective operator $\taut^{\pm}$ since the functions need to be defined in a bigger set.

We end this section proving the stability of the considered discrete scheme by using the previous discrete calculus formulae. For a discrete function $y\in C(\overline{\mathcal{M}}\times \overline{\mathcal{N}})$, we consider the following discrete system 
\begin{equation}\label{discrete:system:well}
    \begin{cases}
    D_{t} y(x,t)-D_{h}^2 \taut^{+}y(x,t)+\taut^{+}(b y)(x,t)=\taut^{+}f(x,t),& \forall (x,t)\in \mathcal{M}\times \mathcal{N}^\ast,\\
    D_{t} y(0,t)-D_{h} \taut^{+}_{+}y(0,t)+\taut^{+}( b_{\Gamma_0}y)(0,t)=0,& \forall t\in \mathcal{N}^\ast,\\
    D_{t} y(1,t)+D_{h} \taut^{+}_{-}y(1,t)+\taut^{+}(b_{\Gamma_1}y)(1,t)=0,&\forall t\in \mathcal{N}^\ast,\\
    y(x,0)=g(x),&\forall x\in \overline{\mathcal{M}}.
    \end{cases}
\end{equation}
Here, $g\in C(\overline{\mathcal{M}})$ and $f\in C(\overline{\mathcal{M}}\times \mathcal{N})$ stand for a given initial data and forcing term for the above system, respectively. Mimicking the continuous methodology, our proof is based on the following discrete Gronwall inequality.

\begin{lemma}(Discrete Gronwall inequality)
Let $\eta\in C(\overline{\mathcal{N}})$ a nonnegative function which satisfies for all $t\in\mathcal{N}^{\ast}$ the difference inequality
\begin{equation}\label{ine:gronwall:hyp}
    D_{t}\eta(t)\leq \gamma\taut^{+}\eta(t)+\taut^{+}g(t),
\end{equation}
where $\gamma>0$ and $g\in C(\overline{\mathcal{N}})$ is nonnegative. Then, for all $t\in\mathcal{N}$ holds
\begin{equation}\label{ine:gronwall}
    \eta(t)\leq e^{\gamma T}\eta(0)+e^{\gamma T}\int_{\mathcal{N}}g,
\end{equation}
provided that $\gamma\Delta t < 1/2$. In particular, if $D_{t}\eta\leq \gamma\taut^{+}\eta$ in $\mathcal{N}^{\ast}$ and $\eta(0)=0$ then $\eta(t)=0$ for all $t\in\mathcal{N}$
\end{lemma}
\begin{proof}
 Let us consider $a:=(1-\gamma \Delta t)^{1/\gamma\Delta t}$. We note that $D_{t}a^{\gamma t}=-\gamma \taut^{-}a^{\gamma t}$. Then, thanks to \eqref{eq:time:product:rule} and \eqref{ine:gronwall:hyp} it follows that
\begin{equation*}
\begin{split}
    D_{t}(\eta(t)a^{\gamma t} )=&D_{t}\eta(t)\taut^{-}a^{\gamma t}-\gamma \taut^{-}a^{\gamma t}\taut^{+}\eta(t)\leq  \taut^{+}g(t)\taut^{-}a^{\gamma t}.
    \end{split}
\end{equation*}
Denoting by $\mathcal{N}_t=(0,t\Delta t/2)\cap \mathcal{N}$ and integrating over $\mathcal{N}^{\ast}_{t}$ we obtain
\begin{equation}
    \begin{split}
    \eta(t)\leq & a^{-\gamma t}\eta(0)+a^{-\gamma t}\int_{\mathcal{N}^{\ast}_{t}}\taut^{+}g\taut^{-}a^{\gamma t},
    \end{split}
\end{equation}
which implies 
\begin{equation}
    \eta(t)\leq e^{\gamma T}\eta(0)+e^{\gamma T}\int_{\mathcal{N}}g,
\end{equation}
where we have used that $\frac{1}{1-\gamma \Delta t }<e^{2\gamma \Delta t}$ provided that $\gamma\Delta t<1/2$.
\end{proof}
To prove the stability of our system, we multiply by $\taut^{+}y$ the main equation of the system \eqref{discrete:system:well} and integrating over $\mathcal{M}$ to obtain
\begin{equation}\label{eq:discrete:energy}
    \int_{\mathcal{M}}D_{t}y\taut^{+}y-\int_{\mathcal{M}}D_{h}^{2}\taut^{+}y\taut^{+}y+\int_{\mathcal{M}}\taut^{+}(b)|\taut^{+}y|^{2}=\int_{\mathcal{M}}\taut^{+}f\taut^{+}y.
\end{equation}
We note that by virtue of \eqref{eq:identity:square} it follows that the first integral of the left hand-side above can be rewritten as
\begin{equation*}
    \int_{\mathcal{M}}D_{t}y\taut^{+}y=\frac{1}{2}\int_{\mathcal{M}}D_{t}(|y|^{2})+\frac{\Delta t}{2}\int_{\mathcal{M}}|D_{t}y|^{2}.
\end{equation*}
Now, applying Proposition \ref{pro:integral:space} on the second integral from the left hand-side of \eqref{eq:discrete:energy} and using the the boundary condition of the system \eqref{discrete:system:well} yield
\begin{equation*}
    \begin{split}
    \int_{\mathcal{M}}D_{h}^{2}\taut^{+}y\taut^{+}y
    =&-\int_{\mathcal{M}^{\ast}}|D_{h}\taut^{+}y|^{2}-\int_{\partial\mathcal{M}}\taut^{+}y\,(D_{t}y+\taut^{+}(by)).
    \end{split}
\end{equation*}
Moreover, proceeding as before for the time discrete variable, we obtain
\begin{equation}\label{eq:D2:well}
    \begin{split}
    \int_{\mathcal{M}}D_{h}^{2}\taut^{+}y\taut^{+}y=&-\int_{\mathcal{M}^{\ast}}|D_{h}\taut^{+}y|^{2}-\int_{\partial\mathcal{M}}\taut^{+}(b|y|^{2})-\frac{1}{2}\int_{\partial\mathcal{M}}D_{t}(|y|^{2})-\frac{\Delta t}{2}\int_{\partial\mathcal{M}}|D_{t}y|^{2}.
    \end{split}
\end{equation}
Thus, combing the above identities on \eqref{eq:discrete:energy}, we have that
\begin{equation}\label{eq:energy:dis}
    \begin{split}
    \int_{\mathcal{M}}\taut^{+}f\taut^{+}y\geq &\frac{1}{2}D_{t}\left(\left\| y(t)\right\|^{2}_{\mathbb{L}_{h}^{2}(\mathcal{M})} \right)+\frac{\Delta t}{2}\left\| D_{t}y(t)\right\|^{2}_{\mathbb{L}_{h}^{2}(\mathcal{M})}\\
    &+\int_{\mathcal{M}^{\ast}}|D_{h}\taut^{+}y|^{2}-\frac{1}{2}\left\|b \right\|_{\mathbb{L}^{\infty}_{h}(\mathcal{M})}\left\| \taut^{+}y(t)\right\|^{2}_{\mathbb{L}_{h}^{2}(\mathcal{M})}.
    \end{split}
\end{equation}
Now from the Cauchy-Schwartz and Young inequalities we observe
\begin{equation*}
\begin{split}
    \int_{\mathcal{M}}\taut^{+}(fy)\leq & \left\|\taut^{+}f \right\|_{L_{h}^{2}(\mathcal{M})}\left\|\taut^{+}y \right\|_{L_{h}^{2}(\mathcal{M})}\leq \frac{1}{2}\left\|\taut^{+}f \right\|^{2}_{L_{h}^{2}(\mathcal{M})}+\frac{1}{2}\left\|\taut^{+}y \right\|^{2}_{L_{h}^{2}(\mathcal{M})}.
    \end{split}
\end{equation*}
We utilize the inequality above in the left hand-side of \eqref{eq:energy:dis} to obtain
\begin{equation*}
    \begin{split}
    \frac{1}{2}\left\|\taut^{+}f \right\|^{2}_{L_{h}^{2}(\mathcal{M})}+\left(\frac{1}{2}+\frac{1}{2}\left\|b \right\|_{\mathbb{L}^{\infty}_{h}(\mathcal{M})}\right)\left\|\taut^{+}y \right\|^{2}_{\mathbb{L}_{h}^{2}(\mathcal{M})}\geq &D_{t}\left(\left\| y\right\|^{2}_{\mathbb{L}_{h}^{2}(\mathcal{M})} \right),
    \end{split}
\end{equation*}
where we have dropped the integral with the terms $D_{t}y$ and $D_{h}\taut^{+}y$. Thus, applying the discrete Gronwall inequality \eqref{ine:gronwall} we get for all $t\in \mathcal{N}$
\begin{equation}\label{ine:discrete:stability}
    \left\| y\right\|^{2}_{\mathbb{L}^{2}_{h}(\mathcal{M})}\leq C\left(\left\| g\right\|^{2}_{\mathbb{L}^{2}_{h}(\mathcal{M})}+\left\|f \right\|^{2}_{L_{h}^{2}(\mathcal{M}\times\mathcal{N})}\right),
\end{equation}
provided that $\max\{\Delta t,\Delta t\left\|b \right\|_{\mathbb{L}^{\infty}_{h}(\mathcal{M})}\}<1/2$ with $C:=\exp({T+T\left\| b\right\|_{\mathbb{L}^{\infty}_{h}(\mathcal{M})}}).$
On the other hand, integrating over $\mathcal{N}^{\ast}$ in inequality 

\subsection{Discrete calculus results on the weight Carleman function}\label{estimaciones}
In this section, we present the result related to discrete (time and space) operations performed on the Carleman weight functions used to establish the Discrete Carleman estimate \eqref{eq:discrete:carleman}. Let us recall that our Carleman weight function is of the form $e^{s\varphi}$, for $s\geq 1$, with $\varphi(x)=e^{\lambda\psi(x)}-e^{\lambda K}$, where $K>\left\| \psi\right\|_{C(\overline{\Omega})}$, $\lambda >0$ and we assume that $\psi$ verify \eqref{funcion-peso-1}. For sake of clarity, we set $r(x,t)=e^{s(t)\varphi(x)}$ and $\rho=r^{-1}$. \\
\textbf{Notation:}  We denote by $\mathcal{O}_{\lambda}(sh)$ the functions that verify $\left\| \mathcal{O}_{\lambda}(sh)\right\|_{L^{\infty}_{h}(Q)}\leq C_{\lambda}sh$ for some constant $C_{\lambda}$ depending on $\lambda$ and $C$ denotes a positive constant uniform with respect to $s,h$ and $\lambda$ which may change from line to line. By $\mathcal{O}(1)$ we denote bounded functions and by $\mathcal{O}_{\lambda}(1)$ a bounded function once $\lambda$ is fixed. 

Loosely put, the results of this section state that the space/time discrete operator on the Carleman weight function has two main terms, the classical continuous derivative and an error term. Thus, we firstly present some estimate for the space continuous derivative on the Carleman weight function. 
\begin{lemma}[{\cite[Lemma 3.7]{boyer-2010}}]\label{3.7}
For $\alpha, \beta\in \mathbb{N}$ we have
\begin{equation}\label{13}
\begin{split}
        \partial^{\beta}_{x}(r\partial_{x}^{\alpha} \rho)
        =&\alpha^{\beta}(-s\phi)^{\alpha}\lambda^{\alpha+\beta}(\partial_{x} \psi)^{\alpha+\beta}+\alpha\beta(s\phi)^{\alpha}\lambda^{\alpha+\beta-1}\mathcal{O}(1)+\alpha(\alpha-1)s^{\alpha-1}\mathcal{O}_{\lambda}(1)\\
        =&s^{\alpha}\mathcal{O}_{\lambda}(1).
        \end{split}
\end{equation}
Moreover, let $\sigma\in [0,1]$ and  $\frac{\tau h}{\delta T^{2}}\leq 1$, then
$\partial_{x}^{\beta}(r(\cdot,t)(\partial_{x}^{\alpha}\rho)(\cdot+\sigma h,t))=\mathcal{O}_{\lambda}(s^{\alpha})$. 
\end{lemma}

\begin{corollary}[{\cite[Corollary 3.8]{boyer-2010}}]\label{3.8}
Let $\alpha$, $\beta$ and $\delta$ be multi-indices. We have
\begin{equation}
    \begin{split}
    \partial^{\delta}_{x}\left( r^{2}(\partial^{\alpha}\rho)\partial^{\beta}\rho\right)
        =&|\alpha+\beta|^{|\delta|}(-s\varphi)^{|\alpha+\beta|}\lambda^{|\alpha+\beta+\delta|}(\nabla\psi)^{\alpha+\beta+\delta}+|\delta||\alpha+\beta|(s\varphi)^{|\alpha+\beta|}\lambda^{|\alpha+\beta+\delta|-1}\mathcal{O}(1)\\
        &+s^{|\alpha+\beta|-1}\left( |\alpha|(|\alpha|-1)+|\beta|(|\beta|-1)\right)\mathcal{O}_{\lambda}(1)\\
        =&\mathcal{O}_{\lambda}(s^{|\alpha+\beta|}).
        \end{split}
\end{equation}
\end{corollary}
Now we have established the estimate for the continuous derivative, we can present the following results to estimate discrete space operations performed on the Carleman weight function. First, we note that the proof presented in \cite{LOP-2020} is given for a time-independent Carleman weight function, yet the estimates hold as mentioned in that work for the time-dependent case replacing the condition  $sh\leq 1$ by $\tau h\max_{[0,T]}\theta(t)\leq 1$, where $s(t):=\tau\theta(t)$. For this reason, we need the condition $\tau h(\delta T^{2})^{-1}\leq 1$ due to $\max_{t\in[0,T]}\theta(t)\leq (\delta T^{2})^{-1}$.

\begin{proposition}[{\cite[Proposition 4.5]{LOP-2020}}]\label{pro:space:estimate}
Let $\alpha,n,m\in\mathbb{N}$. Provided $\tau h(\delta T^{2})^{-1}\leq 1$, we have
\begin{equation}
        rA^{m}_{x}D^{n}_{h}\partial_{x}^{\alpha}\rho= r\partial_{x}^{n}\partial^{\alpha}_{x}\rho+s^{|\alpha|+n}\mathcal{O}_{\lambda}((sh)^{2})=s^{|\alpha|+n}\mathcal{O}_{\lambda}(1).
\end{equation}
\end{proposition}


\begin{lemma}[{\cite[Lemma 4.8]{LOP-2020}}]\label{lem:space:estimate}
 For $\alpha,j,k,m,n\in\mathbb{N}$ and for $\tau h(\delta T^{2})^{-1}\leq 1$, we have
\begin{equation}
        A^{j}_{h}D^{k}_{h}\partial^{\alpha}_{x}\left( rA^{m}_{h}D^{n}_{h}\rho\right)=\partial_{x}^{k}\partial^{\alpha}_{x}(r\partial_{x}^{n}\rho)+s^{n}\mathcal{O}_{\lambda}\left((sh)^{2}\right)=s^{n}\mathcal{O}_{\lambda}(1).
\end{equation}
\end{lemma}


\begin{theorem}[\cite{LOP-2020}, Theorem 4.9]\label{teo:space:estimate} Let $\alpha,\beta,j,k,l,m,n,p\in\mathbb{N}$. Provided $\tau h(\delta T^{2})^{-1}\leq 1$, we have
\begin{equation}
\begin{split}
A^{p}_{h}D^{l}_{h}\partial^{\beta}_{x}(r^{2}A^{j}_{h}D^{k}_{h}(\partial^{\alpha}_{x}\rho)A^{m}_{h}D^{n}_{h}(\rho))=&\partial_{x}^{l}\partial^{\beta}_{x}\left( r^{2}\partial^{k}_{x}\partial^{\alpha}_{x}\rho\partial^{n}_{x}\rho\right)+s^{n+k+|\alpha|}\mathcal{O}_{\lambda}((sh)^{2})\\
=&s^{n+k+|\alpha|}\mathcal{O}_{\lambda}(1).
\end{split}
\end{equation}
\end{theorem}
Let us finally mention the behavior of the time difference operator applied on the Carleman weight function. We borrow from \cite[Proposition 2.14]{BoyerRousseau2014} and \cite[Lemma A.16]{GC-HS-2021} the following estimates
\begin{equation}\label{eq:partial:estimate}
\begin{split}
\partial_{t}(r\partial_{x}^{\alpha}\rho)=&s^{\alpha}T\theta\mathcal{O}_{\lambda}(1),\\
    \partial_{t}^{\beta}(r(x,t)\partial_{x}^{\alpha}\rho(t,x+\sigma h))=&T^{\beta}s^{\alpha}\theta^{\beta}\mathcal{O}_{\lambda}(1),
    \end{split}
\end{equation}
which holds for $\beta=1,2$,  $\alpha\in\mathbb{N}$ and $\tau h(\delta T^{2})^{-1}\leq 1$. This last condition must be corrected in the statement of \cite[Lemma A.16]{GC-HS-2021}.

We note that combining the methodology of the proofs from \cite[Lemma A.16]{GC-HS-2021} and \cite[Proposition 4.5]{LOP-2020} it is possible to extend the results presented in \cite{GC-HS-2021} with respect to the order of the space operator. Indeed, we firstly need the following result from \cite{LOP-2020}.
\begin{corollary}[{\cite[Corollary 4.2]{LOP-2020}}]\label{coro:AD}
Let $f$ be a $(n+4)-$times differentiable function defined on $\mathbb{R}$ and $m,n\in\mathbb{N}$. Then
\begin{equation*}
        \begin{split}
        A^{m}_{h}D^{n}_{h}f&=f^{(n)}+R_{A^{m}_{h}}(f^{(n)})+R_{D^{n}_{h}}(f)+R_{A^{m}_{h}D^{n}_{h}}(f),
        \end{split}
\end{equation*}
where 
{\small
\begin{equation*}
 R_{D^{n}_{h}}(f):=h^{2}\sum_{k=0}^{n}\binom{n}{k}(-1)^{k}\left(\frac{(n-2k)}{2} \right)^{n+2}\int_{0}^{1}\frac{(1-\sigma)^{n+1}}{(n+1)!}f^{(n+2)}(\cdot+\frac{(n-2k)h}{2}\sigma)d\sigma,
 \end{equation*}}
 {\small
 \begin{equation*}R_{A^{n}_{h}}(f):=\frac{h^{2}}{2^{n+2}}\sum_{k=0}^{n}\binom{n}{k}(n-2k)^{2}\int_{0}^{1}(1-\sigma)f^{(2)}(\cdot+\frac{(n-2k)h}{2}\sigma)d\sigma.
 \end{equation*}}
 and
 $R_{A_{h}^{m}D_{h}^{n}}:=R_{A_{h}^{m}}\circ R_{D_{h}^{n}}$.
\end{corollary}
Then, as in the proof of \cite[Proposition 4.5]{LOP-2020} which is similar in spirit to the estimates from \cite{boyer-2010}, applying the above Corollary to $\partial^{\alpha}_{x}\rho$ yields
\begin{equation}
    rA_{h}^{m}D_{h}^{n}(\partial^{\alpha}_{x}\rho)=r\partial_{x}^{n+\alpha}\rho+rR_{A_{h}^{m}}(\partial_{x}^{n+\alpha}\rho)+rR_{D_{h}^{n}}(\partial^{\alpha}_{x}\rho)+rR_{A_{h}^{m}D_{h}^{n}}(\partial^{\alpha}_{x}\rho).
\end{equation}
Now, we use the first order Taylor formula in the time variable to write
\begin{equation}\label{eq:almost:estimate}
\begin{split}
    D_{t}(rA_{h}^{m}D_{h}^{n}(\partial^{\alpha}_{x}\rho))=&\partial_{t}(r\partial_{x}^{n+\alpha}\rho)+\partial_{t}(rR_{A_{h}^{m}}(\partial_{x}^{n+\alpha}\rho))+\partial_{t}(rR_{D_{h}^{n}}(\partial^{\alpha}_{x}\rho))+\partial_{t}(rR_{A_{h}^{m}D_{h}^{n}}(\partial^{\alpha}_{x}\rho))\\
    &+\Delta t\int_{0}^{1}(1-\gamma)\partial_{t}^{2}\left( r(x,t+\gamma\Delta t)A_{h}^{m}D_{h}^{n}(\partial^{\alpha}_{x}\rho(x,t+\gamma\Delta t))\right)d\gamma.
    \end{split}
\end{equation}
We note that  $\partial_{t}(r(x,t)(\partial_{x}^{n+\alpha+2}\rho)(x+(n-2)h\sigma/2,t))=Ts^{\alpha+n+2}\theta(t)\mathcal{O}(1)$ due to \eqref{eq:partial:estimate}. This gives 
\begin{equation}\label{eq:estimate:R}
\begin{split}
\partial_{t}(rR_{A_{h}^{m}}(\partial_{x}^{n+\alpha}\rho))=&T(sh)^{2}s^{n+\alpha}\theta\mathcal{O}_{\lambda}(1),\\ \partial_{t}(rR_{D_{h}^{n}}(\partial^{\alpha}_{x}\rho))=&T(sh)^{2}s^{n+\alpha}\theta\mathcal{O}_{\lambda}(1).
\end{split}
\end{equation}
Similarly, we have 
\begin{equation}\label{eq:estimate:R:01}
\partial_{t}(rR_{A_{h}^{m}D_{h}^{n}}(\partial^{\alpha}_{x}\rho))=T(sh)^{4}s^{n+\alpha}\theta\mathcal{O}_{\lambda}(1).
\end{equation}
For the last term of \eqref{eq:almost:estimate} we proceed as in the proof of \cite[Lemma A.16]{GC-HS-2021}; that is, we use the change of variable $t\mapsto t+\gamma \Delta t$ for $\gamma\in [0,1]$ and then thanks to \eqref{eq:partial:estimate}, \eqref{eq:estimate:R} and \eqref{eq:estimate:R:01} we obtain
\begin{equation}
\begin{split}
    \Delta t\int_{0}^{1}(1-\gamma)\partial_{t}^{2}\left( r(x,t+\gamma\Delta t)A_{h}^{m}D_{h}^{n}(\partial^{\alpha}\rho(x,t+\gamma\Delta t))\right)d\gamma=&\Delta t T^{2}s^{n+\alpha}\theta^{2}\mathcal{O}_{\lambda}(1)\\
    &+\Delta t (sh)^{2}T^{2}s^{n+\alpha}\theta^{2}\mathcal{O}_{\lambda}(1).
    \end{split}
\end{equation}
Therefore, from the preceding discussion we have proven the following important result, which gives us a formula when applying our discrete operator of any order to the weight function that we will use in the Carleman inequality.
\begin{theorem}\label{prop:time:estimate}
Provided $\tau h(\delta T^{2})^{-1}\leq 1$ and $\sigma$ bounded, we have
\begin{equation}\label{eq:prop:time:estimate}
\begin{split}
    D_{t}(rA_{h}^{m}D_{h}^{n}(\partial_{x}^{\alpha}\rho))=&\partial_{t}(r\partial_{x}^{n+\alpha}\rho)+T(sh)^{2}s^{n+\alpha}\theta\mathcal{O}_{\lambda}(1)+\Delta t T^{2}s^{n+\alpha}\theta^{2}\mathcal{O}_{\lambda}(1)\\
    &+\Delta t (sh)^{2}T^{2}s^{n+\alpha}\theta^{2}\mathcal{O}_{\lambda}(1).
\end{split}
\end{equation}
\end{theorem}
We notice that $\frac{\tau\Delta t}{T^{3}\delta^{2}}\leq \frac{1}{2}$ implies $\max_{t\in [0,T+\Delta t]}\theta(t)\leq \frac{2}{\delta T^{2}}$. Thus, this condition in addition to the hypothesis of the above Proposition will be useful to obtain a precise estimate of the error terms from \eqref{eq:prop:time:estimate}, see for instance \cite[Lemma A.16]{GC-HS-2021}.

Finally, for the conjugate on the time variable we need the following estimation and upper bound for the time discrete variable.
\begin{lemma}[{\cite[Lemma B.4]{boyer:victor:2020}}]\label{lem:time:estimate} Under the condition $\Delta t \tau (T^3 \delta^2)^{-1}\leq 1$, we get
$$
  \taut^{-}(r)D_t \rho=-\tau \taut^- (\theta')\varphi +\Delta t \left(\dfrac{\tau}{\delta^3 T^4}+\dfrac{\tau^2}{\delta^4 T^6} \right) \mathcal{O}_\lambda(1).$$
\end{lemma}

\begin{lemma}[{\cite[Lemma B.5]{boyer:victor:2020}}]\label{lem:discrete:theta}There exists a constant $C>0$ independent of $\Delta t, \delta$ and $T$ such that
\begin{itemize}
    \item $|D_t (\theta^\ell)|\leq \ell T\taut^{-}(\theta^{\ell+1}) + C \dfrac{\Delta t}{\delta^{\ell+2} T^{2\ell+2}}$,\quad for each $\ell\in \mathbb{N}$,
    \item $D_t (\theta')\leq C T^2 \taut^-(\theta^3)+C\dfrac{\Delta t}{\delta^4 T^5}$.
\end{itemize}
\end{lemma}

\section{Proof of Theorem \ref{theo:discrete:carleman}}\label{proofcarleman}
The proof of the Carleman estimate follows the classical conjugation for suitable operators involving the heat equation with dynamic boundary conditions. However, contrary to the continuous case, several new terms arise from the dynamic boundary condition and the discretization in space and time. In order to deal with these new difficulties, we repeat some steps from \cite{maniar2017null}, \cite{LOP-2020} and \cite{GC-HS-2021} in a modified form.

In order to simplify the presentation, the proof is done in several steps.

\subsection{Step 1: Conjugation}

By regarding the weight function, the first step is to consider the change of variable $q:=\rho z$ and to split the operator $\mathcal{P}(\rho z):=-D_{t}(\rho z)-D_{h}^{2}(\rho z)$ into a similar form to the continuous case. For the spatial operator we note that using the discrete Leibniz formula from Proposition \ref{lem:product:rule} we have that
\begin{equation*}
    D_{h}^{2}(\rho z)=D_{h}^{2}\rho\,A_{h}^{2}z+2D_{h}A_{h}\rho\,D_{h}A_{h}z+A_{h}^{2}\rho\,D_{h}^{2}z.
\end{equation*}
Then, applying $\taut^-$ in the above equation, and then multiplying by $\taut^-(r)$ we obtain
\begin{equation}\label{eq:space:conjugate}
\begin{split}
 \taut^{-}(r)\taut^{-}D_{h}^{2}(\rho z)=&\taut^{-}\left(rD_{h}^{2}\rho\,A_{h}^{2}z\right)+2\taut^{-}\left(rD_{h}A_{h}\rho\,D_{h}A_{h}z\right)+\taut^{-}\left(rD_{h}^{2}z\,A_{h}^{2}\rho\right).
 \end{split}
\end{equation}

On the other hand, using \eqref{eq:time:product:rule} and then multiplying by $\taut^{-}(r)$ the resulting expression we get
\begin{equation*}
\begin{split}
    \taut^{-}(r)\,D_{t}(\rho z)=&D_{t}z+\taut^{-}(r)D_{t}\rho\,\taut^{+}(z).
    \end{split}
\end{equation*}
By using that $\taut^{+}z=\taut^{-}z+\Delta t D_{t}z$ and Lemma \ref{lem:time:estimate}, provided $\tau\Delta t  (T^3 \delta^2)^{-1}\leq 1$, it follows that
\begin{equation}\label{eq:time:conjugate}
\begin{split}
    \taut^{-}(r)\,D_{t}(\rho z)
    =&D_{t}z-\tau\varphi\taut^{-}(\theta'z)+\Delta t\left(\frac{\tau}{\delta^{3}T^{4}}+\frac{\tau^{2}}{\delta^{4}T^{6}} \mathcal{O}_{\lambda}(1)\right)\taut^{+}(z)+\tau \Delta t\taut^{-}(\theta')\varphi D_{t}z.
    \end{split}
\end{equation}
Now, we set
\begin{equation}
\label{def:Ai:Bj}
\begin{split}
    A_{1}z:=&\taut^{-}(rA_{h}^{2}\rho\,D_{h}^{2}z),\quad A_{2}z:=\taut^{-}\left( rD_{h}^{2}\rho\, A_{h}^{2}z\right),\quad A_{3}z:=-\tau\varphi\taut^{-}(\theta'z),\\
    B_{1}z:=&2\taut^{-}(rD_{h}A_{h}\rho\, D_{h}A_{h}z),\quad B_{2}z:=-2\taut^{-}(s\partial_{x}^{2}\phi\,z),\quad B_{3}z:=D_{t}z.
    \end{split}
\end{equation}
Thus, combining \eqref{eq:space:conjugate} and \eqref{eq:time:conjugate}, we have the following identity
\begin{equation}
\label{eq:ABR}
    Az+Bz=Rz,\text{ in }\mathcal{M}\times \mathcal{N},
\end{equation}
where $A$, $B$ and $R$ are given by
\begin{equation}\label{def:A:B}
    Az:=A_{1}z+A_{2}z+A_{3}z,\quad Bz:=B_{1}z+B_{2}z+B_{3}z,
\end{equation}
and 
\begin{equation}\label{eq:R}
    Rz:=-\taut^{-}(r)\mathcal{P}(\rho z)-\Delta t\left(\frac{\tau}{\delta^{3}T^{4}}+\frac{\tau^{2}}{\delta^{4}T^{6}} \right)\mathcal{O}_{\lambda}(1)\taut^{+}(z)+\tau \Delta t\taut^{-}(\theta')\varphi D_{t}z-2\taut^{-}(s\partial_{x}^{2}\phi z),
\end{equation}
in $\mathcal{M}\times \mathcal{N}$, respectively.

Let us conjugate the equations where the dynamic boundary condition is considered. We note that following the same steps in obtaining equation \eqref{eq:time:conjugate}, and using Lemma \ref{lem:product:rule}, for the left spatial boundary we write  
\begin{align}
    \label{eq:CDR}
    Cz+Dz=R_{\Gamma_0}z\quad\text{in }\{0\}\times\mathcal{N},
\end{align}
where 
\begin{equation*}
\begin{split}
    Cz&:=D_{t}z+\taut^-(r)\taut_{+}^{-}(D_{h}\rho A_{h}z)=:C_{1}z+C_{2}z,\\
    Dz&:=-\tau\varphi\taut^{-}(\theta'z)+\taut^-(r)\taut_{+}^{-} (A_{h}\rho D_{h}z)=:D_{1}z+D_{2}z,\\
    R_{\Gamma_{0}}z&:=-\taut^{-}(r)N_{\Gamma_{0}}(\rho z)-\Delta t\left(\frac{\tau}{\delta^{3}T^{4}}-\frac{\tau^{2}}{\delta^{4}T^{6}} \right)\mathcal{O}_{\lambda}(1)\taut^{+}(z)+\tau \Delta t\taut^{-}(\theta')\varphi D_{t}z,
    \end{split}
\end{equation*}
and $N_{\Gamma_0}(\rho z)=D_{t}(\rho z) +D_h\taut^{-}_{+} (\rho z)$. 

In the same manner, for the right spatial boundary we write
\begin{align}
    \label{eq:EFR}
    Ez+Fz=R_{\Gamma_1}z\text{ in }\{1\}\times\mathcal{N},
\end{align}
where
\begin{equation*}
    \begin{split}
        Ez&:=D_{t}z-\taut^-(r)\taut^{-}_{-} (D_x\rho A_{h}z)=:E_{1}z+E_{2}z,\\
        Fz&:=-\tau\varphi\taut^{-}(\theta'z)+\taut^{-}(r)\taut^{-}_{-} (A_{h}\rho D_{h} z)=:F_{1}z+F_{2}z,\\
        R_{\Gamma_{1}}z&:=-\taut^{-}(r)N_{\Gamma_{1}}(\rho z)-\Delta t\left(\frac{\tau}{\delta^{3}T^{4}}-\frac{\tau^{2}}{\delta^{4}T^{6}} \right)\mathcal{O}_{\lambda}(1)\taut^{+}(z)+\tau \Delta t\taut^{-}(\theta')\varphi D_{t}z,
    \end{split}
\end{equation*}
and $N_{\Gamma_{1}}(\rho z):=D_{t}(\rho z)-D_{h}\taut^{-}_{+}(\rho z)$.



Taking the $L^2(\mathcal{M}\times \mathcal{N})$-norm in \eqref{eq:ABR} and the $L^2(\mathcal{N})$-norm in \eqref{eq:CDR} and \eqref{eq:EFR} , and adding the respective results, we obtain
\begin{align}
    \label{conjugation:01}
    \begin{split} 
    &\int_{\mathcal{M}\times \mathcal{N}}(|Az|^2+ |Bz|^2)+\int_{\mathcal{N}} (|Cz|^2 + |Dz|^2)+\int_{\mathcal{N}} (|Ez|^2 + |Fz|^2)\\
    &+2\int_{\mathcal{M}\times \mathcal{N}}Az\cdot Bz +2 \int_{\mathcal{N}} Cz\cdot Dz +2 \int_{\mathcal{N}} Ez\cdot Fz\\
    =&\int_{\mathcal{M}\times \mathcal{N}} |Rz|^2 + \int_{\mathcal{N}} |R_{\Gamma_0}z|^2 + \int_{\mathcal{N}} |R_{\Gamma_1}z|^2.
    \end{split}
\end{align}
In the next three steps, we devote to compute the three inner products given in the left-hand side of \eqref{conjugation:01} and to estimate its right-hand side.
\subsection{Estimate for the boundary terms}
This section is devoted to present the estimation for the boundary terms, that is, to give a lower bound for $Ez\cdot Fz$ in $\{1\}\times\mathcal{N}$ and $Cz\cdot Dz$ in $\{0\}\times\mathcal{N}$. It worth mentioning that these terms can be absorb with the boundary terms from the next step. We will use the notation 
$\displaystyle J_{ij}:=\int_{\{0\}\times\mathcal{N}}C_{i}zD_{j}z$, thus $\displaystyle \int_{\{0\}\times\mathcal{N}}Cz\cdot Dz=\sum_{i,j}^{2}J_{ij}$. Similarly, we have that $\displaystyle \int_{\{1\}\times\mathcal{N}}Ez\cdot Fz=\sum_{i,j}^{2}K_{ij}$ with $\displaystyle K_{ij}:=\int_{\{1\}\times\mathcal{N}}E_{i}zF_{j}$. The computation of the integrals $J_{ij}$ are in the Appendix. Combining the estimates from the Lemma \ref{lem:J11}-\ref{lem:J22} it follows that 
\begin{equation}
    \begin{split}
    \int_{\mathcal{N}}Cz\cdot Dz+\int_{\mathcal{N}}Ez\cdot Fz+\geq&-\int_{\partial\mathcal{M}\times\mathcal{N}^{\ast}}(\tilde{\alpha}_{11}^{(1)} +\tilde{\alpha}_{21})|z|^{2} -\int_{\partial\mathcal{M}\times\mathcal{N}}\tilde{\alpha}|D_{t}z|^{2}\\
    &-\int_{\partial\mathcal{M}\times\mathcal{N}^{\ast}}\left(1+s+c_{\lambda}s(sh)^{2}\right)t_{r}(|D_{h}z|^{2})\\
    &-\int_{\partial\mathcal{M}\times\mathcal{N}^{\ast}}(\tilde{\alpha}_{21}+s^{2}\mathcal{O}_{\lambda}(1))t_{r}(|A_{h}z|^{2}),
    \end{split}
\end{equation}
where $\tilde{\alpha}^{(1)}_{11}:=(\tau^{2} T^{2}\theta^{3}+\frac{\tau \Delta t}{\delta^{4}T^{5}})\mathcal{O}_{\lambda}(1)$, $\tilde{\alpha}^{(2)}_{11}:=\Delta t\tau T\theta^{2}\mathcal{O}_{\lambda}(1)$, $\tilde{\alpha}_{21}:=\tau^{2}\theta^{3}\mathcal{O}_{\lambda}(1)$, and $\tilde{\alpha}:=(\Delta t)^{2}\taut^{-}(\tilde{\alpha_{11}(1)})-\Delta\taut^{-}(\tilde{\alpha}_{11}^{(2)})+s^{-1}+c_{\lambda}s^{-1}(sh)^{2}$.

\subsection{Estimate for the right-hand side}

Applying the Young's inequality and the triangular inequality we can bound $R$, defined in \eqref{eq:R}, in $\mathcal{M}\times\mathcal{N}$ as follows

\begin{equation}
\begin{split}
    \int_{\mathcal{M}\times\mathcal{N}}|Rz|^{2}\leq&C_{\lambda}\int_{\mathcal{M}\times\mathcal{N}}(\taut^{-}r)^{2}|\mathcal{P}(\rho z)|^{2}+C_{\lambda}\left[ \left(\frac{\Delta t \tau}{\delta^{3}T^{4}}\right)^{2}+\left(\frac{\Delta t\tau^{2}}{\delta^{4}T^{6}}\right)^{2}\right]\int_{\mathcal{M}\times\partial \mathcal{N}^{+}}|\taut^{+}z|^{2}+ X_{R}+W_{R}, 
\end{split}    
\end{equation}
where
\begin{equation*}
    X_{R}:=C_{\lambda}\int_{\mathcal{M}\times \mathcal{N}^{\ast}}s^{2}|z|^{2}+C_{\lambda}\left[ \left(\frac{\Delta t \tau}{\delta^{3}T^{4}}\right)^{2}+\left(\frac{\Delta t\tau^{2}}{\delta^{4}T^{6}}\right)^{2}\right]\int_{\{0\}\times\mathcal{N}^{\ast}}|z|^{2}
\end{equation*}
and $W_{R}$ is given by
\begin{equation*}
    W_{R}:=C_{\lambda}T^{2}(\tau \Delta t)^{2}\int_{\mathcal{M}\times\mathcal{N}}(\taut^{-}\theta)^{4}|D_{t}z|^{2}+C_{\lambda}\frac{\tau^{2}(\Delta t)^{4}}{\delta^{6}T^{8}}\int_{\mathcal{M}\times\mathcal{N}}|D_{t}z|^{2}.
\end{equation*}
Similarly, for $R_{\Gamma_{0}}$ and $R_{\Gamma_{1}}$ we obtain
\begin{equation}
\begin{split}
    \int_{\mathcal{N}}|R_{\Gamma_{0}}|^{2}+|R_{\Gamma_{1}}z|^{2}\leq&  \int_{\mathcal{N}}(\taut^{-}r)^{2}|N_{\Gamma_{0}}(\rho z)|^{2}+C_{\lambda}\int_{\mathcal{N}}(\taut^{-}r)^{2}|N_{\Gamma_{1}}(\rho z)|^{2}\\
    &+C_{\lambda}\left[\left(\frac{\Delta t \tau}{\delta^{3}T^{4}}\right)^{2}+\left(\frac{\Delta t\tau^{2}}{\delta^{4} T^{6}}\right)^z{2}\right]\int_{\partial\mathcal{N}^{+}}|z|^{2}+X_{R_{\Gamma}}+W_{R_{\Gamma}},
    \end{split}
\end{equation}
where
\begin{equation*}
\begin{split}
    X_{R_{\Gamma}}:=&C_{\lambda}\int_{\partial\mathcal{M}\times \mathcal{N}}(\taut^{-}s)^{2}|\taut^{-}z|^{2}+C_{\lambda}\left[ \left(\frac{\Delta t \tau}{\delta^{3}T^{4}}\right)^{2}+\left(\frac{\Delta t\tau^{2}}{\delta^{4}T^{6}}\right)^{2}\right]\int_{\partial\mathcal{M}\times\mathcal{N}^{\ast}}|z|^{2},
    \end{split}
\end{equation*}
and $W_{R_{\Gamma}}$ is given by
\begin{equation*}
    W_{R_{\Gamma}}:=C_{\lambda}T^{2}(\tau \Delta t)^{2}\int_{\partial\mathcal{M}\times\mathcal{N}}(\taut^{-}\theta)^{4}|D_{t}z|^{2}+C_{\lambda}\frac{\tau^{2}(\Delta t)^{4}}{\delta^{6}T^{8}}\int_{\partial\mathcal{M}\times\mathcal{N}}|D_{t}z|^{2}.
\end{equation*}
\subsection{Step 2: Estimates in $\mathcal{M}\times \mathcal{N}$}

Now, we present an estimation for $Az\cdot Bz$ in $\mathcal{M}\times \mathcal{N}$. To do this, we will use the notation
\begin{equation}\label{eq:I_ij}
    I_{ij}:=\int_{\mathcal{M}\times \mathcal{N}} A_iz B_jz,
\end{equation}
where $A_i$ and $B_j$ are given by \eqref{def:Ai:Bj}, with $i,j=1,2,3$. For sake of presentation, we give estimate of each term from \eqref{eq:I_ij}, and its respective proof, in the Appendix. Collecting the estimates from Lemma \ref{I11}-\ref{I33} it follows that for $\epsilon_{1}(\lambda)$ small enough, and $0<\Delta t\tau^{2}(\delta T^{6})^{-1}\leq \epsilon_{1}(\lambda)$ and $0<\tau h(\delta T^{2})^{-1}\leq\epsilon_{1}(\lambda)$, there exists $C_{\lambda} >0$ independent of $\tau$ and $\Delta t$ such that
\begin{equation}\label{eq:A.B}
    \begin{split}
        2\int_{\mathcal{M}\times\mathcal{N}}Az\cdot Bz\geq&2\tau\lambda^{2}\int_{\mathcal{M}^{\ast}\times\mathcal{N}^{\ast}} \theta(\partial_{x}\psi)^2\phi|D_{h}z|^{2}+ 2\tau^{3}\lambda^{4}\int_{\mathcal{M}\times\mathcal{N}^{\ast}}\theta^{3}\phi^{3}(\partial_{x}\psi)^{4}|z|^{2}\\
    &- 2\tau\lambda\int_{\partial\mathcal{M}\times\mathcal{N}^{\ast}} \theta\partial_{x}\psi\phi t_{r}(|D_{h} z|^2)\,n-2\tau^{3}\lambda^{3}\int_{\partial \mathcal{M}\times\mathcal{N}^{\ast}} (\theta\phi\partial_{x}\psi)^{3}t_{r}(|A_{h}z|^{2})n\\
        &-C_{\lambda}\int_{\mathcal{M}^{\ast}\times\partial\mathcal{N}^{+}}|D_{h}\taut^{+}z|^{2}-\int_{\mathcal{M}\times\partial\mathcal{N}}\mathcal{O}_{\lambda}(1)\taut^{+}(s^{2})|\taut^{+}z|^{2}-2X-2Y\\
         &-\int_{\partial\mathcal{M}\times\mathcal{N}}\taut^{-}(\beta)|D_{t}z|^{2},
    \end{split}
\end{equation}
where $X:=X_{11}+X_{22}+X_{21}+X_{13+23}+X_{31}+X_{32}$, $Y:=Y_{11}+Y_{12}-Y_{21}+Y_{13}+Y_{23}-Y_{31}$ and $\beta:=Cs^{-1}+\left[s^{-1}(sh)^{2}+hs^{-1}(sh)^{2}+s^{-1}(sh)^{4}\right]\mathcal{O}_{\lambda}(1)$.
The difference with respect to the equation (68) from \cite{GC-HS-2021} is the boundary term containing $|D_{t}z|^{2}$ on the spatial boundary, due to the Dirichlet condition considered in that work.\\
Thanks to Lemma 2.12 from \cite{GC-HS-2021}, we could add to \eqref{eq:A.B} a positive term  containing $|A_{h}D_{h}z|^{2}$.
Indeed, using the our notation that Lemma takes the form
\begin{lemma}[{\cite[Lemma 2.12]{GC-HS-2021}}]
Let $h_{1}=h_{1}(\lambda)$ be sufficiently small. Then, for $0<h<h_{1}(\lambda)$, we have
\begin{equation}
    \tau\lambda^{2}\int_{\mathcal{M}^{\ast}\times\mathcal{N}^{\ast}}\theta\phi|D_{h}z|^{2}\geq \tau\lambda^{2}\int_{\mathcal{M}\times\mathcal{N}^{\ast}}\theta \phi|A_{h}D_{h}z|^{2}+H-\tilde{X}-J,
\end{equation}
with $\displaystyle H:=\frac{h^{2}\tau\lambda^{2}}{4}\int_{\mathcal{M}\times{\mathcal{N}^{\ast}}}\theta\phi|D_{h}^{2}z|^{2}, 
        \tilde{X}:=h^{2}\int_{\mathcal{M}^{\ast}\times\mathcal{N}^{\ast}}s\mathcal{O}_{\lambda}(1)|D_{h}z|^{2}+\int_{\mathcal{M}\times\mathcal{N}^{\ast}}sh\mathcal{O}_{\lambda}(1)|A_{h}D_{h}z|^{2}$,
and $\displaystyle J:=h^{4}\int_{\mathcal{M}\times\mathcal{N}^{\ast}}s\mathcal{O}_{\lambda}(1)|D_{h}^{2}z|^{2}$.
\end{lemma}
Thus, from \eqref{eq:A.B}, for $\lambda\geq 1$ sufficiently large and for $0<\tau^{2}\Delta t(\delta^{4}T^{6})^{-1}\leq \epsilon_{2}(\lambda)$ and $0<\tau h(\delta T^{2})^{-1}\leq \epsilon_{2}(\lambda)$ small enough, we obtain that
\begin{equation}\label{ine:A.B}
\begin{split}
    &\int_{\mathcal{M}\times\mathcal{N}}(|Az|^{2}+|Bz|^{2})+\tau\int_{\mathcal{M}^{\ast}\times\mathcal{N}^{\ast}} \theta(\partial_{x}\psi)^2\phi|D_{h}z|^{2}+ \tau^{3}\lambda^{4}\int_{\mathcal{M}\times\mathcal{N}^{\ast}}\theta^{3}\phi^{3}(\partial_{x}\psi)^{4}|z|^{2}\\
    &+\tau\lambda^{2}\int_{\mathcal{M}\times\mathcal{N}^{\ast}} \theta(\partial_{x}\psi)^2\phi|A_{h}D_{h}z|^{2}-\tau\lambda\int_{\partial\mathcal{M}\times\mathcal{N}^{\ast}} \theta\partial_{x} \psi\phi t_{r}(|D_{h} z|^2)\,n-\tau^{3}\lambda^{3}\int_{\partial \mathcal{M}\times\mathcal{N}^{\ast}} (\theta\phi\partial_{x}\psi)^{3}t_{r}(|A_{h}z|^{2})n\\
    &\leq  C_{\lambda}\left( \int_{\mathcal{M}\times\mathcal{N}}\tau^{-}(r^{2})|\mathcal{P}(\rho z)|^{2}+\tau\lambda^{2}\int_{\omega^{\ast}\times\mathcal{N}^{\ast}}\theta\phi|D_{h}z|^{2}+\tau^{3}\lambda^{4}\int_{\omega\times\mathcal{N}^{\ast}}\theta^{3}\phi^{3}|z|^{2}\right)\\
    &+C_{\lambda}\left( \int_{\mathcal{M}^{\ast}\times\partial\mathcal{N}^{+}}|D_{h}\taut^{+}z|^{2}+\int_{\mathcal{M}\times\partial\mathcal{N}}\taut^{+}(s^{2})|\taut^{+}z|^{2}+\int_{\partial\mathcal{M}\times\mathcal{N}}\eta  t_{r}(|D_{h}z|^{2})\,n+\int_{\partial\mathcal{M}\times\mathcal{N}}\taut^{-}(\beta)|D_{t}z|^{2}\right)\\
    &+C_{\lambda}\left( \int_{\mathcal{N}}(\taut^{-}r)^{2}|N_{\Gamma_{0}}(\rho z)|^{2}+C_{\lambda}\int_{\mathcal{N}}(\taut^{-}r)^{2}|N_{\Gamma_{1}}(\rho z)|^{2}\right),
\end{split}
\end{equation}
where $\eta:=\left[ s(sh)^{2}+(sh)^{4}\right]\mathcal{O}_{\lambda}(1)$ and $\beta:=Cs^{-1}+\left[s^{-1}(sh)^{2}+hs^{-1}(sh)^{2}+s^{-1}(sh)^{4}\right]\mathcal{O}_{\lambda}(1)$, and we have used that $|\partial_{x}\psi|\geq C> 0$ in $\Omega\setminus \omega$. 

We note that, in particular, considering $z=0$ on $\partial\mathcal{M}$ in \eqref{ine:A.B} we recover the inequality (75) from \cite{GC-HS-2021}.

\subsection{Adding the term of $D_{t}$ and $D_{h}^{2}$}
Firstly, on the boundary we will add the term with $|D_{t}z|^{2}$. From the definition of $Ez$ we have that $D_{t}z=Ez+\taut^{-}(r)\taut^{-}_{-}(D_{h}\rho A_{h}z)$ on $\{1\}\times \mathcal{N}$, then taking $L^{2}(\mathcal{N})-$norm, and considering the Proposition \ref{pro:space:estimate}, it follows that
\begin{equation*}
    \int_{\{1\}\times\mathcal{N}}|D_{t}z|^{2}\leq \int_{\{1 \}\times\mathcal{N}}|Ez|^{2}+C_{\lambda}\int_{\{1\}\times\mathcal{N}}\taut^{-}(s^{2})|\taut_{-}^{-}A_{h}z|^{2}.
\end{equation*}
Similarly, using the definition of $Cz$, for the left boundary we have that 
\begin{equation*}
    \int_{\{0\}\times\mathcal{N}}|D_{t}z|^{2}\leq \int_{\{0\}\times\mathcal{N}}|Cz|^{2}+C_{\lambda}\int_{\{0\}\times\mathcal{N}}\taut^{-}(s^{2})|\taut_{+}^{-}A_{h}z|^{2}.
\end{equation*}
Collecting the above inequalities, and shifting the discrete time variable on the integrals with the term $|A_{h}|$, we obtain
\begin{equation}\label{ine:Dt:boundary}
    \int_{\partial\mathcal{M}\times\mathcal{N}}|D_{t}z|^{2}\leq \int_{\{0\}\times\mathcal{N}}|Cz|^{2}+\int_{\{1\} \times\mathcal{N}}|Ez|^{2}+C_{\lambda}\int_{\partial\mathcal{M}\times\mathcal{N}^{\ast}}s^{2}t_{r}(|A_{h}z|^{2}).
\end{equation}\\
Secondly, in $\mathcal{M}\times\mathcal{N}$, from \eqref{def:Ai:Bj} and \eqref{def:A:B} we have that
\begin{equation*}
    \taut^{-}(rA^{2}_{h}\rho D_{h}^{2}z)=Az-\taut^{-}(rD_{h}^{2}\rho A_{h}^{2}z)+\tau \varphi\taut^{-}(\theta'z).
\end{equation*}
Thus, using Proposition \ref{pro:space:estimate}, Lemma \ref{lem:product:rule} and $|\theta'|\leq CT\theta^{2}$ it follows that the above equation can be estimated as
\begin{equation}\label{eq:Dx2:z}
    D_{h}^{2}\taut^{-}z=Az+\mathcal{O}_{\lambda}(1)\taut^{-}(sh)^{2}D_{h}^{2}\taut^{-}z+\mathcal{O}_{\lambda}(1)\taut^{-}(s^{2}z)+\mathcal{O}_{\lambda}(1)\taut^{-}(sh)^{2}D_{h}^{2}\taut^{-}z +\mathcal{O}_{\lambda}(1)\tau T\taut^{-}(s^{2}z).
\end{equation}
Multiplying \eqref{eq:Dx2:z} by $\taut^{-}s^{-1/2}$, and then taking the $L^{2}(\mathcal{M}\times\mathcal{N})$-norm we obtain, followed by a shift with respect to the time discrete variable,
\begin{equation*}
    \begin{split}
        \int_{\mathcal{M}\times\mathcal{N}^{\ast}}(s^{-1})|D_{h}^{2} z|^{2}\leq & C_{\lambda}\left(\int_{\mathcal{M}\times\mathcal{N}}\taut^{-}(s^{-1})|Az|^{2}+\int_{\mathcal{M}\times\mathcal{N}^{\ast}}s^{-1}(sh)^{4}|D_{h}^{2}z|^{2}\right)\\
        +&C_{\lambda}\left( \int_{\mathcal{M}\times\mathcal{N}^{\ast}}s^{3}|z|^{2}+\int_{\mathcal{M}\times\mathcal{N}^{\ast}}\tau T^{2}\theta^{3}|z|^{2}\right).
    \end{split}
\end{equation*}
Increasing if necessary the value $\tau$ such that $\tau\geq 1$ and $s(t)\geq 1$ for any $t$, we get from the above inequality 
\begin{equation}\label{ine:Dx}
    \begin{split}
        \int_{\mathcal{M}\times\mathcal{N}^{\ast}}s^{-1}|D_{h}^{2} z|^{2}\leq & C_{\lambda}\left(\int_{\mathcal{M}\times\mathcal{N}}|Az|^{2}+ \int_{\mathcal{M}\times\mathcal{N}^{\ast}}s^{3}|z|^{2}\right).
    \end{split}
\end{equation}
Similarly, using the definition of $Bz$ and Proposition \ref{pro:space:estimate} it follows that
\begin{equation}\label{ine:Dt}
    \begin{split}
        \int_{\mathcal{M}\times\mathcal{N}}\taut^{-}s^{-1}|D_{t}z|^{2}\leq C_{\lambda}\left( \int_{\mathcal{M}\times\mathcal{N}}|Bz|^{2}+\int_{\mathcal{M}\times\mathcal{N}^{\ast}}s|A_{h}D_{h}z|^{2}+\int_{\mathcal{M}\times\mathcal{N}^{\ast}}s^{2}|z|^{2}\right).
        \end{split}
\end{equation}
Therefore, combining \eqref{ine:A.B}, \eqref{ine:Dx}, \eqref{ine:Dt} and \eqref{ine:Dt:boundary} we obtain
\begin{equation}\label{ine:C.A:interior}
\begin{split}
    &\iint\limits_{\mathcal{M}\times\mathcal{N}}\taut^{-}(s^{-1})|D_ {t}z|^{2}+\iint\limits_{\mathcal{M}\times\mathcal{N}^{\ast}}s ^{-1}|D_{h}^{2}z|^{2}+\int_{\mathcal{M}^{\ast}\times\mathcal{N}^{\ast}} s|D_{h}z|^{2}+ \int_{\mathcal{M}\times\mathcal{N}^{\ast}}s^{3}|z|^{2}\\
    &+\int_{\mathcal{M}\times\mathcal{N}^{\ast}} s|A_{h}D_{h}z|^{2}+\int_{\partial\mathcal{M}\times\mathcal{N}^{\ast}}st_{r}(|D_{h} z|^2)+\int_{\partial \mathcal{M}\times\mathcal{N}^{\ast}}s^{3}t_{r}(|A_{h}z|^{2})+\int_{\partial \mathcal{M}\times \mathcal{N}}|D_{t}z|^{2}\\
    \leq & C_{\lambda}\left( \int_{\mathcal{M}\times\mathcal{N}}\tau^{-}(r^{2})|\mathcal{P}(\rho z)|^{2}+\tau\lambda^{2}\int_{\omega^{\ast}\times\mathcal{N}^{\ast}}\tau^{-}(\theta)\phi|D_{h}z|^{2}+\tau^{3}\lambda^{4}\int_{\omega\times\mathcal{N}^{\ast}}\taut^{-}(\theta^{3})\phi^{3}|z|^{2}\right)\\
    &+C_{\lambda}\left(\hspace{0.2cm} \iint\limits_{\mathcal{M}^{\ast}\times\partial\mathcal{N}^{+}}\hspace{-0.2cm}|D_{h}\taut^{+}z|^{2}+\int\limits_{\mathcal{M}\times\partial\mathcal{N}}\taut^{+}(s^{2})|\taut^{+}z|^{2} \int_{\mathcal{N}}(\taut^{-}r)^{2}|N_{\Gamma_{0}}(\rho z)|^{2}+C_{\lambda}\int_{\mathcal{N}}(\taut^{-}r)^{2}|N_{\Gamma_{1}}(\rho z)|^{2}\right).
\end{split}
\end{equation}
At this stage, we will remove the local term $|D_{h}z|^{2}$ from the right-hand side by increasing the size of the observation set. This is done in \cite[Lemma 2.16]{GC-HS-2021}, in our notation takes the following form.
\begin{lemma}[{\cite[Lemma 2.16]{GC-HS-2021}}]
For any $\gamma>0$, there exists $C>0$ uniform with respect to $h$ and $\Delta t$ such that
\begin{equation}\label{eq:lem:2.16}
\begin{split}
    \int_{\omega^{\ast}\times\mathcal{N}^{\ast}}s|D_{h}z|^{2}\leq &C\left( 1+\frac{1}{\gamma}\right)\int_{\mathcal{B}^{\ast}\times\mathcal{N}^{\ast}}s^{3}|z|^{2}+C\int_{\mathcal{M}\times\mathcal{N}^{\ast}}|D_{h}A_{h}z|^{2}+\int_{\mathcal{M}\times\mathcal{N}^{\ast}}s|z|^{2}\\
    &+\gamma\int_{\mathcal{M}^{\ast}\times\mathcal{N}^{\ast}}s^{-1}|D_{h}z|^{2}+\int_{\mathcal{M}^{\ast}\times\mathcal{N}^{\ast}}s^{-1}(sh)^{2}|D_{h}z|^{2}.
    \end{split}
\end{equation}
\end{lemma}
Moreover, the time boundary terms can be estimates as follows
\begin{equation*}
    \int_{\mathcal{M}\times\partial\mathcal{N}}s^{2}|\taut^{+}z|^{2}\leq 4h^{-2}\int_{\mathcal{M}\times\partial\mathcal{N}}\left( \frac{\tau h}{\delta T^{2}}\right)^{2}|\taut^{+}z|^{2}\leq 4\epsilon_{6}^{2}h^{-2}\int_{\mathcal{M}\times\partial\mathcal{N}}|\taut^{+}z|^{2},
\end{equation*}
since $\Delta t\leq \delta T/2$ and therefore $\max_{t\in[0,T+\Delta t]}\theta(t)\leq 2/(\delta T^{2})$. For the spatial discrete derivative on the time discrete boundary, we note that $|D_{h}z|^{2}\leq Ch^{-2}\left( |\taut_{+}z|^{2}+|\taut_{-}z|^{2}\right)$. Then 
\begin{equation*}
    \int_{\mathcal{M}^{\ast}\times\partial\mathcal{N}}|D_{h}\taut^{+}z|^{2}\leq Ch^{-2}\int_{\mathcal{\overline{M}}\times\partial\mathcal{N}}|\taut^{+}z|^{2}.
\end{equation*}
Thus, we obtain
\begin{equation}\label{eq:boundary:time:Dx}
\int_{\mathcal{M}\times\partial\mathcal{N}}s^{2}|\taut^{+}z|^{2}+\int_{\mathcal{M}^{\ast}\times\partial\mathcal{N} }|D_{h}\taut^{+}z|^{2}\leq \left(1+4\epsilon_{6}^{2} \right)h^{-2}\int_{\mathcal{M}\times\times\mathcal{N}}|\taut^{+}z|^{2}+Ch^{-2}\int_{\partial\mathcal{M}\times\partial\mathcal{N}}|\taut^{+}z|^{2}.
\end{equation}
\par Therefore, combining \eqref{ine:C.A:interior}, \eqref{eq:lem:2.16} and \eqref{eq:boundary:time:Dx} we obtain
\begin{equation}\label{ine:A.B:dual:variable}
\begin{split}
    &\int_{\mathcal{M}\times\mathcal{N}}\taut^{-}(s^{-1})|D_ {t}z|^{2}+\int_{\mathcal{M}\times\mathcal{N}^{\ast}}s ^{-1}|D_{h}^{2}z|^{2}+\int_{\mathcal{M}^{\ast}\times\mathcal{N}^{\ast}} s|D_{h}z|^{2}+ \int_{\mathcal{M}\times\mathcal{N}^{\ast}}s^{3}|z|^{2}\\
    &+\int_{\mathcal{M}\times\mathcal{N}^{\ast}} s|A_{h}D_{h}z|^{2}+\int_{\partial\mathcal{M}\times\mathcal{N}^{\ast}}st_{r}(|D_{}h z|^2)+\int_{\partial \mathcal{M}\times\mathcal{N}^{\ast}}s^{3}t_{r}(|A_{h}z|^{2})+\int_{\partial \mathcal{M}\times \mathcal{N}}|D_{t}z|^{2}\\
    \leq & C_{\lambda_{1}}\left( \int_{\mathcal{M}\times\mathcal{N}}\tau^{-}(r^{2})|\mathcal{P}(\rho z)|^{2}+\int_{\mathcal{B}^{\ast}\times\mathcal{N}^{\ast}}s^{3}|z|^{2}+ h^{-2}\int_{\partial\mathcal{M}\times\partial\mathcal{N}}|\taut^{+}z|^{2}+h^{-2}\int_{\mathcal{M}\times\partial\mathcal{N}}|\taut^{+}z|^{2}\right)\\
    &+C_{\lambda_{1}}\left(  \int_{\mathcal{N}}(\taut^{-}r)^{2}|N_{\Gamma_{0}}(\rho z)|^{2}+C_{\lambda}\int_{\mathcal{N}}(\taut^{-}r)^{2}|N_{\Gamma_{1}}(\rho z)|^{2}\right),
\end{split}
\end{equation}
where we have chosen $\gamma=\frac{1}{2C_{\lambda_{1}}}$ and $\tau\geq \tau_{2}(T+T^{2})$ such that the extra terms from \eqref{eq:lem:2.16} can be absorbed. \\
\subsection{Returning to the original variable}
Finally, using the following Lemma we can return to the original variable
\begin{lemma}\label{lem:return:variable}
For $\Delta t\tau (T^{3}\delta^{2})^{-1}\leq 1$ and $\tau h(\delta T^{2})^{-1}\leq 1$ we have the following estimates for the term $|D_{h}q|$, $|A_{h}D_{h}q|$, $|D_{h}^{2}q|$ and $|D_{t}q|$ respectively
\begin{equation*}\label{eq:return:Dx}
        \int_{\mathcal{M}^{\ast}\times\mathcal{N}^{\ast}}r^{2}s|D_{h}q|^{2}\leq C_{\lambda}\left( \int_{\mathcal{M}^{\ast}\times\mathcal{N}^{\ast}}s^{3}|z|^{2}+\int_{\partial \mathcal{M}\times\mathcal{N}^{\ast}}s^{2}sh|z|^{2}+\int_{\mathcal{M}^{\ast}\times\mathcal{N}^{\ast}}s|D_{h}z|^{2}\right),
\end{equation*}
\begin{equation*}\label{eq:return:AxDx}
\begin{split}
    \int_{\mathcal{M}\times\mathcal{N}^{\ast}}sr^{2}|A_{h}D_{h}q|^{2}\leq&C_{\lambda} \left( \int_{\mathcal{M}\times\mathcal{N}^{\ast}}s^{3}|z|^{2}+\int_{\mathcal{M}\times\mathcal{N}^{\ast}}s|A_{h}D_{h}z|^{2}\right)\\
    &+C_{\lambda}\left( \int_{\mathcal{M}\times\mathcal{N}^{\ast}}s^{-1}(sh)^{4}|D_{h}z|^{2}+\int_{\mathcal{M}\times\mathcal{N}^{\ast}}s(sh)^{4}|A_{h}D_{h}z|^{2}\right),
    \end{split}
\end{equation*}
\begin{equation*}\label{eq:return:DxDx}
\begin{split}
    \int_{\mathcal{M}\times\mathcal{N}^{\ast}}r^{2}s^{-1}|D_{h}^{2}q|^{2}\leq &C_{\lambda}\left( \int_{\mathcal{M}\times\mathcal{N}^{\ast}}s^{3}|z|^{2}+\int_{\mathcal{M}\times\mathcal{N}^{\ast}}s^{-1}|D_{h}^{2}z|^{2}+\int_{\mathcal{M}\times\mathcal{N}^{\ast}}s|A_{h}D_{h}z|^{2}\right)\\
    &+C_{\lambda}\frac{h^{4}}{\delta^{4}T^{8}}\left( \int_{\mathcal{M}\times\mathcal{N}^{\ast}}s^{-1}|D_{h}^{2}z|^{2}\right),
    \end{split}
\end{equation*}
\begin{equation*}\label{eq:return:Dt}
    \int_{\mathcal{M}\times\mathcal{N}}\taut^{-}(r^{2}s^{-1})|D_{t}q|^{2}\leq 2\int_{\mathcal{M}\times\mathcal{N}}\taut^{-} r^{2}|\mathcal{P}q|^{2}+2\int_{\mathcal{M}\times\mathcal{N}}\taut^{-}r^{2}s^{-1}|D_{h}^{2}\taut^{-}q|^{2}.
\end{equation*}
Additionally, for the spatial boundary terms we have
\begin{equation*}\label{eq:return:Ax}
    \int_{\partial \mathcal{M}\times\mathcal{N}^{\ast}}s^{3}r^{2}t_{r}(A_{h}|q|^{2})\leq C_{\lambda}\int_{\partial\mathcal{M}\times\mathcal{N}^{\ast}}t_{r}(|A_{h}z|^{2})+C_{\lambda}\int_{\partial\mathcal{M}\times\mathcal{N}^{\ast}}st_{r}(|D_{h}z|^{2}),
\end{equation*}
\begin{equation*}\label{eq:return:Dx:boundary}
    \int_{\partial\mathcal{M}\times\mathcal{N}^{\ast}}sr^{2}t_{r}(|D_{h}q|^{2})\leq C_{\lambda}\left(\int_{\partial\mathcal{M}\times\mathcal{N}^{\ast}}s^{3}t_{r}(|A_{h}z|^{2})+\int_{\partial\mathcal{M}\times\mathcal{N}^{\ast}}st_{r}(|D_{h}z|^{2})\right),
\end{equation*}
and
\begin{equation*}\label{eq:return:boundary}
    \int_{\partial\mathcal{M}\times\mathcal{N}}\taut^{-}(r^{2})|D_{t}q|^{2}\leq C_{\lambda}\int_{\partial\mathcal{M}\times\mathcal{N}^{\ast}}st_{r}(|D_{h}z|^{2})+C_{\lambda}\int_{\partial\mathcal{M}\times\mathcal{N}^{\ast}}s^{3}|A_{h}z|^{2}.
\end{equation*}
\end{lemma}
Combining the inequalities from the Lemma above and \eqref{ine:A.B:dual:variable} we obtain
\begin{equation*}
\begin{split}
    &\int_{\mathcal{M}\times\mathcal{N}}\taut^{-}(s^{-1})|D_ {t}q|^{2}+\int_{\mathcal{M}\times\mathcal{N}^{\ast}}s ^{-1}|D_{h}^{2}q|^{2}+\int_{\mathcal{M}^{\ast}\times\mathcal{N}^{\ast}} s|D_{h}q|^{2}+ \int_{\mathcal{M}\times\mathcal{N}^{\ast}}s^{3}|q|^{2}\\
    &+\int_{\mathcal{M}^{\ast}\times\mathcal{N}^{\ast}} s|A_{h}D_{h}q|^{2}+\int_{\partial\mathcal{M}\times\mathcal{N}^{\ast}}st_{r}(|D_{h} q|^2)+\int_{\partial \mathcal{M}\times\mathcal{N}^{\ast}}s^{3}t_{r}(A_{h}(|q|^{2}))+\int_{\partial \mathcal{M}\times \mathcal{N}}|D_{t}q|^{2}\\
    \leq & C_{\lambda_{1}}\left( \int_{\mathcal{M}\times\mathcal{N}}\tau^{-}(r^{2})|\mathcal{P}q|^{2}+\int_{\mathcal{B}\times\mathcal{N}^{\ast}}s^{3}r^{2}|q|^{2}+ h^{-2}\int_{\partial\mathcal{M}\times\partial\mathcal{N}}|\taut^{+}rq|^{2}+h^{-2}\int_{\mathcal{M}\times\partial\mathcal{N}}|\taut^{+}rq^{2}\right)\\
    &+C_{\lambda_{1}}\left(  \int_{\mathcal{N}}(\taut^{-}r)^{2}|N_{\Gamma_{0}}q|^{2}+C_{\lambda}\int_{\mathcal{N}}(\taut^{-}r)^{2}|N_{\Gamma_{1}}q|^{2}\right),
\end{split}
\end{equation*}
for all $\tau\geq \tau_{0}(T+T^{2})$, $h\leq h_{0}$, $\displaystyle \frac{\tau h}{\delta T^{2}}\leq \epsilon_{0}$ and  $\displaystyle \frac{\tau^{4}\Delta t }{\delta^{4}T^{6}}\leq \epsilon_{0}$, which completes the proof of Theorem \ref{theo:discrete:carleman}.

\section{$\phi(h)$-null controllability}\label{sec:obs}
\subsection{Proof of Theorem \ref{observ-SZ}}
In this section, we prove the relaxed observability inequality \eqref{ine:discrete:obs} for the adjoint system \eqref{discrete primal system}, that is
\begin{equation}\label{system:adjoint}
    \begin{cases}
    -D_t q-D_{h}^2 \taut^- q+b\taut^- (q)=0,& (x,t)\in \mathcal{M}\times \mathcal{N},\\
    -D_t q(0,t)-D_{h} \taut^{-}_{+} q(0,t)+b_{\Gamma_0}(t)\taut^- q(0,t)=0,& t\in \mathcal{N},\\
    -D_t q(1,t)+D_{h} \taut^{-}_{-} q(1,t)+b_{\Gamma_1}(t) \taut^- q(1,t)=0,& t\in \mathcal{N},\\
    q(x,T+\Delta t/2)=q_{T}(x),&  x\in \overline{\mathcal{M}}.
    \end{cases}
\end{equation}
Firstly, let us define $\gamma:=\max\{\left\| b\right\|_{L_h^\infty(\mathcal{M})},\left\| b_{\Gamma_{0}}\right\|_{L_h^\infty(\mathcal{N})},\left\| b_{\Gamma_{1}}\right\|_{L_h^\infty(\mathcal{N})}\}$ then we notice that the change of variable $\tilde{q}:=qa^{(t-T)\gamma}$, for $\gamma \Delta t<1$ and with 
$$a:=\left( \frac{1}{1-\gamma\Delta t}\right)^{1/\gamma\Delta t},$$
enables us to consider the potential $b$, in $\overline{\mathcal{M}}$, to be non-negative since in this case it follows that $D_{t}a^{\gamma t}=\gamma\taut^{+}a^{\gamma t}$. 

Now, our goal is to prove that the solution of the system \eqref{system:adjoint} verifies that $\left\|\taut^{+}q(0) \right\|_{\mathbb{L}_{h}^{2}(\mathcal{M})}\leq \left\|\taut^{+}q(t) \right\|_{\mathbb{L}_{h}^{2}(\mathcal{M})}$ for $t\in\mathcal{N}$, where we recall that  
$\displaystyle
\left\|u \right\|_{\mathbb{L}_{h}^{2}(\mathcal{W})}^2:=\left\| u\right\|_{L^{2}_{h}(\mathcal{W})}^2+\left\|u \right\|_{L^{2}_{h}(\partial\mathcal{W})}^{2}$. To this end, we consider $t\in \mathcal{N}$, $q\in C(\overline{\mathcal{M}}\times \mathcal{N}^{\ast})$ and define $\mathcal{N}_{t}:=(0,t+\Delta t/2)\cap\mathcal{N}$. Then, taking the $L^{2}_{h}(\mathcal{M}\times{\mathcal{N}_{t}})$-norm between main equation of the above system with $\taut^{-}q\in C(\overline{\mathcal{M}}\times \mathcal{N}_t)$ we have that
\begin{equation*}
    -\int_{\mathcal{M}\times\mathcal{N}_{t}}D_{t}q\taut^{-}q-\int_{\mathcal{M}\times\mathcal{N}_{t}}D_{h}^{2}\taut^{-}q\taut^{-}q+\int_{\mathcal{M}\times\mathcal{N}_{t}}b|\taut^{-}q|^{2}=0.
\end{equation*}
We now estimate the first two integrals from the left-hand side above. By virtue of \eqref{eq:identity:square}, \eqref{eq:tau:mas} and Proposition \ref{pro:integral:space} we see that
\begin{equation}\label{ine:main:equation}
    0\geq-\int_{\mathcal{M}\times\mathcal{N}_{t}}D_{t}q\taut^{-}q-\int_{\mathcal{M}\times\mathcal{N}_{t}}D_{h}^{2}\taut^{-}q\taut^{-}q\geq -\int_{\mathcal{M}\times\partial\mathcal{N}_{t}}|\taut^{+}q|^{2}\,n-\int_{\partial\mathcal{M}\times\mathcal{N}_{t}}\taut^{-}qt_{r}(D_{h}\taut^{-}q)n.
\end{equation}
Similarly, we derive from the boundary conditions of the system \eqref{system:adjoint} this
\begin{equation*}
    -\int_{\{0\}\times\mathcal{N}}D_{t}q\taut^{-}q-\int_{\{0\}\times\mathcal{N}_{t}}D_{h}\taut^{-}_{+}q\taut^{-}q\geq-\int_{\{0\}\times\mathcal{N}_{t}}D_{h}\taut^{-}_{+}q\taut^{-}q -\int_{\{0\}\times\partial\mathcal{N}_{t}}|\taut^{+}q|^{2}\,n
\end{equation*}
and 
\begin{equation*}
    -\int_{\{1\}\times\mathcal{N}_{t}}D_{t}q\taut^{-}q-\int_{\{1\}\times\mathcal{N}_{t}}D_{h}\taut^{-}_{-}q\taut^{-}q\geq -\int_{\{1\}\times\mathcal{N}_{t}}D_{h}\taut^{-}_{-}q\taut^{-}q -\int_{\{1\}\times\partial\mathcal{N}_{t}}|\taut^{+}q|^{2}\,n.
\end{equation*}
Adding the latter inequalities yields
\begin{equation}\label{ine:boundary}
    -\int_{\partial\mathcal{M}\times\mathcal{N}_{t}}D_{t}q\taut^{-}q-\int_{\partial\mathcal{M}\times\mathcal{N}_{t}}t_{r}(D_{h}\taut^{-}q)\taut^{-}qn\geq -\int_{\partial\mathcal{M}\times\partial\mathcal{N}_{t}}|\taut^{+}q|^{2}\,n.
\end{equation}
Next, combining \eqref{ine:main:equation} and \eqref{ine:boundary} entail that
\begin{equation}\label{ine:dissipative}
    0\geq  -\int_{\mathcal{M}\times\partial\mathcal{N}_{t}}|\taut^{+}q|^{2}\,n -\int_{\partial\mathcal{M}\times\partial\mathcal{N}_{t}}|\taut^{+}q|^{2}\,n,
\end{equation}
 for all $t\in\mathcal{N}$ which completes the proof of our claim. 

On the other hand, we now focus on the observability inequality for the system \eqref{system:adjoint}.
From the Carleman estimate \eqref{eq:discrete:carleman} we see that
\begin{equation}\label{ine:almost:observability}
\begin{split}
    \tau^{3}\int_{\partial\mathcal{M}\times\mathcal{N}^{\ast}}e^{2\taut\theta\varphi}|q|^{2}+\tau^{3}\int_{\mathcal{M}\times\mathcal{N}^{\ast}}e^{2\tau\theta\varphi}|q|^{2}\leq & C\tau^{3}\int_{\omega\times\mathcal{N}^{\ast}}e^{2\tau\theta\varphi}\theta^{3}|q|^{2}\\
    &+ Ch^{-2}\int_{\partial\mathcal{M}\times\partial\mathcal{N}}|\taut^{+}e^{\tau\theta\varphi}q|^{2}\\
    &+Ch^{-2}\int_{\mathcal{M}\times\partial\mathcal{N}}|\taut^{+}e^{\tau\theta\varphi}q|^{2},
    \end{split}
\end{equation}
provided that $\tau_{1}\geq \tau_{0}\left(T+T^{2}+T^{2}\gamma^{2/3}\right)$. Our first task is to estimate the left-hand side of the above inequality. Let us consider
\begin{equation*}
    \eta(t):=\begin{cases}
    1 &t\in\mathcal{N^{\ast}}\cap(M/4,3M/4),\\
    0& \text{otherwise}.
    \end{cases}
\end{equation*}
Then, we have that
\begin{multline*}
    \tau^{3}\int_{\partial\mathcal{M}\times\mathcal{N}^{\ast}}e^{2\tau\theta\varphi}\theta^{3}|q|^{2}+\tau^{3}\int_{\mathcal{M}\times\mathcal{N}^{\ast}}e^{2\tau\theta\varphi}\theta^{3}|q|^{2}\geq\tau^{3}\int_{\mathcal{M}\times\mathcal{N}^{\ast}}\eta(t)e^{2\tau\theta\varphi}\theta^{3}|q|^{2}\\
    +\tau^{3}\int_{\partial\mathcal{M}\times\mathcal{N}^{\ast}}\eta(t)e^{2\tau\theta\varphi}\theta^{3}|q|^{2}.
\end{multline*}
Using that $e^{2\tau\theta\varphi}\geq e^{-\frac{2^{5}\tau K_{0}}{3T^{2}}}$, for $t\in\mathcal{N}\cap(M/4,3M/4)$ where $K_{0}:=\max_{\overline{\Omega}}\{-\varphi\}$, we obtain
\begin{equation*}
    \tau^{3}\int_{\mathcal{M}\times\mathcal{N}^{\ast}}e^{2\tau\theta\varphi}\theta^{3}|q|^{2}\geq \frac{\tau^{3}}{T^{6}}e^{\frac{2^{5}\tau K_{0}}{3T^{2}}}\left(\int_{\mathcal{M}\times\mathcal{N}^{\ast}}\eta(t)|q|^{2}+\int_{\partial\mathcal{M}\times\mathcal{N}^{\ast}}\eta(t)|q|^{2}\right).
\end{equation*}
Applying the inequality \eqref{ine:dissipative}
 on the later inequality gives
 \begin{equation}\label{ine:lhs}
    \tau^{3}\int_{\partial\mathcal{M}\times\mathcal{N}^{\ast}}e^{2\tau\theta\varphi}\theta^{3}|q|^{2}+ \tau^{3}\int_{\mathcal{M}\times\mathcal{N}^{\ast}}e^{2\tau\theta\varphi}\theta^{3}|q|^{2}\geq CTe^{-\frac{C\tau}{T^{2}}}\left(\int_{\mathcal{M}}|\tau^{+}q(0)|^{2}+\int_{\partial\mathcal{M}}|\taut^{+}q(0 )|^{2}\right),
\end{equation}
where we have used that $\theta\geq 1/T^{2}$.

The second task is to bound the last two integral from the right-hand side of \eqref{ine:almost:observability}. Using \eqref{ine:dissipative} shows that
\begin{equation}\label{ine:rhs}
    \begin{split}
    h^{-2}\int_{\partial\mathcal{M}\times\partial\mathcal{N}}|\taut^{+}e^{\tau\theta\varphi}q|^{2}+h^{-2}\int_{\mathcal{M}\times\partial\mathcal{N}}|\taut^{+}e^{\tau\theta\varphi}q|^{2}\leq h^{-2}e^{-\frac{4k_{0}\tau}{\delta  T^{2}}}\left\|q_{T}\right\|^{2}_{\mathbb{L}^{2}_{h}(\mathcal{M})},
    \end{split}
\end{equation}
where $k_{0}:=\min_{\Omega}\{-\varphi(x)\}$.

As a final task, we collect the estimates obtained above. Recalling the change of variable at the beginning of this section and using that $\frac{1}{1-\gamma\Delta t}< e^{2\gamma\Delta t}$ provided that $0<\gamma\Delta t <1/2$, we obtain combining \eqref{ine:lhs} and \eqref{ine:rhs} on \eqref{ine:almost:observability} that
\begin{equation*}
   \left\|\taut^{+}q(0) \right\|^{2}_{\mathbb{L}_{h}^{2}(\mathcal{M})}\leq CT^{-1}e^{\frac{C\tau}{T^{2}}+CT\gamma}\left(\int_{\omega\times\mathcal{N}^{\ast}}|q|^{2}\right)+ Ch^{-2}e^{\frac{\tau}{T^{2}}(C-\frac{C'}{\delta})+CT\gamma}\left\|q_{T} \right\|^{2}_{\mathbb{L}_{h}^{2}(\mathcal{M})},
\end{equation*}
for $\tau\geq \tau_{2}(T+T^{2}+T^{2}\gamma)$ with $\tau_{2}=\max\{\tau_{1},3/8k_{0}\}$. Then, for $0<\delta\leq \delta_{1}<1/2$ small enough it follows from the above inequality
\begin{equation}\label{ine:almost:ob}
   \left\|\taut^{+}q(0) \right\|^{2}_{\mathbb{L}_{h}^{2}(\mathcal{M})}\leq CT^{-1}e^{\frac{C\tau}{T^{2}}+CT\gamma}\int_{\omega\times\mathcal{N}^{\ast}}|q|^{2}+ Ch^{-2}e^{-\frac{\tau C}{T^{2}\delta}+CT\gamma}\left\|q_{T} \right\|^{2}_{\mathbb{L}_{h}^{2}(\mathcal{M})}.
\end{equation}
The last step, which is standard, is to connect the discrete parameters $h,\Delta t$ and the parameter $\delta$. The discrete Carleman estimates holds provided that
\begin{equation}\label{ine:conditions}
    \frac{\tau h}{\delta T^{2}}\leq \epsilon_{0} \quad \text{ and }\quad \frac{\tau^{4}\Delta t}{\delta^{4}T^{6}}\leq \epsilon_{0}.
\end{equation}
Moreover, from the previous steps to obtain \eqref{ine:almost:ob} we need $0<h\leq h_{0}$, $0<\delta\leq \delta_{1}$, $\gamma\Delta t \leq 1/4$ and $\tau \geq \tau_{2}(T+T^{2}+T^{2}\gamma^{2/3})$. 
Let us fix $\tau=\tau_{2}\left( T+T^{2}+T^{2}\gamma^{2/3}\right)$ and for $\mu\geq 1$ we define 
\begin{equation}\label{eq:h}
    h_{1}:=\epsilon_{0}\left( \frac{\delta_{1}}{\tau_{}2}\left( 1+\frac{1}{T}+\gamma^{2/3}\right)^{-1}\right)^{M_{\mu}},
\end{equation}
and $\tilde{\Delta t}:=h_{1}^{\mu}$,
where $M_{\mu}:=\max\{1,4/\mu\}$. These definitions imply 
\begin{equation*}
\frac{h_{1}\tau^{M_{\mu}}}{\delta_{1}^{M_{\mu}}T^{2M_{\mu}}}=\epsilon_{0} \text{ and }\frac{\tilde{\Delta t}\tau^{M_{\mu}\mu}}{\delta^{M_{\mu}\mu}T^{2M_{\mu}\mu}}=\epsilon_{0}^{\mu}.
\end{equation*}
Choosing $h\leq \min\{h_{0},h_{1}\}$, such that $h/h_{1}\leq 1$, and setting $\displaystyle \delta=\left( \frac{h}{h_{1}}\right)^{m_{\mu}}\delta_{1}\leq \delta_{1}$ with $m_{\mu}:=\min\{\mu/4,1\}$ we get
\begin{equation}\label{ine:epsilon:zero}
    \epsilon_{0}=\frac{\tau^{M_{\mu}}h}{\delta^{M_{\mu}}T^{2M_{\mu}}}\geq \frac{\tau h}{\delta T^{2}}
\end{equation}
since $M_{\mu}\geq 1$. The above inequality is the first condition from \eqref{ine:conditions}. For the second condition, we note that
\begin{equation*}
    \epsilon_{0}^{\mu}=\frac{\tilde{\Delta t}\tau^{M_{\mu}\mu}}{\delta_{1}^{M_{\mu}\mu}T^{2M_{\mu}\mu}}=\frac{\tau^{M_{\mu}\mu}h^{\mu}}{\delta^{M_{\mu}\mu}T^{2}M_{\mu}\mu}.
\end{equation*}
Then, choosing $\Delta t\leq \min\{T^{-2}h^{\mu},(4\gamma)^{-1}\}$ we obtain
\begin{equation*}
    \frac{\Delta t \tau^{4}}{\delta^{4}T^{6}}\leq \epsilon_{0}^{\mu}\leq \epsilon_{0},
\end{equation*}
since $M_{\mu}\mu\geq 4$.

From equation \eqref{ine:epsilon:zero} we have that $\displaystyle \frac{\tau}{\delta T^{2}}=\left(\frac{\epsilon_{0}}{h} \right)^{1/M_{\mu}}=\left( \frac{\epsilon_{0}}{h}\right)^{m_{\mu}}$ for $\mu\geq 1$. Using this and \eqref{eq:h} on the inequality \eqref{ine:almost:ob} it follows that
\begin{equation*}
    \left\|\taut^{+}q(0) \right\|^{2}_{\mathbb{L}_{h}^{2}(\mathcal{\overline{M}})}\leq CT^{-1}e^{C(1+\frac{1}{T}+\gamma^{2/3}+T\gamma)}\int_{\omega\times\mathcal{N}^{\ast}}|q|^{2}+ Ch^{-2}e^{-C\frac{\epsilon_{0}^{m_{\mu}}}{h^{m_{\mu}}}+CT\gamma}\left\|q_{T}\right\|^{2}_{\mathbb{L}_{h}^{2}(\mathcal{\overline{M}})}
\end{equation*}
Hence,
\begin{equation*}
    \left\|\taut^{+}q(0) \right\|^{2}_{\mathbb{L}_{h}^{2}(\mathcal{\overline{M}})}\leq C_{\text{obs}}\left(\int_{\omega\times\mathcal{N}^{\ast}}|q|^{2}+ e^{-\frac{C_{1}}{h^{\min\{\mu/4,1\}}}}\left\|q_{T}\right\|^{2}_{\mathbb{L}_{h}^{2}(\mathcal{\overline{M}})}\right)
\end{equation*}
with $C_{\text{obs}}:=e^{C(1+\frac{1}{T}+\gamma^{2/3}+T\gamma)}$, which completes the proof.

\subsection{Proof Theorem \ref{theo:control}}
In this section, we use Theorem \ref{observ-SZ} to derive a controllability result for the discrete parabolic system \eqref{discrete primal system}. 
Let us consider the following discrete penalized functional
\begin{equation}
\begin{split}
    J(q_{T}):=&\frac{1}{2}\int_{\mathcal{B}\times\mathcal{N}^{\ast}}|q|^{2}+\frac{\phi(h)}{2}\left\|q_{T} \right\|^{2}_{\mathbb{L}^{2}_{h}(\mathcal{M})}+\int_{\mathcal{M}}y(0)\taut^{+}q(0)+\int_{\partial\mathcal{M}}y(0)\taut^{+}q(0).
\end{split}    
\end{equation}
where $q$ is solution of the system \eqref{discrete primal system}, $y(0)$ is the initial datum of \eqref{semi-discrete:model:ops} and $\phi(h):=e^{-\frac{C}{h^{\min\{\mu/4,1\}}}}$. Thanks to Cauchy--Schwartz and Young inequalities, and the observability inequality \eqref{ine:discrete:obs} it follows that the functional $J$ verifies
\begin{equation*}
\begin{split}
    J(q_{T})\geq& \frac{1}{2}\int_{\mathcal{B}\times\mathcal{N}^{\ast}}|q|^{2}+ \frac{\phi(h)}{2}\left\|q_{T} \right\|^{2}_{\mathbb{L}^{2}_{h}(\mathcal{M})}-\frac{1}{4C_{\text{obs}}^{2}}\left\|\taut^{+}q(0)\right\|^{2}_{\mathbb{L}^{2}(\mathcal{M})}-C_{\text{obs}}^{2}\left\|y(0) \right\|^{2}_{\mathbb{L}^{2}_{h}(\mathcal{M})}\\
    \geq &\frac{1}{4}\int_{\mathcal{B}\times\mathcal{N}^{\ast}}|q|^{2}+ \frac{\phi(h)}{4}\left\|q_{T}
    \right\|^{2}_{\mathbb{L}^{2}_{h}(\mathcal{M})}-C_{\text{obs}}^{2}\left\|y(0) \right\|^{2}_{\mathbb{L}^{2}_{h}(\mathcal{M})}.
    \end{split}
\end{equation*}
The above inequality entails that the functional $J$ is coercive, and hence there exists a unique minimizer $\tilde{q}_T$. Let  $\tilde{q}$ the solution of \eqref{discrete primal system} associated to the initial data $\tilde{q}_T$. Thus, from the Euler-Lagrange equation, $\tilde{q}$ verifies
\begin{equation}\label{eq:E:L}
    \int_{\mathcal{B}\times\mathcal{N}^{\ast}}\tilde{q}q+\phi(h)\langle \tilde{q}_{T},q_{T}\rangle_{\mathbb{L}^{2}_{h}(\mathcal{M})}=-\langle y(0),\taut^{+}q(0)\rangle_{\mathbb{L}^{2}_{h}(\mathcal{M})},
\end{equation}
for any $q\in C(\overline{\mathcal{M}}\times \overline{\mathcal{N}}^*)$. Setting the control $v(x,t):=\mathbbm{1}_{\mathcal{B}}\tilde{q}(t)$, we consider $y$ being solution of the following controlled system
\begin{align}
    \begin{cases}
    D_{t} y(x,t)-D_{h}^2 \taut^{+}y(x,t)+\taut^{+}(b y)(x,t)=\mathbbm{1}_{\mathcal{B}}\tilde{q}(x,t),& (x,t)\in \mathcal{M}\times \mathcal{N}^\ast,\\
    D_{t} y(0,t)-D_{h} \taut^{+}_{+}y(0,t)+\taut^{+}( b_{\Gamma_0}y)(0,t)=0,&  t\in \mathcal{N}^\ast,\\
    D_{t} y(1,t)+D_{h} \taut^{+}_{-}y(1,t)+\taut^{+}(yb_{\Gamma_1})(1,t)=0,& t\in \mathcal{N}^\ast,\\
    y(x,0)=g(x),&\forall x\in \overline{\mathcal{M}}.
    \end{cases}
\end{align}
By duality, is follows that 
\begin{equation}\label{eq:duality}
    \int_{\mathcal{B}\times\mathcal{N}^{\ast}}\tilde{q}q=\langle y(T),q_{T}\rangle_{\mathbb{L}^{2}_{h}(\mathcal{M})}-\langle y(0),\taut^{+}q(0)\rangle_{\mathbb{L}^{2}_{h}(\mathcal{M})},
\end{equation}
for all initial datum $q_{T}$. Then, combining \eqref{eq:E:L} and \eqref{eq:duality} we obtain
\begin{equation}\label{eq:y:q}
    y(x,T)=-\phi(h)\tilde{q}_{T}, \quad \forall x\in\overline{\mathcal{M}}.
\end{equation}
Therefore, from equation \eqref{eq:E:L}, taking $q_{T}=\tilde{q}_{T}$ enables us to write
\begin{equation*}
    \int_{\omega\times\mathcal{N}^{\ast}}|\tilde{q}|^{2}+\phi(h)\left\| \tilde{q}_{T}\right\|^{2}_{\mathbb{L}^{2}_{h}(\mathcal{M})}\leq \left\|y(0) \right\|^{2}_{\mathbb{L}^{2}_{h}(\mathcal{M})}\left\|\taut^{+}\tilde{q}(0) \right\|^{2}_{\mathbb{L}^{2}_{h}(\mathcal{M})}.
\end{equation*}
Thus, using the observability inequality \eqref{ine:discrete:obs} on the above estimation we get
\begin{equation*}
    \left\|v\right\|_{L^{2}_{h}(\mathcal{B}\times\mathcal{N}^{\ast})}\leq C_{\text{obs}}\left\|y(x,0) \right\|_{\mathbb{L}^{2}_{h}(\mathcal{M})}
\end{equation*}
and
\begin{equation*}
    \sqrt{\phi(h)}\left\| \tilde{q}_{T}\right\|_{\mathbb{L}^{2}_{h}(\mathcal{M})}\leq C_{\text{obs}}\left\|y(x,0) \right\|_{\mathbb{L}^{2}_{h}(\mathcal{M})}.
\end{equation*}
This last inequality, together with \eqref{eq:y:q}, allows us to conclude that
\begin{equation}
    \left\| y(x,T)\right\|_{\mathbb{L}^{2}_{h}(\mathcal{M})}\leq C_{\text{obs}}\sqrt{\phi(h)}\left\|y(x,0) \right\|_{\mathbb{L}^{2}_{h}(\mathcal{M})}.
\end{equation}
Hence, the $\phi$-controllability holds for $\phi(h):=e^{-\frac{C}{h^{\min\{\mu/4,1\}}}}$ and the proof is finished.

\begin{remark} 
The previous result can be stated for a function $h\mapsto \phi(h)$ such that 
\begin{equation*}
    \liminf_{h\rightarrow 0}\phi(h)e^{\frac{C}{h^{\min\{\mu/4,1\}}}}>0.
\end{equation*}
Indeed, from above we note that there exists $h^{\ast}>0$ such that $e^{-\frac{-C}{h^{\min\{\mu/4,1\}}}}\leq \phi(h)$ for all $h\leq h^{\ast}$. Then, for $h\leq \min\{h_{0},h_{1},h^{\ast}\}$ and $\Delta t\leq T^{-2}h^{4/\mu}$, Theorem \ref{observ-SZ} holds for such $\phi$ and we can follow the same steps as before to achieve the $\phi$-controllability for the system \eqref{semi-discrete:model:ops}.
\end{remark}

\section{Comments}
Since the null controllability holds in dimension $d\geq 2$, via Carleman estimate \cite{maniar2017null}, it seems promising to extend the results presented here in that direction. We have presented the $\phi$-controllability for the fully discrete approximation in the 1D case. Although, it is worth mentioning that the techniques and results given in Section \ref{sec:preliminaries} still hold in arbitrary dimension as soon as we restrict to finite difference schemes on Cartesian grids. In this direction, the first task would be to state for arbitrary spatial dimension Theorem \ref{teo:space:estimate} and Theorem \ref{prop:time:estimate}. These results stand for useful tools to obtain fully-discrete Carleman estimate where the Carleman weight function is a variant the weight function from \cite{fursikov1996controllability}. See for instance the results obtained for a four-order semi-discrete parabolic equation in \cite{CLTP-2022}. 

The results presented in Section \ref{estimaciones} are of independent interest in view of its potential applications on problems related to discrete Carleman estimates. For instance, it could be used to answer the challenge proposed in \cite{GC-HS-2021}, where several open problems related to fully discrete system were proposed (relaxed insensitizing control and hierarchic control). Additionally, is the appropriate input to be able to study controllability results for other types of models related to dynamic boundary conditions, such as the boundary control studied in \cite{maniar2022boundary}, null controllability with drift terms, optimal control problems \cite{MR3009728}.

Another unexplored research topic is the fully discrete inverse source problem in anisotropic parabolic equation with dynamic boundary conditions, which continuous setting was studied recently in \cite{L:inverse}. The proof of this result is based on a suitable Carleman estimate, then it could be interesting to obtain a fully discrete version of this result.

Finally, since the controllability of the fully-discrete system \eqref{semi-discrete:model:ops} has been obtained by using the penalized HUM method, would be an interesting problem study numerical algorithms in order to illustrate the theoretical results proved in this article, following the seminal paper \cite{boyer2013hum}.

\appendix
\section{Estimates for the Carleman estimate}\label{sec:proof:carleman}
In this section, we proof the technical results used in the development of the discrete Carleman estimate. We consider $\tau h(\delta T^{2})^{-1}\leq 1$ in the following Lemmas in order to ensure that every result from Section \ref{estimaciones} holds. Recall that our Carleman weight function is of the form $r(t,x):=e^{s(t)\varphi(x)}$, for $s\geq 1$, with $\varphi(x)=e^{\lambda\psi(x)}-e^{\lambda K}<0$, where $\psi\in C^{k}$ for $k$ sufficiently large and $\lambda\geq 1$, $K>\left\| \psi\right\|_{\infty}$ and $s(t):=\tau\theta(t)$ for $\theta(t):=\left( (t+\delta T)(T+\delta T-t)\right)^{-1}$. We denote $\rho:=r^{-1}$.
\par The estimates for $I_{11}$ and $I_{12}$ are similar, after a shift with respect to the discrete time variable, to \cite[Lemma 5.3]{LOP-2020} and \cite[Lemma 5.4]{LOP-2020} since in contrast to Lemma 2.1 and Lemma 2.2 from \cite{GC-HS-2021}, in \cite{LOP-2020} is considered not null boundary condition. For sake of completeness we give a sketch of the proof for that estimates. 
\subsection{Estimate of $I_{11}$ }
\begin{lemma}\label{I11}
For $\Delta t \tau (T^3 \delta^2)^{-1}\leq 1$ and $\tau h(\delta T^{2})^{-1}\leq 1$, we have
\begin{equation*}
\begin{split}
    I_{11}\geq&-\tau\lambda^{2}\int_{\mathcal{M}^{\ast}\times\mathcal{N}^{\ast}} \theta(\partial_{x}\psi)^{2}\phi|D_{h} z|^{2}-\int_{\mathcal{M}^*\times \mathcal{N}^{\ast}} s\lambda\partial_{x}^{2}\psi\phi|D_{h} z|^{2}\\
    &- \tau\lambda\int_{\partial\mathcal{M}\times\mathcal{N}^{\ast}} \theta\partial_{x} \psi\phi t_{r}(|D_{h} z|^2)\,n- X_{11}- Y_{11},
    \end{split}
\end{equation*}
where $X_{11}$ and $Y_{11}$ are given by\\
$\displaystyle X_{11}:=\int_{\mathcal{M}^*\times \mathcal{N}^{\ast}}s \mathcal{O}_\lambda ((sh)^2)|D_{h} z|^{2}$
and $\displaystyle Y_{11}:=\int_{\partial\mathcal{M}\times\mathcal{N}^*} (s\mathcal{O}_\lambda (sh)^2)\,t_{r}(|D_{h}z|^2)n.$
\end{lemma}
\begin{proof}
We set $\alpha_{11}=r^2 A_{h}^2 \rho D_{h} A_{h} \rho$. Then, $I_{11}$ can be written as $\displaystyle I_{11}=2\int_{\mathcal{M}\times \mathcal{N}} \taut^-(\alpha_{11} D_{h}^2 z D_{h}A_{h} z)$.
Thanks to \eqref{eq:tau:menos} we can shift the discrete 
time variable to obtain $\displaystyle
    I_{11}=2\int_{\mathcal{M}\times \mathcal{N}^{\ast}} \alpha_{11} D_{h}^2 z D_{h} A_{h} z$.
Using Lemma \ref{lem:product:rule} and a discrete integration by parts with respect to the difference operator $D_{h}$, it follows that 
\begin{align}
    \label{estimate:I11:01}
    \begin{split} 
    I_{11}=&-\int_{\mathcal{M}\times \mathcal{N}^*} D_{h}(\alpha_{11})|D_{h} z|^2 +\int_{\partial\mathcal{M}\times\mathcal{N}^*} \alpha_{11}t_{r} (|D_{h} z|^2)\,n.
    \end{split}
\end{align}
Now, by Proposition \ref{pro:space:estimate} and Corollary \ref{3.8} we have that $
    \alpha_{11}=-s\lambda\varphi\partial_{x}\psi+s\mathcal{O}_{\lambda  }((sh)^{2})$ and $D_{h} \alpha_{11}=-s\lambda^2 (\partial_{x} \psi)^2 \phi - \lambda s \partial_{x}^{2} \psi \phi +s\mathcal{O}_\lambda ((sh)^2)$. 
Substituting these estimates in \eqref{estimate:I11:01} complete the proof.
\end{proof}

\subsection{Estimate of $I_{12}$}
\begin{lemma}\label{I12}
For $\Delta t \tau (T^3 \delta^2)^{-1}\leq 1$ and $\tau h(\delta T^{2})^{-1}\leq 1$, we have
    \begin{equation*}
    \begin{split}
        I_{12}
         \geq &2\tau\lambda^{2}\int_{ \mathcal{M}^{\ast}\times\mathcal{N}^{\ast}}\theta(\partial_{x}\psi)^{2}\phi|D_{h}z|^{2}+\tau\lambda\int_{ \mathcal{M}^{\ast}\times\mathcal{N}^{\ast}}\theta\phi\mathcal{O}(1)|D_{h}z|^{2}-X_{12}-Y_{12},
    \end{split}
\end{equation*}
where $\displaystyle 
X_{12}:=\int_{ \mathcal{M}\times\mathcal{N}^{\ast}}s\mathcal{O}_{\lambda}(1)|z|^{2}+\int_{ \mathcal{M}^{\ast}\times\mathcal{N}^{\ast}}s\mathcal{O}_{\lambda}(h^{2}+(sh)^{2})|D_{h}z|^{2}$
and
\begin{equation*}
Y_{12}:=-\int_{\partial \mathcal{M}\times\mathcal{N}^{\ast}}s\mathcal{O}_{\lambda}(1)|z|^{2}+\int_{\partial \mathcal{M}\times\mathcal{N}^{\ast}}s^{2} \mathcal{O}_{\lambda}(1)|z|^{2}+\int_{\partial \mathcal{M}\times\mathcal{N}^{\ast}}\mathcal{O}_{\lambda}(1)t_{r}(|D_{h}z|^{2}).
\end{equation*}
\end{lemma}
\begin{proof}
Let us set $\alpha_{12}:=2r\,A_{h}^{2}\rho\,\partial_{x}^{2}\phi$. Similar to the previous estimate we use \eqref{eq:tau:menos} to shift the discrete time variable, thus we have to estimate $\displaystyle I_{12} =-\int_{\mathcal{M}\times\mathcal{N}^{\ast}}s\,\alpha_{12}z\,D_{h}^{2}z.$
According to Proposition \ref{pro:integral:space} we have 
\begin{equation}\label{estimate:I12:01}
\begin{split}
I_{12}=&\int_{\mathcal{M}^{\ast}\times\mathcal{N}^{\ast}} s\,D_{h}( \alpha_{12}\,z)D_{h}z -\int_{\partial \mathcal{M}\times\mathcal{N}^{\ast}}s\,\alpha_{12}\,z\,t_{r}(D_{h}z)\,n=:I_{12}^{(a)}-I_{12}^{(b)}.
\end{split}
\end{equation}
Let us focus on $I_{12}^{(a)}$. From Lemma \ref{lem:product:rule} it follows that
\begin{equation}
\begin{split}
    I_{12}^{(a)}=&\int_{\mathcal{M} ^{\ast}\times \mathcal{N}^{\ast}} s\,D_{h}(\alpha_{12})\,A_{h}z\,D_{h}z+\int_{ \mathcal{M}^{\ast}\times\mathcal{N}^{\ast}}s\,A_{h}(\alpha_{12})|D_{h}z|^{2}=:I_{12}^{(a_{1})}+I_{12}^{(a_{2})}.
    \end{split}
\end{equation}
We note that
\begin{equation}\label{estimate:I12:02}
    \begin{split}
        I_{12}^{(a_{2})}
        =&\int_{\mathcal{M}^{\ast}\times\mathcal{N}^{\ast}}2s\lambda^{2}(\partial_{x}\psi)^{2}\phi|D_{h}z|^{2} +\int_{\mathcal{M}^{\ast}\times\mathcal{N}^{\ast}}s\lambda\phi\partial_{x}^{2}\psi|D_{h}z|^{2}+\int_{\mathcal{M}^{\ast}\times\mathcal{N}^{\ast}}s\mathcal{O}_{\lambda}(h^{2}+(sh)^{2})|D_{h}z|^{2},
    \end{split}
\end{equation}
being consequence of Proposition \ref{pro:space:estimate} and Lemma \ref{lem:space:estimate}. Additionally, to estimate $I_{12}^{(a_{1})}$ we note that by Lemma \ref{lem:product:rule} and a spatial discrete integration by parts yield
\begin{equation*}
\begin{split}
    I_{12}^{(a_{1})}=&-\frac{1}{2}\int_{\mathcal{M}\times\mathcal{N}^{\ast}}sD_{h}^{2}\alpha_{12}\,|z|^{2}+\frac{1}{2}\int_{\partial \mathcal{M}\times\mathcal{N}^{\ast}}s\ t_{r}(D_{h}\alpha_{12})|z|^{2}\,n.
    \end{split}
\end{equation*}
Now, by virtue of Lemma \ref{lem:space:estimate} and Proposition \ref{pro:space:estimate}, for $I_{12}^{(a_{1})}$ we get
 \begin{equation}\label{estimate:I12:03}
     I_{12}^{(a_{1})}=-\int_{\mathcal{M}\times\mathcal{N}^{\ast}}s\mathcal{O}_{\lambda}(1)|z|^{2}+\int_{\partial \mathcal{M}\times\mathcal{N}^{\ast}}s\mathcal{O}_{\lambda}(1)|z|^{2}.
 \end{equation}
 Similarly, using Young's inequality we have
 \begin{equation}\label{estimate:I12:04}
     |I_{12}^{(b)}|\leq \int_{\partial \mathcal{M}\times\mathcal{N}^{\ast}}s^{2} |\mathcal{O}_{\lambda}(1)||z|^{2}+\int_{\partial \mathcal{M}\times\mathcal{N}^{\ast}}|\mathcal{O}_{\lambda}(1)|t_{r}(|D_{h}z|^{2}).
 \end{equation}
 Therefore, combining \eqref{estimate:I12:02}, \eqref{estimate:I12:03} and \eqref{estimate:I12:04},  $I_{12}$ can be estimated as
\begin{equation*}
    \begin{split}
        I_{12}
         \geq &\int_{ \mathcal{M}^{\ast}\times\mathcal{N}^{\ast}}2s\lambda^{2}(\partial_{x}\psi)^{2}\phi|D_{h}z|^{2}+\int_{ \mathcal{M}^{\ast}\times\mathcal{N}^{\ast}}s\lambda\phi\mathcal{O}(1)|D_{h}z|^{2}-X_{12}-Y_{12},
    \end{split}
\end{equation*}
and the Lemma follows.
\end{proof}

\subsection{Estimate of $I_{13}$ and $I_{23}$}
These estimates are similar to Lemma 2.8 from \cite{GC-HS-2021}. The only difference is that we need to deal with space boundary terms which have been considered in \cite{LOP-2020}. We follow a similar strategy developed in \cite[Lemma 2.8]{GC-HS-2021}, however we point out that following the steps described on the proof of Lemma 2.8 is not possible to get equation (190) in the page 64, since the integral with the term $|z|^{2}$ must be shifted in time to the right and not to the left after the application of the equation (155) in that work. For this reason, some new terms will appear and additional steps are needed to obtain the desired result, although it can be controlled, thus the main results from \cite{GC-HS-2021} still hold. 
\begin{lemma}\label{I_{13}}
For $\tau^{2} \Delta t(\delta^{4}T^{6})^{-1}\leq \epsilon_{1}(\lambda)$ and $\tau h(\delta T^{2})^{-1}\leq \epsilon_{1}(\lambda)$, there exist constant $c_{\lambda}, c'_{\lambda}>0$ uniform with respect to $\Delta t$ and $h$ such that
\begin{equation*}
    \begin{split}
        I_{13}+I_{23}\geq&-X_{13}-X_{23}-Y_{13}-Y_{23}-C_{\lambda}\int_{\mathcal{M}^{\ast}\times\partial\mathcal{N}^{+}}|D_{h}\taut^{+}z|^{2}-\int_{\mathcal{M}^{\ast}\times \partial \mathcal{N}}\mathcal{O}_{\lambda}(1)\taut^{+}((sh)^{2})|D_{h}\taut^{+}z|^{2})\\
        &-\int_{\mathcal{M}\times\partial\mathcal{N}}\mathcal{O}_{\lambda}(1)\taut^{+}(s^{2})|\taut^{+}z|^{2}-C\int_{\partial\mathcal{M}\times\mathcal{N}^{\ast}}s\,t_{r}(|D_{h}z|^{2}),
    \end{split}
\end{equation*}
where $\displaystyle X_{13}+X_{23}:=\int_{\mathcal{M}\times\mathcal{N}^{\ast}}\mu\,|z|^{2}+\int_{\mathcal{M}^{\ast}\times\mathcal{N}^{\ast}}\nu|D_{h}z|^{2}-C_{\lambda}\Delta t \int_{\mathcal{M}^{\ast}\times\mathcal{N}}|D_{ht}^{2}z|^{2}+\int_{\mathcal{M}\times\mathcal{N}}\taut^{-}(\gamma)|D_{t}z|^{2}$, $\displaystyle Y_{13}+Y_{23}:=\int_{\partial\mathcal{M}\times\mathcal{N}^{\ast}}\eta t_{r}(|D_{h}z|^{2})\,n+\int_{\partial\mathcal{M}\times\mathcal{N}}\taut^{-}(\beta)|D_{t}z|^{2}$
and
\begin{equation*}
    \begin{split}
        \nu:=&\left[ T\theta(sh)^{2}+s(sh)^{2}+\left( \frac{\tau \Delta t}{\delta^{4}T^{4}}\right)\left( \frac{\tau h}{\delta T^{2}}\right)+\left( \frac{\tau^{2}\Delta t}{\delta^{4}T^{6}}\right)+\left( \frac{\tau \Delta t}{\delta^{3}T^{4}}\right)\left( \frac{\tau h}{\delta T^{2}}\right)^{3}\right]\mathcal{O}_{\lambda}(1),\\
        \mu:=&\left[ Ts^{2}\theta+\left( \frac{\tau^{2}\Delta t}{\delta^{4}T^{6}}\right)+\left(\frac{\tau \Delta t}{\delta T^{4}}\right)\left( \frac{\tau h}{\delta t^{2}}\right)^{3}\right]\mathcal{O}_{\lambda}(1),\\
        \gamma:=&\left[ s^{-1}(sh)^{2}+\Delta t s^{2}\right]\mathcal{O}_{\lambda}(1),\\
        \beta:=&Cs^{-1}+\left[ s^{-1}(sh)^{2}+hs^{-1}(sh)^{2}+s^{-1}(sh)^{4}\right]\mathcal{O}_{\lambda}(1), \\
        \eta:=&\left[s(sh)^{2}+(sh)^{4}\right]\mathcal{O}_{\lambda}(1).
    \end{split}
\end{equation*}
\end{lemma}
\begin{proof}
For $I_{13}$, let us set $\alpha_{13}:=rA_{h}^{2}\rho$. Then, we have that $\displaystyle 
    I_{13}:=\int_{\mathcal{M}\times\mathcal{N}}\taut^{-}(\alpha_{13}D_{h}^{2}z)\,D_{t}z$. Followed by a discrete integration by parts with respect to the difference operator $D_{h}$, $I_{13}$ can be written as
\begin{equation*}
\begin{split}
I_{13}=&-\int_{\mathcal{M}^{\ast}\times\mathcal{N}}D_{h}\taut^{-}(z)\,D_{h}(\taut^{-}(\alpha_{13})\,D_{t}z)+\int_{\partial \mathcal{M}\times\mathcal{N}}\taut^{-}(\alpha_{13})D_{t}z\,t_{r}\left( D_{h}\taut^{-}z\right)\,n=:I_{13}^{(a)}+I_{13}^{(b)}.
\end{split}
\end{equation*}
We note that the product rule, Lemma \ref{lem:product:rule},  for the difference operator $D_{h}$ yields for $I_{13}^{(a)}$
\begin{equation*}
\begin{split}
    I_{13}^{(a)}=&-\int_{\mathcal{M}^{\ast}\times\mathcal{N}}D_{h}\taut^{-}(\alpha_{13})\,A_{h}D_{t}z\,D_{h}\taut^{-}z-\int_{\mathcal{M}^{\ast}\times\mathcal{N}}A_{h}\taut^{-}(\alpha_{13})\,D_{xt}^{2}z\,D_{h}\taut^{-}(z)=:I_{13}^{(a_{1})}+I_{13}^{(a_{2})}.
    \end{split}
\end{equation*}
Now, by using the identity \eqref{eq:identity:square}, for $I_{13}^{(a_{2})}$ we get
\begin{equation*}
\begin{split}
    I_{13}^{(a_{2})}=&-\frac{1}{2}\int_{\mathcal{M}^{\ast}\times\mathcal{N}}A_{h}\taut^{-}(\alpha_{13})\,D_{t}(|D_{h}z|^{2})+\frac{\Delta t}{2}\int_{\mathcal{M}^{\ast}\times\mathcal{N}}A_{h}\taut^{-}(\alpha_{13})\,|D_{ht}^{2}z|^{2}.
    \end{split}
\end{equation*}
Using the integration by parts with respect to the time difference operator $D_{t}$ given by \eqref{eq:integral:time:primal}, on the first integral of $I_{13}^{(a_{2})}$, we obtain
\begin{equation*}
\begin{split}
    I_{13}^{(a_{2})}=&\frac{1}{2}\int_{\mathcal{M}^{\ast}\times\mathcal{N}}A_{h}D_{t}\alpha_{13}\,|D_{h}\taut^{+}z|^{2}-\int_{\mathcal{M}^{\ast}\times\partial\mathcal{N}}A_{h}\taut^{+}\alpha_{13}|D_{h}\taut^{+}z|^{2}n+\frac{\Delta t}{2}\int_{\mathcal{M}^{\ast}\times\mathcal{N}}A_{h}\taut^{-}(\alpha_{13})\,|D_{ht}^{2}z|^{2}.
    \end{split}
\end{equation*}
By virtue of Lemma \ref{lem:space:estimate} we note that $A_{h}\taut^{\pm}\alpha_{13}= C_{\lambda}>0$ provided $\tau h(\delta T^{2})^{-1}\leq \varepsilon_{1}(\lambda)$ small enough, we thus have the following lower bound for $I_{13}^{(a_{2})}$
\begin{equation*}
\begin{split}
    I_{13}^{(a_{2})}\geq&\frac{1}{2}\int_{\mathcal{M}^{\ast}\times\mathcal{N}}A_{h}D_{t}\alpha_{13}\,|D_{h}\taut^{+}z|^{2}-C_{\lambda}\int_{\mathcal{M}^{\ast}\times\partial\mathcal{N}^{+}}|D_{h}\taut^{+}z|^{2}+C_{\lambda}\Delta t \int_{\mathcal{M}^{\ast}\times\mathcal{N}}|D_{ht}^{2}z|^{2},
    \end{split}
\end{equation*}
where we have dropped the left time boundary term since it is positive. 
Moreover, by using the inequality $|D_{h}\taut^{+}z|^{2}\leq C|D_{h}\taut^{-}z|^{2}+C(\Delta t)^{2}|D_{ht}^{2}z|^{2}$, we get
\begin{equation*}
\begin{split}
    I_{13}^{(a_{2})}\geq&-C\int_{\mathcal{M}^{\ast}\times\mathcal{N}}|A_{h}D_{t}\alpha_{13}|\,|D_{h}\taut^{-}z|^{2}-C(\Delta t)^{2}\int_{\mathcal{M}^{\ast}\times\mathcal{N}}|A_{h}D_{t}\alpha_{13}|\,|D_{ht}^{2}z|^{2}\\
    &-C_{\lambda}\int_{\mathcal{M}^{\ast}\times\partial\mathcal{N}^{+}}|D_{h}\taut^{+}z|^{2}+C_{\lambda}\Delta t \int_{\mathcal{M}^{\ast}\times\mathcal{N}}|D_{ht}^{2}z|^{2}.
    \end{split}
\end{equation*}
We note that provided $\frac{\Delta t \tau}{T^{3}\delta^{2}}\leq \frac{1}{2}$, from Theorem \ref{prop:time:estimate},  we have that
$A_{h}D_{t}\alpha_{13}=T\taut^{-}(\theta(sh)^{2})\mathcal{O}_{\lambda}(1)+\left( \frac{\tau\Delta t}{\delta^{3}T^{4}}\right)\left( \frac{\tau h}{\delta T^{2}}\right)\mathcal{O}_{\lambda}(1)$. Then, we deduce
\begin{equation*}
\begin{split}
    I_{13}^{(a_{2})}\geq&-\int_{\mathcal{M}^{\ast}\times\mathcal{N}^{\ast}}\left( T\theta(sh)^{2}\mathcal{O}_{\lambda}(1)+\left( \frac{\tau\Delta t}{\delta^{3}T^{4}}\right)\left( \frac{\tau h}{\delta T^{2}}\right)\mathcal{O}_{\lambda}(1)\right)\,|D_{h}z|^{2}-C_{\lambda}\int_{\mathcal{M}^{\ast}\times\partial\mathcal{N}^{+}}|D_{h}\taut^{+}z|^{2}\\
    &-C_{\lambda}\Delta t\int_{\mathcal{M}^{\ast}\times\mathcal{N}}\left( \frac{\tau\Delta t}{\delta^{3}T^{4}}\right)\left( \frac{\tau h}{\delta T^{2}}\right)\,|D_{ht}^{2}z|^{2}+C_{\lambda}\Delta t \int_{\mathcal{M}^{\ast}\times\mathcal{N}}|D_{ht}^{2}z|^{2},
    \end{split}
\end{equation*}
 where we have shifted the discrete time  variable on the first integral above. Hence, taking $\epsilon_{1}(\lambda)$ small enough, for $I_{13}^{(a_{2})}$ we have that
\begin{equation}\label{eq:I13:a2}
\begin{split}
    I_{13}^{(a_{2})}\geq&-\int_{\mathcal{M}^{\ast}\times\mathcal{N}^{\ast}}\left( T\theta(sh)^{2}\mathcal{O}_{\lambda}(1)+\left( \frac{\tau\Delta t}{\delta^{3}T^{4}}\right)\left( \frac{\tau h}{\delta T^{2}}\right)\mathcal{O}_{\lambda}(1)\right)\,|D_{h}z|^{2}\\
    &-C_{\lambda}\int_{\mathcal{M}^{\ast}\times\partial\mathcal{N}^{+}}|D_{h}\taut^{+}z|^{2}+C_{\lambda}\Delta t \int_{\mathcal{M}^{\ast}\times\mathcal{N}}|D_{ht}^{2}z|^{2}.
    \end{split}
\end{equation}

On the other hand, by virtue of Proposition \ref{pro:space:estimate} we note that $D_{h}\alpha_{13}=(sh)^{2}\mathcal{O}_{\lambda}(1)$. Then, this estimate and the Young's inequality for $I_{13}^{(a_{1})}$ yield 
\begin{equation*}
    \begin{split}
      |I_{13}^{(a_{1})}|\leq&C_{\lambda} \int_{\mathcal{M}^{\ast}\times\mathcal{N}}\taut^{-}(s^{-1}(sh)^{2}|D_{t}A_{h}z|^{2}+C_{\lambda}\int_{\mathcal{M}^{\ast}\times\mathcal{N}^{\ast}}s(sh)^{2}|D_{h}z|^{2},
    \end{split}
\end{equation*}
where in the last integral above we have used equation \eqref{eq:tau:menos} to shift the discrete time variable. Now, using that $|D_{t}A_{h}z|^{2}\leq A_{h}(|D_{t}z|^{2})$ and an integration by parts with respect to the average operator $A_{h}$ on the first integral of the above expression, we obtain
\begin{equation}\label{eq:I13:a1}
    \begin{split}
      |I_{13}^{(a_{1})}|\leq
        &C_{\lambda} \left(\int_{\mathcal{M}\times\mathcal{N}}\taut^{-}(s^{-1}(sh)^{2})|D_{t}z|^{2}+h\int_{\partial\mathcal{M}\times\mathcal{N}}\taut^{-}(s^{-1}(sh)^{2})|D_{t}z|^{2}+\int_{\mathcal{M}^{\ast}\times\mathcal{N}^{\ast}}s(sh)^{2}|D_{h}z|^{2}\right).
    \end{split}
\end{equation}
Thus, collecting \eqref{eq:I13:a2} and \eqref{eq:I13:a1} we obtain that $I_{13}^{(a)}$ can be bounded as 
$$
\displaystyle I_{13}^{(a)}\geq -W_{13}-C_{\lambda}\int_{\mathcal{M}\times\partial\mathcal{N}^{+}}|D_{h}\taut^{+}z|^{2}$$
where
\begin{equation}
    \begin{split}
          W_{13}=&\int_{\mathcal{M}^{\ast}\times\mathcal{N}^{\ast}}\left( T\theta(sh)^{2}\mathcal{O}_{\lambda}(1)+\left( \frac{\tau\Delta t}{\delta^{3}T^{4}}\right)\left( \frac{\tau h}{\delta T^{2}}\right)\mathcal{O}_{\lambda}(1)\right)\,|D_{h}z|^{2}\\
        &-C_{\lambda}\Delta t \int_{\mathcal{M}^{\ast}\times\mathcal{N}}|D_{ht}^{2}z|^{2}+C_{\lambda}\int_{\mathcal{M}^{\ast}\times\mathcal{N}^{\ast}}s(sh)^{2}|D_{h}z|^{2}\\
        &+C_{\lambda} \int_{\mathcal{M}\times\mathcal{N}}\taut^{-}(s^{-1}(sh)^{2})|D_{t}z|^{2}-C_{\lambda}h\int_{\partial\mathcal{M}\times\mathcal{N}}\taut^{-}(s^{-1}(sh)^{2})|D_{t}z|^{2}. 
    \end{split}
\end{equation}
\par Finally, by using the Young's inequality on $I_{13}^{(b)}$, and the estimate $\taut^{-}(\alpha_{13})=1+C_{\lambda}\taut^{-}((sh)^{2})$ due to Proposition \ref{pro:space:estimate}, we get
\begin{equation}\label{eq:I13b}
    \begin{split}
    |I_{13}^{(b)}|\leq &\int_{\partial\mathcal{M}\times\mathcal{N}}(C\taut^{-}(s^{-1})+C_{\lambda}\taut^{-}(s^{-1}(sh)^{4})|D_{t}z|^{2}\\
    &+\int_{\partial\mathcal{M}\times\mathcal{N}}(C\taut^{-}(s)+C_{\lambda}\taut^{-}((sh)^{4})t_{r}(|D_{h}\taut^{-}z|^{2})\\
    &=:Y_{13}.
    \end{split}
\end{equation}
Therefore, combining the lower bound for $I_{13}^{(a)}$ and \eqref{eq:I13b} it follows that $I_{13}$ can be estimated as
\begin{equation}\label{eq:I13}
\begin{split}
    I_{13}\geq&X_{13}-Y_{13}-C_{\lambda}\int_{\mathcal{M}^{\ast}\times\partial\mathcal{N}^{+}}|\taut^{+}D_{h}z|^{2}.
    \end{split}
\end{equation}
\par Now we focus on $I_{23}$. Let us set $\alpha_{23}:=rD_{h}^{2}\rho$, then $I_{23}$ can be written as
$$\displaystyle
I_{23}=\int_{\mathcal{M}\times\mathcal{N}}\taut^{-}(\alpha_{23}A_{h}^{2}z)\,D_{t}z.$$

Using the identity $A_{h}^{2}z=z+\frac{h^{2}}{4}D_{h}^{2}z$, $I_{23}$ can be rewritten as
\begin{equation*}
\begin{split}
    I_{23}=&\int_{\mathcal{M}\times\mathcal{N}}\taut^{-}(\alpha_{23})\taut^{-}(z)\,D_{t}z+\frac{h^{2}}{4}\int_{\mathcal{M}\times\mathcal{N}}\taut^{-}(\alpha_{23}\,D_{h}^{2}z)\,D_{t}z=:I_{23}^{(a)}+I_{23}^{(b)}.
    \end{split}
    \end{equation*}
\par We note that the identity $\taut^{-}z\,D_{t}z=\frac{1}{2}D_{t}(|z|^{2})-\frac{\Delta t}{2}|D_{t}z|^{2}$, for $I_{23}^{(a)}$, yields 
\begin{equation*}
\begin{split}
    I_{23}^{(a)}=&\frac{1}{2}\int_{\mathcal{M}\times\mathcal{N}}\taut^{-}(\alpha_{23})\,D_{t}(|z|^{2})-\frac{\Delta t}{2}\int_{\mathcal{M}\times\mathcal{N}}\taut^{-}(\alpha_{23})\,|D_{t}z|^{2}.
    \end{split}
    \end{equation*}
A discrete integration by parts with respect to the time difference operator $D_{t}$ enables us to write
\begin{equation*}
\begin{split}
    I_{23}^{(a)}=&-\int_{\mathcal{M}\times\mathcal{N}}D_{t}(\alpha_{23})\,\taut^{+}(|z|^{2})+\frac{1}{2}\int_{\mathcal{M}\times\partial\mathcal{N}}\taut^{+}(\alpha_{23}|z|^{2})\,n-\frac{\Delta t}{2}\int_{\mathcal{M}\times\mathcal{N}}\taut^{-}(\alpha_{23})\,|D_{t}z|^{2}.
\end{split}
\end{equation*}
Now, using the following inequality $|\taut^{+}z|^{2}\leq C|\taut^{-}z|^{2}+C(\Delta t)^{2}|D_{t}z|^{2}$, Theorem \ref{prop:time:estimate} and Proposition \ref{pro:space:estimate}, for $I_{23}^{(a)}$ we obtain
\begin{equation}\label{eq:I23:a}
\begin{split}
    I_{23}^{(a)}\geq &-\int_{\mathcal{M}\times\mathcal{N}}\taut^{-}\mu\,|\taut^{-}z|^{2}-(\Delta  t)^{2}\int_{\mathcal{M}\times  \mathcal{N}}\taut^{-}(s^{2})(|D_{t}z|^{2})\\
    &-\int_{\mathcal{M}\times\partial\mathcal{N}}\mathcal{O}_{\lambda}(1)\taut^{+}(s^{2}|z|^{2})-\Delta t\int_{\mathcal{M}\times\mathcal{N}}\taut^{-}(s^{2})\mathcal{O}_{\lambda}(1)\,|D_{t}z|^{2},
    \end{split}
\end{equation}
where $\mu:=\left[Ts^{2}\theta\mathcal{O}_{\lambda}(1)+\left(\frac{\tau^{2}\Delta t}{\delta^{4}T^{6}}\right)\mathcal{O}_{\lambda}(1)+\left(\frac{\tau \Delta t}{\delta^{3}T^{4}}\right)\left(\frac{\tau h}{\delta T^{2}}\right)^{3}\right]\mathcal{O}_{\lambda}(1)$.
\par Let us now estimate $I_{23}^{(b)}$. Using a discrete integration by parts for the operator $D_{h}$ we obtain
\begin{equation*}
    I_{23}^{(b)}=-\frac{h^{2}}{4}\int_{\mathcal{M}^{\ast}\times\mathcal{N}}D_{h}(\taut^{-}\alpha_{23}D_{t}z)\taut^{-}(D_{h}z)+\frac{h^{2}}{4}\int_{\partial\mathcal{M}\times\mathcal{N}}\taut^{-}(\alpha_{23})D_{t}z\,t_{r}(\taut^{-}D_{h}z)\,n.
\end{equation*}
By virtue of the product rule for the difference operator $D_{h}$, see Lemma \ref{lem:product:rule}, we write
\begin{equation*}
\begin{split}
    I_{23}^{(b)}=&-\frac{h^{2}}{4}\int_{\mathcal{M}^{\ast}\times\mathcal{N}}\taut^{-}(D_{h}\alpha_{23})A_{h}D_{t}z\,\taut^{-}(D_{h}z)-\frac{h^{2}}{4}\int_{\mathcal{M}^{\ast}\times\mathcal{N}}\taut^{-}(A_{h}\alpha_{23})D^{2}_{ht}z\,\taut^{-}(D_{h}z)\\
    &+\frac{h^{2}}{4}\int_{\partial\mathcal{M}\times\mathcal{N}}\taut^{-}(\alpha_{23})D_{t}z\,t_{r}(\taut^{-}D_{h}z)\,n\\
    =:&I_{23}^{(b_{1})}+I_{23}^{(b_{2})}+I_{23}^{(b_{3})}.
    \end{split}
\end{equation*}
Thus, our goal is to estimate $I_{23}^{(b_{1})}, I_{23}^{(b_{2})}$ and $I_{23}^{(b_{3})}$. We note that $D_{xt}^{2}z\,\taut^{-}(D_{h}z)=\frac{1}{2}D_{t}(|D_{h}z|^{2})-\frac{\Delta t}{2}|D^{2}_{tx}z|^{2}$, thanks to the identity $\taut^{-}z\,D_{t}z=\frac{1}{2}D_{t}(|z|^{2})-\frac{\Delta t}{2}|D_{t}z|^{2}$. Then, $I_{23}^{(b_{2})}$ can be rewritten as
\begin{equation*}
    I_{23}^{(b_{2})}=-\frac{h^{2}}{8}\int_{\mathcal{M}^{\ast}\times\mathcal{N}}\taut^{-}(A_{h}\alpha_{23})D_{t}(|D_{h}z|^{2})+\frac{\Delta t\,h^{2}}{8}\int_{\mathcal{M}\times\mathcal{N}}\taut^{-}(A_{h}\alpha_{23}) 
    |D^{2}_{ht}z|^{2}.
\end{equation*}
Applying a discrete integration by parts for the discrete difference time operator $D_{t}$ on the first integral above, given by \eqref{eq:integral:time:primal}, leads to
\begin{equation}\label{eq:I23:b2:01}
\begin{split}
    I_{23}^{(b_{2})}=&\frac{h^{2}}{8}\int_{\mathcal{M}^{\ast}\times\mathcal{N}}A_{h}D_{t}\alpha_{23}|D_{h}\taut^{+}z|^{2}-\frac{h^{2}}{8}\int_{\mathcal{M}^{\ast}\times\partial\mathcal{N}}\taut^{+}(A_{h}\alpha_{23}|D_{h}z|^{2})\,n\\
    &+\frac{\Delta th^{2}}{8}\int_{\mathcal{M}\times\mathcal{N}}\taut^{-}(A_{h}\alpha_{23})\,|D^{2}_{ht}z|^{2}.
    \end{split}
\end{equation}
We note that using the estimates from Theorem \ref{prop:time:estimate} and Theorem \ref{teo:space:estimate}, for $\tau \Delta t(T^{3}\delta^{2})\leq 1$, it follows that
\begin{equation*}
    \begin{split}
        A_{h}D_{t}\alpha_{23}&=T\taut^{-}(s^{2}\theta)\mathcal{O}_{\lambda}(1)+\frac{\tau^{2}\Delta t}{\delta T^{6}}\mathcal{O}_{\lambda}(1)+\left( \frac{\tau \Delta t}{\delta^{3}T^{4}}\right)\left( \frac{\tau h}{\delta T^{2}}\right)^{3}\mathcal{O}_{\lambda}(1),\\
        A_{h}\alpha_{23}&=s^{2}\mathcal{O}_{\lambda}(1).
    \end{split}
\end{equation*}
Combining the above estimates with \eqref{eq:I23:b2:01}, and using that $|\taut^{+}D_{h}z|^{2}\leq C|\taut^{-}D_{h}z|^{2}+C(\Delta t)^{2}|D_{ht}^{2}z|^{2}$, we have that $I_{23}^{(b_{2})}$ can be estimated as 
\begin{equation}\label{eq:I23:b2}
\begin{split}
    |I_{23}^{(b_{2})}|\leq&\int_{\mathcal{M}^{\ast}\times\mathcal{N}}T\taut^{-}(\theta (sh)^{2})\taut^{-}(|D_{h}z|^{2})+\int_{\mathcal{M}^{\ast}\times\mathcal{N}}\left(\frac{\tau^{2}\Delta t}{\delta T^{6}}\mathcal{O}_{\lambda}(1)+\left( \frac{\tau \Delta t}{\delta^{3}T^{4}}\right)\left( \frac{\tau h}{\delta T^{2}}\right)^{3}\right)\taut^{-}(|D_{h}z|^{2})\\
    &+\int_{\mathcal{M}^{\ast}\times \partial \mathcal{N}}\mathcal{O}_{\lambda}(1)\taut^{+}((sh)^{2}|D_{h}z|^{2})+\Delta t\int_{\mathcal{M}\times\mathcal{N}}\mathcal{O}_{\lambda}(1)\left( \frac{\tau h}{\delta T^{2}}\right)^{2}\,|D^{2}_{ht}z|^{2}\\
    &+C\Delta t\int_{\mathcal{M}^{\ast}\times\mathcal{N}}\left( \frac{\tau \Delta t}{\delta^{2}T^{3}}\right)\left( \frac{\tau h}{\delta T^{2}}\right)|D_{ht}^{2}z|^{2}\\
    &+C\Delta t\int_{\mathcal{M}^{\ast}\times\mathcal{N}}\left[ \left( \frac{\tau^{2}\Delta t}{\delta^{4}T^{6}}\right)+\left( \frac{\tau \Delta t}{\delta^{3}T^{4}}\right)\left(\frac{\tau h}{\delta T^{2}}  \right)^{3}\right]|D^{2}_{ht}z|^{2}.
    \end{split}
\end{equation}
\par On the other hand, we focus now on $I_{23}^{(b_{1})}$. We note that Young's inequality yields
\begin{equation*}
    |I_{23}^{(b_{1})}|\leq C \int_{\mathcal{M}^{\ast}\times\mathcal{N}}h^{2}|\taut^{-}(s^{-1}D_{h}\alpha_{23})|A_{h}D_{t}z|^{2}+\int_{\mathcal{M}^{\ast}\times\mathcal{N}}h^{2}|\taut^{-}(D_{h}\alpha_{23})|\taut^{-}(s|D_{h}z|^{2}).
\end{equation*}
Using that $|A_{h}D_{t}z|^{2}\leq A_{h}(|D_{t}z|^{2})$ on the first integral from the above expression and then a discrete integration by part for the average operator $A_{h}$, we have that
\begin{equation*}
\begin{split}
    |I_{23}^{(b_{1})}|\leq &C \int_{\mathcal{M}\times\mathcal{N}}h^{2}
    |\taut^{-}(s^{-1}A_{h}D_{h}\alpha_{23})|\,|D_{t}z|^{2}+C\int_{\partial\mathcal{M}\times \mathcal{N}}h^{3}|D_{t}z|^{2}t_{r}(\taut^{-}(s^{-1}|D_{h}\alpha_{23}|))\\
    &+\int_{\mathcal{M}^{\ast}\times\mathcal{N}}h^{2}|\taut^{-}(D_{h}\alpha_{23})|\taut^{-}(s|D_{h}z|^{2}).
    \end{split}
\end{equation*}
Finally, using Proposition \ref{pro:space:estimate}, $I_{23}^{(b_{1})}$ can be estimated as follows
\begin{equation}\label{eq:I23:b1}
    \begin{split}
    |I_{23}^{(b_{1})}|\leq &C_{\lambda} \int_{\mathcal{M}\times\mathcal{N}}\taut^{-}(s^{-1}(sh)^{2})\,|D_{t}z|^{2}+C_{\lambda}\int_{\partial\mathcal{M}\times \mathcal{N}}h\taut^{-}(s^{-1}(sh)^{2})|D_{t}z|^{2}\\
    &+C_{\lambda}\int_{\mathcal{M}^{\ast}\times\mathcal{N}}\taut^{-}(s(sh)^{2}|D_{h}z|^{2})
    \end{split}
\end{equation}
\par Similarly, for $I_{23}^{(b_{3})}$ we have
\begin{equation*}
\begin{split}
    |I_{23}^{(b_{3})}|\leq& \int_{\partial\mathcal{M}\times\mathcal{N}}h^{2}\taut^{-}(s^{-1})|\taut^{-}(\alpha_{23})||D_{t}z|^{2}+\int_{\partial\mathcal{M}\times\mathcal{N}}h^{2}\taut^{-}(s)|\taut^{-}(\alpha_{23})|t_{r}(|D_{h}\taut^{-}z|^{2}).
    \end{split}
\end{equation*}
Using the estimates from Proposition \ref{pro:space:estimate}, and shifting the time discrete variable on the last integral above, we obtain
\begin{equation}\label{eq:I23:b3}
\begin{split}
    |I_{23}^{(b_{3})}|\leq& C_{\lambda}\int_{\partial\mathcal{M}\times\mathcal{N}}\taut^{-}(s^{-1}(sh)^{2})|D_{t}z|^{2}+C_{\lambda}\int_{\partial\mathcal{M}\times\mathcal{N}^{\ast}}s(sh)^{2}t_{r}(|D_{h}z|^{2}).
    \end{split}
\end{equation}
Thus, combining estimates \eqref{eq:I13}, \eqref{eq:I23:b2}, \eqref{eq:I23:b1} and \eqref{eq:I23:b3}, and decreasing the parameter $\epsilon_{1}(\lambda)$ such that the term with $|D_{xt}^{2}z|^{2}$ that controls the other similar terms comes from $I_{13}$, we obtain the claimed inequality.
\end{proof}

\subsection{Estimate of $I_{21}$}
The proof of this estimate is similar to Lemma 5.4 from \cite{LOP-2020} since the discrete time variable does not play a major role. In contrast to Lemma 2.3 from \cite{GC-HS-2021}, spatial boundary terms appear. 
\begin{lemma}\label{I21}
Provided $\Delta t \tau (T^3 \delta^2)^{-1}\leq 1$ and $\tau h(\delta T^{2})^{-1}\leq 1$, we have
\begin{equation*}
    \begin{split}        
        I_{21}=&3\tau^{3}\lambda^{4}\int_{\mathcal{M}\times\mathcal{N}^{\ast}}\theta^{3} \phi^{3}(\partial_{x}\psi)^{4}|z|^{2}+\int_{\mathcal{M}\times\mathcal{N}^{\ast}}(s\lambda\phi)^{3}\mathcal{O}(1)|z|^{2}-\tau^{3}\lambda^{3}\int_{\partial \mathcal{M}\times\mathcal{N}^{\ast}} (\theta\phi\partial_{x}\psi)^{3}t_{r}(|A_{h}z|^{2})n\\
        &-X_{21}+Y_{21},
    \end{split}
\end{equation*}
where $\displaystyle X_{21}=\int_{\mathcal{M}\times\mathcal{N}^{\ast}}s^{2}\mathcal{O}_{\lambda}(1)+s^{3}\mathcal{O}_{\lambda}((sh)^{2}) |z|^{2}-\int_{\mathcal{M}^{\ast}\times\mathcal{N}^{\ast}} s\mathcal{O}_{\lambda}((sh)^{2})|D_{h}z|^{2}$ and
\begin{equation*}
    \begin{split}
        Y_{21}:=&\int_{\partial \mathcal{M}\times\mathcal{N}^{\ast}}\left(s^{2}\mathcal{O}_{\lambda}(1)+s^{3}\mathcal{O}_{\lambda}((sh)^{2})\right)t_{r}(|A_{h}z|^{2})n-\int_{\partial \mathcal{M}\times\mathcal{N}^{\ast}}s\mathcal{O}_{\lambda}((sh)^{2})t_{r}(|D_{h}z|^{2})\,n.
        \end{split}
    \end{equation*}
\end{lemma}
\begin{proof}
Let us set $\alpha_{21}:=r^{2}D_{h}^{2}\rho\,A_{h}D_{h}\rho$. By regarding a shift with respect to the discrete time variable, $I_{21}$ is given by $\displaystyle I_{21}=\int_{\mathcal{M}\times\mathcal{N}^{\ast}} 2\alpha_{21}\,A_{h}^{2}z\,D_{h}A_{h}z$. Then, by using the identity $D_{h}(|A_{h}z|^{2})=2A_{h}^{2}z\,D_{h}A_{h}z$ due to Lemma \ref{lem:product:rule}, we obtain $\displaystyle I_{21}=\int_{\mathcal{M}\times\mathcal{N}^{\ast}} \alpha_{21}\,D_{h}((A_{h}z)^{2})$. After a spatial discrete integration by parts we have
\begin{equation*}
    \begin{split}
        I_{21}
        =&-\int_{\mathcal{M}^{\ast}\times\mathcal{N}^{\ast}}D_{h}(\alpha_{21})(A_{h}z)^{2}+\int_{\partial \mathcal{M}\times\mathcal{N}^{\ast}}\alpha_{21}\,t_{r}\left((A_{h}z)^{2}\right)\,n=:I_{21}^{(a)}+I_{21}^{(b)}.
        \end{split}
        \end{equation*}
Note that by using the identity $|A_{h}z|^{2}=A_{h}(|z|^{2})-\frac{h^{2}}{4}|D_{h}z|^{2}$ it follows that
\begin{equation*}
\begin{split}
        I_{21}^{(a)}
        =&-\int_{\mathcal{M}^{\ast}\times\mathcal{N}^{\ast}} D_{h}(\alpha_{21})\,A_{h}(z^{2})+\frac{h^{2}}{4}\int_{\mathcal{M}^{\ast}\times\mathcal{N}^{\ast}} D_{h}(\alpha_{21})\,|D_{h}z|^{2}.
    \end{split}
    \end{equation*}
    It then follows from Proposition \ref{pro:integral:space} that $I_{21}^{(a)}$ can be written as
    \begin{equation*}
    \begin{split}
    I_{21}^{(a)}=&-\int_{\mathcal{M}\times\mathcal{N}^{\ast}} A_{h}D_{h}(\alpha_{21})\,|z|^{2}-\frac{h}{2}\int_{\partial \mathcal{M}\times\mathcal{N}^{\ast}}t_{r}(D_{h}\alpha_{21})\,|z|^{2}+\frac{h^{2}}{4}\int_{\mathcal{M}^{\ast}\times\mathcal{N}^{\ast}} D_{h}(\alpha_{21})\,|D_{h}z|^{2}.
        \end{split}
        \end{equation*}
On the other hand, we have by Proposition \ref{pro:space:estimate} and Corollary \ref{3.8} that $A_{h}D_{h}(\alpha_{21})=-3s^{3}\lambda^{4}\phi^{3}(\partial_{x}\psi)^{4}+(s\lambda\phi)^{3}\mathcal{O}(1)+s^{2}\mathcal{O}_{\lambda}(1)+s^{3}\mathcal{O}_{\lambda}((sh)^{2})$ and
$ D_{h}(\alpha_{21})=s^{3}\mathcal{O}_{\lambda}(1)$. Therefore, with these estimates, $I_{21}^{(a)}$ can be estimated as
\begin{equation}
\begin{split}
       I_{21}^{(a)}=&
\int_{\mathcal{M}\times\mathcal{N}^{\ast}}3s^{3}\lambda^{4} \phi^{3}(\partial_{x}\psi)^{4}|z|^{2}+\int_{\mathcal{M}\times\mathcal{N}^{\ast}}(s\lambda\phi)^{3}\mathcal{O}(1)|z|^{2}-\int_{\mathcal{M}\times\mathcal{N}^{\ast}}(s^{2}\mathcal{O}_{\lambda}(1)+s^{3}\mathcal{O}_{\lambda}((sh)^{2}) |z|^{2}\\
&-\int_{\mathcal{M}^{\ast}\times\mathcal{N}^{\ast}} s\mathcal{O}_{\lambda}((sh)^{2})|D_{h}z|^{2}-\int_{\partial \mathcal{M}\times\mathcal{N}^{\ast}}s^{2}\mathcal{O}_{\lambda}(sh)|z|^{2}.
\end{split}
\end{equation}
Finally, for the estimate of $I_{21}^{(b)}$, we note that
 \begin{equation*}
\begin{split}
    I_{21}^{(b)}=&\int_{\partial \mathcal{M}\times\mathcal{N}^{\ast}}\left(-(s\lambda\phi\partial_{x}\psi)^{3}+s^{2}\mathcal{O}_{\lambda}(1)+s^{3}\mathcal{O}_{\lambda}((sh)^{2})\right)\,t_{r}(|A_{h}z|^{2}))\,n,
    \end{split}
\end{equation*}
since $\alpha_{21}=-(s\lambda\partial_{x}\psi)^{3}+s^{2}\mathcal{O}_{\lambda}(1)+s^{3}\mathcal{O}_{\lambda}((sh)^{2})$, due to  Proposition \ref{pro:space:estimate}. Then, $I_{21}$ can be estimated as
\begin{equation}
    \begin{split}        
        I_{21}=&\int_{\mathcal{M}\times\mathcal{N}^{\ast}}3s^{3}\lambda^{4} \phi^{3}(\partial_{x}\psi)^{4}|z|^{2}+\int_{\mathcal{M}\times\mathcal{N}^{\ast}}(s\lambda\phi)^{3}\mathcal{O}(1)|z|^{2}-X_{21}+Y_{21}.
    \end{split}
\end{equation}
\end{proof}

\subsection{Estimate of $I_{22}$}
\begin{lemma}\label{I22}
Provided $\Delta t \tau (T^3 \delta^2)^{-1}\leq 1$ and $\tau h(\delta T^{2})^{-1}\leq 1$, we have
\begin{equation*}
    \begin{split}
    I_{22}
    \geq &-\tau^{3}\lambda^{4}\int_{\mathcal{M}\times\mathcal{N}^{\ast}}\theta^{3}\phi^{3}(\partial_{x}\psi)^{4}|z|^{2}+\tau^{3}\lambda^{3}\int_{\mathcal{M}\times\mathcal{N}^{\ast}} \theta^{3}\phi^{2}(\partial_{x}\psi)^{2}\partial_{x}^{2}\psi|z|^{2}-X_{22}+Y_{22},
    \end{split}
\end{equation*}
where 
\begin{equation*}
    X_{22}:=\int_{\mathcal{M}\times\mathcal{N}^{\ast} }\left(s^{2}\mathcal{O}_{\lambda}(1)+s^{3}\mathcal{O}_{\lambda}((sh)^{2})\right)|z|^{2}
    +\int_{\mathcal{M}\times\mathcal{N}^{\ast}} s\mathcal{O}_{\lambda}((sh)^{2})|z|^{2}+\int_{\mathcal{M}^{\ast}\times\mathcal{N}^{\ast}}s \mathcal{O}_{\lambda}((sh)^{2})|D_{h}z|^{2},
\end{equation*}
and $\displaystyle Y_{22}:=\int_{\partial \mathcal{M}\times\mathcal{N}^{\ast} }s\mathcal{O}_{\lambda}((sh)^{2})|z|^{2}+\int_{\partial \mathcal{M}\times\mathcal{N}^{\ast} }s\mathcal{O}_{\lambda}((sh)^{2})t_{r}(|D_{h}z|^{2})\,n$.
\end{lemma}
\begin{proof}
Let us set $\alpha_{22}:=s\,r\,D_{h}^{2}\rho\,\partial_{x}^{2}\phi$. We shift the time variable to obtain $\displaystyle I_{22}=-
    \int_{\mathcal{M}\times\mathcal{N}^{\ast}}2\alpha_{22}\, z\,A_{h}^{2}z$. Then, by using the identity $A_{h}^{2}z=z+\frac{h^{2}}{4}D_{h}^{2}z$, we write
\begin{equation*}
\begin{split}
I_{22}
    =&-\int_{\mathcal{M}\times\mathcal{N}^{\ast}}2\alpha_{22}|z|^{2}
    -\frac{h^{2}}{2}\int_{\mathcal{M}\times\mathcal{N}^{\ast}}\alpha_{22} z\,D_{h}^{2}z=:I_{22}^{(a)}+I_{22}^{(b)}.
    \end{split}
    \end{equation*}
By virtue of Proposition \ref{pro:space:estimate} and Lemma \ref{3.7} we note that
$rD_{h}^{2}\rho=(s\lambda\phi)^{2}(\partial_{x}\psi)^{2}+s\mathcal{O}_{\lambda}(1)+s ^{2}\mathcal{O}_{\lambda}((sh)^{2})$, and in addition to $\partial_{x}^{2}\phi=\lambda^{2}(\partial_{x}\psi)^{2}\phi+\lambda\phi\partial_{x}^{2}\psi$, we get that
\begin{equation}\label{eq:alpha:22}
\begin{split}
    \alpha_{22}=&s^{3}\lambda^{4}\phi^{3}(\partial_{x}\psi)^{4}+s^{3}\lambda^{3}\phi^{2}(\partial_{x}\psi)^{2}\partial_{x}^{2}\psi+s^{2}\mathcal{O}_{\lambda}(1)+s^{3}\mathcal{O}_{\lambda}((sh)^{2})=s^{3}\mathcal{O}_{\lambda}(1).
    \end{split}
    \end{equation}
    Then, using \eqref{eq:alpha:22}, for $I_{22}^{(a)}$ we get
    \begin{equation}\label{eq:I22:a}
        I_{22}^{(a)}=-\int_{\mathcal{M}\times\mathcal{N}^{\ast}}s^{3}\lambda^{4}\phi^{3}(\partial_{x}\psi)^{4}|z|^{2}-\int_{\mathcal{M}\times\mathcal{N}^{\ast}}\left( s^{3}\lambda^{3}\phi^{2}(\partial_{x}\psi)^{2}\partial_{x}^{2}\psi+s^{2}\mathcal{O}_{\lambda}(1)+s^{3}\mathcal{O}_{\lambda}((sh)^{2})\right)|z|^{2}.
    \end{equation}
    Similarly, an integration by parts for $I_{22}^{(b)}$ yields 
    \begin{equation*}
        \begin{split}
    I_{22}^{(b)}=&\frac{h^{2}}{2}\int_{\mathcal{M}^{\ast}\times\mathcal{N}^{\ast}} D_{h}(\alpha_{22}\, z)\,D_{h}z-\frac{h^{2}}{2}\int_{\partial \mathcal{M}\times\mathcal{N}^{\ast} }\alpha_{22}\, z\,t_{r}(D_{h}z)\,n=:I_{22}^{(b_{1})}+I_{22}^{(b_{2})}.
    \end{split}
    \end{equation*}
Note that by using \eqref{eq:alpha:22} and Young's inequality, $I_{22}^{(b_{2})}$ can be estimated as
\begin{equation}\label{eq:I22:b2}
    |I_{22}^{(b_{2})}|\leq \int_{\partial \mathcal{M}\times\mathcal{N}^{\ast} }s|\mathcal{O}_{\lambda}((sh)^{2})|\,|z|^{2}+ \int_{\partial \mathcal{M}\times\mathcal{N}^{\ast}}s|\mathcal{O}_{\lambda}((sh)^{2})|\,t_{r}(|D_{h}z|^{2}).
\end{equation}
Now, for $I_{22}^{(b_{1})}$, by using Lemma \ref{lem:product:rule} we write $D_{h}(|z|^{2})=2D_{h}z\,A_{h}z$. Hence, 
      a discrete integration by parts with respect to the difference operator $D_{h}$ yields 
    \begin{equation*}
    \begin{split}
    I_{22}^{(b_{1})}=&-\frac{h^{2}}{8}\int_{\mathcal{M}\times\mathcal{N}^{\ast} } D^{2}_{h}(\alpha_{22})|z|^{2}+\frac{h^{2}}{8}\int_{\partial \mathcal{M}\times\mathcal{N}^{\ast}} |z|^{2}t_{r}(D_{h}\alpha_{22})n+\frac{h^{2}}{4}\int_{\mathcal{M} ^{\ast}\times\mathcal{N}^{\ast}} A_{h}(\alpha_{22})|D_{h}z|^{2}.
    \end{split}
\end{equation*}
Thanks to $\partial_{x}^{2}\phi=\mathcal{O}_{\lambda}(1)$, applying Lemma \ref{lem:product:rule}, and then Lemma \ref{lem:space:estimate} and Proposition \ref{pro:space:estimate} it follows that $A_{h}(\alpha_{22})=s^{3}\mathcal{O}_{\lambda}(1)$, $D_{h}^{2}\alpha_{22}=s^{3}\mathcal{O}_{\lambda}(1)$ and $D_{h}\alpha_{22}=s^{3}\mathcal{O}_{\lambda}(1)$. We thus have
\begin{equation}\label{eq:I22:b1}
    \begin{split}
     I_{22}^{(b_{1})} =&-\int_{\mathcal{M}\times\mathcal{N}^{\ast}}s \mathcal{O}_{\lambda}((sh)^{2})|z|^{2}+\int_{\mathcal{M}^{\ast}\times\mathcal{N}^{\ast}}s \mathcal{O}_{\lambda}((sh)^{2})|D_{h}z|^{2}+\int_{\partial\mathcal{M}\times\mathcal{N}^{\ast} }s\mathcal{O}_{\lambda}((sh)^{2}) |z|^{2}.
    \end{split}
\end{equation}
Therefore, combining \eqref{eq:I22:a}, \eqref{eq:I22:b2} and \eqref{eq:I22:b1} we obtain the desired result.
\end{proof}

\subsection{Estimate of $I_{31}$}
\begin{lemma}\label{I31}
Provided $\Delta t \tau (T^3 \delta^2)^{-1}\leq 1$ and $\tau h(\delta T^{2})^{-1}\leq 1$, we have
\begin{equation*}
\begin{split}
    I_{31}\geq &X_{31}+Y_{31}
    ,
    \end{split}
\end{equation*}
where $\displaystyle X_{31}:=\int_{\mathcal{M}\times\mathcal{N}^{\ast}}\lambda\varphi\partial_{t}\theta s\mathcal{O}_{\lambda}(1)\,|z|^{2}+\int_{\mathcal{M}^{\ast}\times\mathcal{N}^{\ast}}s\mathcal{O}_{\lambda}(1)\,|D_{h}z|^{2}$ and
\begin{equation*}
    Y_{31}:=\int_{\partial \mathcal{M}\times\mathcal{N}^{\ast}}\lambda\varphi\partial_{t}\theta s\mathcal{O}_{\lambda}(1)\,|z|^{2} n- h\int_{\partial \mathcal{M}\times\mathcal{N}^{\ast}}\lambda\varphi\partial_{t}\theta s\mathcal{O}_{\lambda}(1)\,|z|^{2}-h\int_{\partial \mathcal{M}\times\mathcal{N}^{\ast}}\lambda\varphi\partial_{t}\theta s\mathcal{O}_{\lambda}(1)t_{r}(|D_{h}z|^{2}).
\end{equation*}
\end{lemma}
\begin{proof}
Let us set $\alpha_{31}:=-\lambda\varphi\theta'rD_{h}A_{h}\rho$. Then, shifting the time variable for $I_{31}$ yields $\displaystyle I_{31}=\int_{\mathcal{M}\times\mathcal{N}^{\ast}}2\alpha_{31}z\,D_{h}A_{h}z$. A discrete integration by parts with respect to the average operator $A_{h}$ leads to
\begin{equation*}
\begin{split}
    I_{31}=&\int_{\mathcal{M}^{\ast}\times\mathcal{N}^{\ast}}2A_{h}(\alpha_{31}z)D_{h}z-\frac{h}{2}\int_{\partial \mathcal{M}\times\mathcal{N}^{\ast}}2\alpha_{31}\,z\,t_{r}(D_{h}z)=:I_{31}^{(a)}+I_{31}^{(b)}.
    \end{split}
\end{equation*}
From Lemma \ref{lem:product:rule} and the identity $2D_{h}A_{h}z=D_{h}(|z|^{2})$, for $I_{31}^{(a)}$, we have
\begin{equation*}
\begin{split}
    I_{31}^{(a)}
    =&\int_{\mathcal{M}^{\ast}\times\mathcal{N}^{\ast}}A_{h}(\alpha_{31})\,D_{h}(|z|^{2})+\int_{\mathcal{M}^{\ast}\times\mathcal{N}^{\ast}}2D_{h}(\alpha_{31})\,|D_{h}z|^{2}.
    \end{split}
\end{equation*}
A discrete integration by parts with respect to the difference operator $D_{h}$, on the first integral above, yields
\begin{equation*}
\begin{split}
    I_{31}^{(a)}=&-\int_{\mathcal{M}\times\mathcal{N}^{\ast}}D_{h}A_{h}(\alpha_{31})\,|z|^{2}+\int_{\partial\mathcal{M} \times\mathcal{N}^{\ast}}|z|^{2}\,t_{r}(A_{h}\alpha_{31})\, n+\int_{\mathcal{M}^{\ast}\times\mathcal{N}^{\ast}}2D_{h}(\alpha_{31})\,|D_{h}z|^{2}
    \end{split}
\end{equation*}
\par On the other hand, the Young's inequality applying on $I_{31}^{(b)}$ yields
\begin{equation*}
    |I_{31}^{(b)}|\leq \frac{h}{2}\int_{\partial \mathcal{M}\times\mathcal{N}^{\ast}}|\alpha_{31}| (|z|^{2}+t_{r}(|D_{h}z|^{2})).
\end{equation*}
We note that $e^{s\varphi}\partial_{t}(e^{-s\varphi})=-\lambda\varphi\partial_{t}\theta$ and by Theorem \ref{teo:space:estimate} we have that $rA_{h}D_{h}\rho=s\mathcal{O}_{\lambda}(1)$. Thus, for $I_{31}^{(a)}$ we have the following estimate
\begin{equation}
\begin{split}
    I_{31}^{(a)}=&\int_{\mathcal{M}\times\mathcal{N}^{\ast}}\lambda\varphi\partial_{t}\theta s\mathcal{O}_{\lambda}(1)\,|z|^{2}+\int_{\partial \mathcal{M}\times\mathcal{N}^{\ast}}\lambda\varphi\partial_{t}\theta s\mathcal{O}_{\lambda}(1)\,|z|^{2} n+\int_{\mathcal{M}^{\ast}\times\mathcal{N}^{\ast}}s\mathcal{O}_{\lambda}(1)\,|D_{h}z|^{2}
    \end{split}
\end{equation}
and for $I_{31}^{(b)}$ 
\begin{equation}
    |I_{31}^{(b)}|\leq h\int_{\partial \mathcal{M}\times\mathcal{N}^{\ast}}\lambda\varphi\partial_{t}\theta s\mathcal{O}_{\lambda}(1)\,|z|^{2}+h\int_{\partial \mathcal{M}\times\mathcal{N}^{\ast}}\lambda\varphi\partial_{t}\theta s\mathcal{O}_{\lambda}(1)t_{r}(|D_{h}z|^{2}).
\end{equation}
Therefore, $I_{31}$ can be estimated as $\displaystyle I_{31}\geq \int_{\mathcal{M}\times\mathcal{N}^{\ast}}\lambda\varphi\partial_{t}\theta s\mathcal{O}_{\lambda}(1)\,|z|^{2}+\int_{\mathcal{M}^{\ast}\times\mathcal{N}^{\ast}}s\mathcal{O}_{\lambda}(1)\,|D_{h}z|^{2}+Y_{31}$, and the proof is complete.
\end{proof}

\subsection{Estimate of $I_{32}$}
\begin{lemma}\label{I32}
Provided $\Delta t \tau (T^3 \delta^2)^{-1}\leq 1$ and $\tau h(\delta T^{2})^{-1}\leq 1$, we have
\begin{equation*}
I_{32}\geq-\int_{\mathcal{M}\times\mathcal{N}^{\ast}}Ts^{2}\theta\mathcal{O}_{\lambda}(1)|z|^{2}:=-X_{32}.
\end{equation*}
\end{lemma}
\begin{proof}
Let us set $\alpha_{32}:=2s\partial_{x}^{2}\phi\lambda\varphi\theta'$. We have that shifting the discrete time variable $I_{32}$ is given by $\displaystyle I_{32}=\int_{\mathcal{M}\times\mathcal{N}^{\ast}}\alpha_{32}|z|^{2}$. Then, it follows that $I_{32}\geq -\int_{\mathcal{M}\times\mathcal{N}^{\ast}}Ts^{2}\theta\mathcal{O}_{\lambda}(1)|z|^{2}$,
since $|\theta'|\leq CT\theta^{2}$, $\left\| \partial_{x}^{2}\phi\right\|_{L^{\infty}_{h}(\mathcal{M}\times\mathcal{N}^{\ast})}\leq \mathcal{O}_{\lambda}(1)$ and $\left\| \varphi\right\|_{C(\Omega)}=\mathcal{O}_{\lambda}(1)$ which complete the proof.
\end{proof}

\subsection{Estimate of $I_{33}$}
In this estimate we face a similar situation as in the proof of the estimate for $I_{23}$. The equation (181) in the proof of Lemma 2.7 from \cite{GC-HS-2021} needs some additional steps since after the application of the integration by parts (155) from that work the term with $|z|^{2}$ should be shifted to the right in time and not to the left. Thus, some new terms will appear although, as in the proof of the estimate for $I_{23}$, it will not change the principal term in the final estimate.
\begin{lemma}\label{I33}
Provided $\Delta t \tau (T^3 \delta^2)^{-1}\leq 1$ and $\tau h(\delta T^{2})^{-1}\leq 1$, we have
\begin{equation*}
    \begin{split}
        I_{33}\geq &-\int_{\mathcal{M}\times\mathcal{N}^{\ast}}\mu_{23}|z|^{2}-\int_{\mathcal{M}\times\mathcal{N}}(\Delta t)^{2}\taut^{-}(\mu_{23})|D_{t}z|^{2}-\int_{\mathcal{M}\times\mathcal{N}}\taut^{-}(\gamma_{23})|D_{t}z|^{2}=:-X_{33},
    \end{split}
\end{equation*}
with $\mu_{23}:=(\tau T^{2}\theta^{3}+\frac{\tau \Delta}{\delta^{4}T^{5}})\mathcal{O}_{\lambda}(1)$ and $\gamma_{33}:=\Delta t \tau T\theta^{2}\mathcal{O}_{\lambda}(1)$.
\end{lemma}
\begin{proof}
We have to estimate
$\displaystyle I_{33}:=-\int_{\mathcal{M}\times\mathcal{N}}\tau\varphi\taut^{-}(\theta'\,z)D_{t}z$. Note that by using the identity \eqref{eq:identity:square}, $I_{33}$ can be rewritten as $\displaystyle I_{33}=-\frac{1}{2}\int_{\mathcal{M}\times\mathcal{N}}\tau\varphi\taut^{-}(\theta')\,D_{t}(|z|^{2})+\frac{\Delta t}{2}\int_{\mathcal{M}\times\mathcal{N}}\tau\varphi\taut^{-}(\theta')|D_{t}z|^{2}$. Now, using \eqref{eq:integral:time:primal} on the first integral above we have that
\begin{equation*}
\begin{split}
    I_{33}=&\frac{1}{2}\int_{\mathcal{M}\times\mathcal{N}}\tau\varphi\,D_{t}(\theta')\,\taut^{+}(|z|^{2})-\int_{\mathcal{M}\times\partial \mathcal{N}}\tau\varphi\taut^{+}(|z|^{2})\taut^{+}(\theta')n+\frac{\Delta t}{2}\int_{\mathcal{M}\times\mathcal{N}^{\ast}}\tau\varphi\taut^{-}(\theta')|D_{t}z|^{2}.
    \end{split}
\end{equation*}
We note that $\taut^{+}(\theta')\,n>0$ on $\mathcal{M}\times\partial\mathcal{N}$ since $\theta'=(2t-T)\theta^{2}$. Moreover, thanks to $\varphi<0$ in $\mathcal{M}$, it follows that
\begin{equation}
\int_{\mathcal{M}\times\partial\mathcal{N}}\lambda\varphi\taut^{+}(|z|^{2})\taut^{+}(\theta')\,n<0.
\end{equation}
Thus, we can drop the boundary terms and $I_{33}$ can be estimated as
\begin{equation*}
    I_{33}\geq\frac{1}{2}\int_{\mathcal{M}\times\mathcal{N}}\tau\varphi\,D_{t}(\theta')\,\taut^{+}(|z|^{2})+\frac{\Delta t}{2}\int_{\mathcal{M}\times\mathcal{N}^{\ast}}\lambda\varphi\taut^{-}(\theta')|D_{t}z|^{2}.
\end{equation*}    
Now, let us focus on the integral with the term $|\taut^{+}z|^{2}$. Using that $|\taut^{+}z|^{2}\leq C|\taut^{-}z|^{2}+C(\Delta t)^{2}|D_{t}z|^{2}$ and Lemma  \ref{lem:discrete:theta} we obtain
\begin{equation*}
    \begin{split}
        I_{33}\geq &-\int_{\mathcal{M}\times\mathcal{N}}\taut^{-}(\mu_{23})|\taut^{-}z|^{2}-\int_{\mathcal{M}\times\mathcal{N}}(\Delta t)^{2}\taut^{-}(\mu_{23})|D_{t}z|^{2}-\int_{\mathcal{M}\times\mathcal{N}}\taut^{-}(\gamma_{23})|D_{t}z|^{2},
    \end{split}
\end{equation*}
with $\mu_{23}:=(\tau T^{2}\theta^{3}+\frac{\tau \Delta}{\delta^{4}T^{5}})\mathcal{O}_{\lambda}(1)$ and $\gamma_{33}:=\tau\Delta t  T\theta^{2}\mathcal{O}_{\lambda}(1)$. Using \eqref{eq:tau:menos} we can shift the integral with the term $|\taut^{-}z|^{2}$ which completes the proof.
\end{proof}

\section{Boundary estimates}\label{apen:B}
In this section we prove the estimates for the left boundary terms. Since the proof for the right boundary terms is analogous we omit its proofs. 
\subsection{Estimate of $J_{11}$}
\begin{lemma}\label{lem:J11}
For $\tau\Delta t  (T^3 \delta^2)^{-1}\leq 1$ and $\tau h(\delta T^{2})^{-1}\leq 1$, we have
\begin{equation*}
    J_{11}\geq-\int_{\{0\}\times\mathcal{N}^{\ast}}\tilde{\alpha}_{11}^{(1)} |z|^{2} -(\Delta t)^{2}\int_{\{0\}\times\mathcal{N}}\taut^{-}(\tilde{\alpha}_{11}^{(1)})|D_{t}z|^{2}-\Delta t \int_{\mathcal{N}\times\{0\}}\taut^{-}(\tilde{\alpha}_{11}^{(2)})|D_{t}z|^{2}.
\end{equation*}
with $\tilde{\alpha}^{(1)}_{11}:=(\tau^{2} T^{2}\theta^{3}+\frac{\tau \Delta t}{\delta^{4}T^{5}})\mathcal{O}_{\lambda}(1)$ and $\gamma^{(2)}_{11}:=\Delta t\tau T\theta^{2}\mathcal{O}_{\lambda}(1)$.
\end{lemma}
\begin{proof}
Let us recall that $J_{11}$ is given by $\displaystyle J_{11}:=-\tau\int_{\{0\}\times\mathcal{N}}\varphi\taut^{-}(\theta'z)D_{t}z$. 
By using the identity \eqref{eq:identity:square}, we have that $D_{t}z\,\tau^{+}z=\frac{1}{2}D_{t}(|z|^{2})+\frac{1}{2}\Delta t |D_{t}z|^{2}$. Then, the integral $J_{11}$ can be rewritten as 
\begin{equation*}
    J_{11}=-\frac{\tau}{2}\int_{\{0\}\times\mathcal{N}}\varphi\taut^{-}(\theta')D_{t}(|z|^{2})+\frac{\tau\Delta t}{2}\int_{\{0\}\times\mathcal{N}}\varphi\taut^{-}(\theta')|D_{t}z|^{2}.
\end{equation*}
Applying the discrete integral by parts \eqref{eq:integral:time:primal} on the first integral above, it follows that
\begin{equation*}
    J_{11}=\frac{\tau}{2}\int_{\{0\}\times\mathcal{N}^{\ast}}\varphi D_{t}(\theta')|\taut^{+}z|^{2}-\frac{\tau}{2}\int_{\{0\}\times\partial \mathcal{N}}\varphi \taut^{+}(\theta'|z|^{2})n+\frac{\tau\Delta t}{2}\int_{\{0\}\times\mathcal{N}}\varphi\taut^{-}(\theta')|D_{t}z|^{2}.
\end{equation*}
We note that the time boundary terms, the second integral above, can be drooped since it is positive. Thus, using the estimation from Lemma \ref{lem:discrete:theta} and the inequality $|\taut^{+}z|^{2}\leq C|\taut^{-}z|^{2}+C(\Delta t)^{2}|D_{t}z|^{2}$, we obtain the result.
\end{proof}
\subsection{Estimate of $J_{12}$}
\begin{lemma}\label{lem:J12}
For $\Delta t \tau (T^3 \delta^2)^{-1}\leq 1$ and $\tau h(\delta T^{2})^{-1}\leq 1$, we have
\begin{equation*}
    J_{12}\geq - \int_{\{0\}\times\mathcal{N}}\left(s^{-1}+c_ {\lambda}s^{-1}(sh)^{2}\right)|D_{t}z|^{2}-\int_{\{h/2\}\times\mathcal{N}^{\ast}}\left(s+c_{\lambda}s(sh)^{2}\right)|D_{h}z|^{2}.
\end{equation*}
\end{lemma}
\begin{proof}
Setting $\tilde{\alpha}_{12}:=\taut_{+}^{-}(rA_{h}\rho)$, $J_{12}$ is defined by $\displaystyle J_{12}:=\int_{\{0\}\times\mathcal{N}}\tilde{\alpha}_{12}\,D_{t}z\taut^{-}_{+}(D_{h}z)$. Thanks to Proposition \ref{pro:space:estimate} we have $\tilde{\alpha}_{12}=1+c_{\lambda}\taut^{-}((sh)^{2})$, then using Young's inequality for $J_{12}$ we obtain
\begin{equation*}
    |J_{12}|\leq \int_{\{0\}\times\mathcal{N}}\left(s^{-1}+c_ {\lambda}s^{-1}(sh)^{2}\right)|D_{t}z|^{2}+\int_{\{h/2\}\times\mathcal{N}^{\ast}}\left(s+c_{\lambda}s(sh)^{2}\right)|D_{h}z|^{2}.
\end{equation*}
\end{proof}
\subsection{Estimate of $J_{21}$}
\begin{lemma}\label{lem:J21}
For $\tau\Delta t  (T^3 \delta^2)^{-1}\leq 1$ and $\tau h(\delta T^{2})^{-1}\leq 1$, we have
\begin{equation*}
    J_{21}\geq  -\int_{\{0\}\times\mathcal{N}^{\ast}}\tilde{\alpha}_{21}|A_{h}z|^{2}-\int_{\{0\}\times\mathcal{N}^{\ast}}\tilde{\alpha}_{21}|z|^{2},
\end{equation*}
where $\tilde{\alpha}_{21}:=\tau^{2}\theta^{3}\mathcal{O}_{\lambda}(1)$.
\end{lemma}
\begin{proof}
Setting $\tilde{\beta}_{21}:=-\tau\varphi\theta'r\taut_{+}(D_{h}\rho)$, we note that $J_{21}$ is given by $\displaystyle
    J_{21}:=\int_{\{0\}\times\mathcal{N}}\taut^{-}(\tilde{\beta}_{21}zA_{h}\taut_{+}z)$. Then, using Young's inequality we have
\begin{equation*}
    |J_{21}|\leq \int_{\{h/2\}\times\mathcal{N}^{\ast}}|\tilde{\beta}_{21}||A_{h}z|^{2}+\int_{\{0\}\times\mathcal{N}^{\ast}}|\tilde{\beta}_{21}||z|^{2},
\end{equation*}
where we have used a shift in the discrete time variable. Thus, applying Proposition \ref{properties:time:discret:ops} and Lemma \ref{lem:discrete:theta} we obtain for $J_{21}$
\begin{equation*}
    |J_{21}|\leq \int_{\{h/2\}\times\mathcal{N}^{\ast}}\gamma_{21}|A_{h}z|^{2}+\int_{\{0\}\times\mathcal{N}^{\ast}}\gamma_{21}|z|^{2},
\end{equation*}
with $\tilde{\alpha}_{21}:=\tau^{2}\theta^{3}\mathcal{O}_{\lambda}(1)$.
\end{proof}
\subsection{Estimate of $J_{22}$}
\begin{lemma}\label{lem:J22}
For $\tau\Delta t  (T^3 \delta^2)^{-1}\leq 1$ and $\tau h(\delta T^{2})^{-1}\leq 1$, we have
\begin{equation*}
    J_{22}\geq-\int_{\{h/2\}\times\mathcal{N}^{\ast}}s^{2}\mathcal{O}_{\lambda}(1)\,|A_{h}z|^{2}-\int_{\{h/2\}\times\mathcal{N}^{\ast}}|D_{h}z|^{2}.
\end{equation*}
\end{lemma}
\begin{proof}
Let us define $\tilde{\beta}_{22}:=r^{2}\taut_{+}(D_{h}\rho A_{h}\rho)$. We have to estimate 
$\displaystyle J_{22}:=\int_{\{0\}\times\mathcal{N}}\tau^{-}(\tilde{\beta}_{22})\taut_{+}^{-} (A_{h}z D_{h}z)$. Firstly, thanks to Proposition \ref{properties:time:discret:ops}, we can shift $J_{22}$ with respect to the discrete time variable obtaining $\displaystyle J_{22}=\int_{\{0\}\times\mathcal{N}^{\ast}}\tilde{\beta}_{22}\,\tau_{+}(D_{h}zA_{h}z)$. Now, by virtue of Young's inequality and Proposition \ref{pro:space:estimate}, $J_{22}$ can be estimated as follows
\begin{equation*}
    |J_{22}|\leq\int_{\{h/2\}\times\mathcal{N}^{\ast}}s^{2}\mathcal{O}_{\lambda}(1)\,|A_{h}z|^{2}+\int_{\{h/2\}\times\mathcal{N}^{\ast}}|D_{h}z|^{2},
\end{equation*}
which completes the proof.
\end{proof}

\subsection{Adding the term with $|D_{t}z|^{2}$ on the boundary}
From the definition of $Ez$ we have that $D_{t}z=Ez+\taut^{-}(r)\taut^{-}_{-}(D_{h}\rho A_{h}z)$ on $\{1\}\times \mathcal{N}$, then taking $L^{2}(\mathcal{N})$-norm, and considering Proposition \ref{pro:space:estimate}, it follows that
\begin{equation*}
    \int_{\{1\}\times\mathcal{N}}|D_{t}z|^{2}\leq \int_{\{1 \}\times\mathcal{N}}|Ez|^{2}+C_{\lambda}\int_{\{1\}\times\mathcal{N}}\taut^{-}(s^{2})|\taut_{-}^{-}A_{h}z|^{2}.
\end{equation*}
Similarly, using the definition of $Cz$, for the left boundary we have that 
\begin{equation*}
    \int_{\{0\}\times\mathcal{N}}|D_{t}z|^{2}\leq \int_{\{0\}\times\mathcal{N}}|Cz|^{2}+C_{\lambda}\int_{\{0\}\times\mathcal{N}}\taut^{-}(s^{2})|\taut_{+}^{-}A_{h}z|^{2}.
\end{equation*}
Collecting the above inequalities, and shifting the discrete time variable on the integrals with the term $|A_{h}|$, we obtain
\begin{equation}
    \int_{\partial\mathcal{M}\times\mathcal{N}}|D_{t}z|^{2}\leq \int_{\{0\}\times\mathcal{N}}|Cz|^{2}+\int_{\{1\} \times\mathcal{N}}|Ez|^{2}+C_{\lambda}\int_{\partial\mathcal{M}\times\mathcal{N}^{\ast}}s^{2}t_{r}(|A_{h}z|^{2})
\end{equation}

\subsection{Proof of Lemma \ref{lem:return:variable}}
Let us recall that $\rho z=q$. Then, using Lemma \ref{lem:product:rule} we have that
\begin{equation*}
    D_{h}q=D_{h}\rho\,A_{h}z+A_{h}\rho\,D_{h}z.
\end{equation*}
To proof \eqref{eq:return:Dx}, we multiply by $rs^{1/2}$ the above expression, and taking the $L^{2}(\mathcal{M}^{\ast}\times\mathcal{N}^{\ast})$-norm to the result, to get
\begin{equation*}
    \int_{\mathcal{M}^{\ast}\times\mathcal{N}^{\ast}}r^{2}s|D_{h}q|^{2}\leq \int_{\mathcal{M}^{\ast}\times\mathcal{N}^{\ast}}sr^{2}|D_{h}\rho|^{2}|A_{h}z|^{2}+\int_{\mathcal{M}^{\ast}\times\mathcal{N}^{\ast}}sr^{2}|A_{h}\rho|^{2}|D_{h}z|^{2}.
\end{equation*}
Now, using Proposition \ref{pro:space:estimate} and that $|A_{h}z|^{2}\leq A_{h}(|z|^{2})$, it follows from the above inequality 
\begin{equation*}
        \int_{\mathcal{M}^{\ast}\times\mathcal{N}^{\ast}}r^{2}s|D_{h}q|^{2}\leq C_{\lambda}\left( \int_{\mathcal{M}^{\ast}\times\mathcal{N}^{\ast}}s^{3}A_{h}(|z|^{2})+\int_{\mathcal{M}^{\ast}\times\mathcal{N}^{\ast}}s|D_{h}z|^{2}\right).
\end{equation*}
Applying a discrete integration by parts with respect to the average operator on the first integral from the right-hand side above gives
\begin{equation*}
        \int_{\mathcal{M}^{\ast}\times\mathcal{N}^{\ast}}r^{2}s|D_{h}q|^{2}\leq C_{\lambda}\left( \int_{\mathcal{M}^{\ast}\times\mathcal{N}^{\ast}}s^{3}|z|^{2}+\int_{\partial \mathcal{M}\times\mathcal{N}^{\ast}}s^{2}sh|z|^{2}+\int_{\mathcal{M}^{\ast}\times\mathcal{N}^{\ast}}s|D_{h}z|^{2}\right).
\end{equation*}
\par To estimate the integral with the term $|A_{h}D_{h}q|^{2}$, to get \eqref{eq:return:AxDx}, we use Lemma \ref{lem:product:rule} to write
\begin{equation}\label{eq:Ax:Dx}
    A_{h}D_{h}(q)=A_{h}^{2}\rho\,A_{h}D_{h}z+A_{h}D_{h}\rho\, z+\frac{h^{2}}{4}\left( 2A_{h}D_{h}\rho\,D_{h}^{2}z+D_{h}^{2}\rho\,A_{h}D_{h}z\right).
\end{equation}
Following similar steps as before, we multiply by $rs^{1/2}$ equation \eqref{eq:Ax:Dx}, then we take the $L^{2}(\mathcal{M}\times\mathcal{N}^{\ast})$-norm to the resulting expression, and thanks to Proposition \ref{pro:space:estimate} and Young inequality, we obtain
\begin{equation*}
\begin{split}
    \int_{\mathcal{M}\times\mathcal{N}^{\ast}}sr^{2}|A_{h}D_{h}q|^{2}\leq&C_{\lambda} \left( \int_{\mathcal{M}\times\mathcal{N}^{\ast}}s^{3}|z|^{2}+\int_{\mathcal{M}\times\mathcal{N}^{\ast}}s|A_{h}D_{h}z|^{2}\right)\\
    &+C_{\lambda}\left( +\int_{\mathcal{M}\times\mathcal{N}^{\ast}}s^{-1}(sh)^{4}|D_{h}z|^{2}+\int_{\mathcal{M}\times\mathcal{N}^{\ast}}s(sh)^{4}|A_{h}D_{h}z|^{2}\right).
    \end{split}
\end{equation*}
For \eqref{eq:return:DxDx}, note that arguing as before we get
\begin{equation*}
\begin{split}
    \int_{\mathcal{M}\times\mathcal{N}^{\ast}}r^{2}s^{-1}|D_{h}^{2}q|^{2}\leq &C_{\lambda}\left( \int_{\mathcal{M}\times\mathcal{N}^{\ast}}s^{3}|z|^{2}+\int_{\mathcal{M}\times\mathcal{N}^{\ast}}s^{-1}|D_{h}z|^{2}+\int_{\mathcal{M}\times\mathcal{N}^{\ast}}s|A_{h}D_{h}z|^{2}\right)\\
    &+C_{\lambda}\frac{h^{4}}{\delta^{4}T^{8}}\left( \int_{\mathcal{M}\times\mathcal{N}^{\ast}}s^{-1}|D_{h}z|^{2}\right).
    \end{split}
\end{equation*}
To obtain \eqref{eq:return:Dt} we note that $-D_{t}q=\mathcal{P}q+D_{h}^{2}\taut^{-}q$, then using that $s^{-1}\leq 1$ for $\tau\geq 1$ large enough we get
\begin{equation*}
    \int_{\mathcal{M}\times\mathcal{N}}\taut^{-}(r^{2}s^{-1})|D_{t}q|^{2}\leq 2\int_{\mathcal{M}\times\mathcal{N}}\taut^{-} r^{2}|\mathcal{P}q|^{2}+2\int_{\mathcal{M}\times\mathcal{N}}\taut^{-}r^{2}s^{-1}|D_{h}^{2}\taut^{-}q|^{2}.
\end{equation*}
\par For the spatial boundary terms we proceed analogously, for $t\in\mathcal{N}^{\ast}$ we have 
\begin{equation*}
    D_{h}q(h/2,t)=D_{h}(\rho)A_{h}z(h/2,t)+A_{h}(\rho)D_{h}(z)(h/2,t),
\end{equation*}
due to Lemma \ref{lem:product:rule}. Hence, Young's inequality and Proposition \ref{pro:space:estimate} yield
\begin{equation*}
    r^{2}s|D_{h}q(h/2,t)|^{2}\leq C_{\lambda}s^{3}|A_{h}z|^{2}(h/2,t)+C_{\lambda}s|D_{h}z|^{2}(h/2,t).
\end{equation*}
Since we can obtain a similar estimate for the points $(1-h/2,t)$, with $t\in\mathcal{N}^{\ast}$, it follows that 
\begin{equation*}
    \int_{\partial\mathcal{M}\times\mathcal{N}^{\ast}}sr^{2}t_{r}(|D_{h}q|^{2})\leq C_{\lambda}\left(\int_{\partial\mathcal{M}\times\mathcal{N}^{\ast}}s^{3}t_{r}(|A_{h}z|^{2})+\int_{\partial\mathcal{M}\times\mathcal{N}^{\ast}}st_{r}(|D_{h}z|^{2})\right).
\end{equation*}
\par Now, we use Lemma \ref{lem:product:rule} to write $\displaystyle A_{h}(|q|^{2})=|A_{h}q|^{2}+\frac{h^{2}}{4}|D_{h}q|^{2}$. Then, since $q:=\rho z$ repeated application of Lemma \ref{lem:product:rule}, Young's inequality and Proposition \ref{pro:space:estimate} lead to
\begin{equation*}
    r^{2}s^{3}A_{h}(|q|^{2})\leq  C_{\lambda}s^{3}|A_{h}z|^{2}+C_{\lambda}s(sh)^{3}|D_{h}z|^{2}.
\end{equation*}
Thus, for $A_{h}(|q|^{2})$ on the spatial boundary we have
\begin{equation*}
    \int_{\partial \mathcal{M}\times\mathcal{N}^{\ast}}s^{3}r^{2}t_{r}(A_{h}|q|^{2})\leq C_{\lambda}\int_{\partial\mathcal{M}\times\mathcal{N}^{\ast}}s^{3}t_{r}(|A_{h}z|^{2})+C_{\lambda}\int_{\partial\mathcal{M}\times\mathcal{N}^{\ast}}st_{r}(|D_{h}z|^{2}),
\end{equation*}
this proves \eqref{eq:return:Ax}.\\
Finally, to prove \eqref{eq:return:boundary}, we apply the discrete time operator $D_{t}$ on $q:=\rho\,z$. In this case, using Lemma \ref{lem:time:estimate} it follows that $\taut^{-}r\,D_{t}q=\taut^{-}rD_{t}\rho\taut^{+}z+D_{t}z$. Then, Cauchy-Swartz inequality gives
\begin{equation*}
\begin{split}
    \int_{\partial\mathcal{M}\times\mathcal{N}}|\taut^{-}rD_{t}q|^{2}\leq& C_{\lambda} \tau^{2}\int_{\partial\mathcal{M}\times\mathcal{N}}|\theta'\varphi|^{2}|\taut^{+}z|^{2}+C_{\lambda}\Delta t\int_{\partial\mathcal{M}\times\mathcal{N}}\left( \frac{\tau}{\delta^{3}T^{4}}+\frac{\tau^{2}}{\delta^{4}T^{6}}\right)|\taut^{+}z|^{2}\\
    &+\int_{\partial\mathcal{M}\times\mathcal{N}}|D_{t}z|^{2}.
    \end{split}
    \end{equation*}


\end{document}